\documentclass[11pt,letter]{article}

\usepackage[normalem]{ulem}

\usepackage[pagewise]{lineno}

\vfuzz \hfuzz
\usepackage[letterpaper, top=1in, bottom=1in, left=1in, right=1in, includefoot]{geometry}
\usepackage{macros_killavortex}
\usepackage{mathtools}

\usepackage{enumerate}
\title{Killing a Vortex\thanks{An earlier version of this paper has appeared at \href{https://focs2022.eecs.berkeley.edu/}{FOCS 2022} \cite{ThilikosW22Killing}.}}

\author{\bigskip\large Dimitrios M. Thilikos\thanks{LIRMM, Université de Montpellier, CNRS, Montpellier, France.\ Email: 
\texttt{sedthilk@thilikos.info}.}~$^{,}$\thanks{Supported by the ANR projects DEMOGRAPH (ANR-16-CE40-0028), ESIGMA (ANR-17-CE23-0010), and the French-German Collaboration ANR/DFG Project UTMA (ANR-20-CE92-0027).}\and\large 
\and\large 
Sebastian Wiederrecht\thanks{LIRMM, Université de Montpellier, CNRS, Montpellier, France and Discrete Mathematics Group, Institute for Basic Science, Daejeon, South Korea.\ Email: \texttt{sebastian.wiederrecht@gmail.com}}~$^{,}$\thanks{Supported by the ANR project ESIGMA (ANR-17-CE23-0010).}~$^{,}$\thanks{Supported by the Institute for Basic Science (IBS-R029-C1).}}
\date{}

\begin{document}

\maketitle

\vspace{-10mm}

\begin{abstract}
\noindent The Graph Minors Structure Theorem of Robertson and Seymour asserts that, for every graph $H,$ every $H$-minor-free graph can be obtained by clique-sums of ``almost embeddable'' graphs.
Here a graph is ``almost embeddable'' if it can be obtained from a graph of bounded Euler-genus by pasting graphs of bounded pathwidth in an ``orderly fashion'' into a bounded number of faces, called the \textit{vortices}, and then adding a bounded number of additional vertices, called \textit{apices}, with arbitrary neighborhoods.
Our main result is a {full classification}  of all graphs $H$ for which the use of vortices in the theorem above can be avoided.
To this end we identify a (parametric) graph $\mathscr{S}_{t}$ and prove that all $\mathscr{S}_{t}$-minor-free graphs can be obtained by clique-sums of graphs embeddable in a surface of bounded Euler-genus after deleting a bounded number of vertices.
We show that this result is tight in the sense that the appearance of vortices cannot be avoided for $H$-minor-free graphs, whenever $H$ is not a minor of $\mathscr{S}_{t}$ for some $t\in\mathbb{N}.$

Using our new structure theorem, we design an algorithm that, given an $\mathscr{S}_{t}$-minor-free graph $G,$ computes  the generating function of all perfect matchings of $G$ in polynomial time. 
Our results, combined with known complexity results, imply a complete characterization of minor-closed graph classes where the number of perfect matchings is polynomially computable: They are exactly those graph classes that do not contain every $\mathscr{S}_{t}$ as a minor.
This provides a \textit{sharp} complexity dichotomy for the problem of counting perfect matchings in minor-closed classes.
\end{abstract}

\medskip


\newpage

\tableofcontents

\newpage
\section{Introduction}\label{@photographic}
We consider the problem $\mbox{\textsc{\#Perfect Matching}},$ asking for the number of perfect matchings, 
on minor-closed graph classes. 
The first polynomial algorithm for this problem was given for the class of planar graphs  
by Kasteleyn in 1961 \cite{Kasteleyn61thest}, motivated by the dimer problem in Theore\-tical Physics \cite{Kasteleyn67graph,Kasteleyn61thest,TemperleyF61dimer,Kasteleyn63dimer}  (see also the results by Temperley and Fisher \cite{TemperleyF61dimer}). For this algorithm, Kasteleyn introduced the celebrated FKT-method and the concept of Pfaffian orientations. 

Using these concepts as a departure point, the tractability horizon has been extended to several minor-closed graph classes, further than planar graphs, and it was an open problem 
whether this horizon contained all minor-closed graph classes. A negative answer to this question  
was given by Curticapean and Xia in~\cite{curticapean2022parameterizing} who proved that the classic result of Valiant on the \#$\mathsf{P}$-completeness of $\mbox{\textsc{\#Perfect Matching}}$ holds even when restricted to $K_{8}$-minor free graphs.

In this paper, we completely resolve the complexity of $\mbox{\textsc{\#Perfect Matching}}$
on minor-closed graph classes by providing a sharp characterization of the classes for which the problem is tractable.

\subsection{Some history} Given a $n\!\times\! n$ matrix $A\! =\! (a_{i,j})$ the \emph{permanent} and the \emph{determinant} of $A$ are defined as
\begin{eqnarray*}
\mathsf{perm}(A)=\sum_{ \sigma\in S_{n}}\prod_{i=1}^{n}a_{i, \sigma(i)} & \mbox{and} & 
\mathsf{det}(A)=\sum_{ \sigma\in S_{n}}\mathsf{sgn}( \sigma)\cdot \prod_{i=1}^{n}a_{i, \sigma(i)}
\end{eqnarray*}
respectively, where $S_{n}$ is the set of all possible permutations of the set $[n]=\{1,\ldots,n\}$ and 
$\mathsf{sgn}( \sigma)$ is the sign of the permutation $ \sigma\in S_{n}.$
The permanent is closely related to the $\mbox{\textsc{\#Perfect Matching}}$ problem
as the number of perfect matchings of a bipartite graph $B$ is equal to 
$\mathsf{perm}(A'_{B})$ where $A'_{B}$ is the biadjacency matrix\footnote{The \emph{biadjacency matrix}  of a bipartite graph $B=(X\dot\cup Y,E)$ where $X=\{x_{1},\ldots,x_{n}\},$ $Y=\{y_{1},\ldots,y_{n}\},$ and  $E\subseteq X\times Y,$ is the  binary $n\times n$ matrix $A = (a_{i,j})$ where $a_{i,j}=1$ if and only if $\{x_{i},y_{i}\}\in E(B).$} of $B.$

In 1913 György Pólya \cite{Polya13aufga} asked when it is possible to change the sings 
of the entries of a binary $n\times n$ matrix $A = (a_{i,j})$ and obtain a new matrix $A'$ where 
$\mathsf{perm}(A)=\mathsf{det}(A').$ Notice that, in such cases, the computation of the permanent
is reduced to the computation of the determinant that, in turn, can be computed in polynomial time (see \cite{Agrawal06deter} for an exposition on the permanent versus determinant problem).

Kasteleyn in 1961, in an attempt to solve the dimer problem (originated from Statistical Physics~\cite{Kasteleyn67graph,Kasteleyn61thest,TemperleyF61dimer})
introduced the concept of a Pfaffian graph: a matchable$^2$ graph $G$ is \emph{Pfaffian} if 
it admits an orientation $\overrightarrow{G}$ such that every conformal\footnote{A graph is \emph{matchable}
if it contains at least one perfect matching. A cycle $C$ of a graph $G$ is \emph{conformal} if it is even and $G-C$ is matchable.} cycle $C$ of $G$ has an odd number of directed edges in agreement to the orientation $\overrightarrow{G},$ when traversed  clockwise. Such an orientation of $G$ is called \emph{Pfaffian}.

A Pfaffian orientation of a graph $G$ implies a scheme to change the signs of the adjacency matrix $A(G)$ such that the determinant of the resulting matrix still yields the number of perfect matchings of $G.$ As a special case of Kasteleyn's approach, Pólya's question can be answered affirmatively for a matrix $A$ if and only if the bipartite graph $B$ that has $A$ as its biadjacency  matrix is Pfaffian (see \cite{McCuaig04polya}). { In particular, the requested change of signs follows directly from a Pfaffian orientation $\overrightarrow{B}$ of $B.$}   Moreover, Kasteleyn proved that planar (matchable) graphs are Pfaffian
and gave a polynomial algorithm for computing a Pfaffian orientation of a planar graph.
This algorithm is widely known as the FKT-method (making also reference to the authors of \cite{TemperleyF61dimer}) and implies that $\mbox{\textsc{\#Perfect Matching}}$ is polynomially solvable in planar graphs. 

In 1972, Little treated Pólya's problem by giving a \textsl{complete} combinatorial characterization of Pfaffian bipartite graphs.
Later, McCuaig, Robertson, Seymour, and Thomas gave a structural characterization
of Little's condition that permitted the design of a polynomial algorithm checking whether a bipartite graph is Pfaffian 
\cite{RobertsonST99perma,McCuaigRST97perma}. This immediately implied a polynomial algorithm for $\mbox{\textsc{\#Perfect Matching}}$ on Pfaffian bipartite graphs. On the negative side (motivated by the permanent vs determinant problem) 
Valiant introduced the counting complexity class $\#P$ and proved that $\mbox{\textsc{\#Perfect Matching}}$ is 
$\#P$-complete~\cite{Valiant79theco}.

See \autoref{@desecrated} for a timeline of the known results on the $\mbox{\textsc{\#Perfect Matching}}$ problem.

\subsection{Counting matchings in minor-closed graph classes}
\label{@preponderant}

In what concerns general (i.e. non-bipartite) graphs, Kasteleyn claimed in~\cite{Kasteleyn61thest,Kasteleyn67graph} that his algorithm for 
$\mbox{\textsc{\#Perfect Matching}}$ on planar graphs can be extended to graphs of bounded Euler-genus.
This was proved by Galluccio and Loebl \cite{GalluccioL99onthe} for orientable surfaces and by Tesler \cite{Tesler00match} for non-orientable surfaces.
The later results were based on a reduction of the problem to the computation of $2^{g}$ orientations\footnote{The original result suggested by Kasteleyn used $4^g$ many orientations. This was since Kasteleyn worked with orientable genus, hence the number $h$ of handles, while the result we present here refers to the Euler-genus which is $2h+c,$ where $c$ is the number of crosscaps involved.}, where $g$ is the Euler-genus of the input graph.
A unified algorithm for counting the number of perfect matchings on graphs of bounded Euler-genus without the use of orientations was given by Curticapean and Xia in \cite{CurticapeanX15param}.

Notice that graphs of bounded Euler-genus are minor-closed. The emerging question is 
whether $\mbox{\textsc{\#Perfect Matching}}$ is polynomially solvable for all minor-closed graph classes.
For this, given a finite set of graphs $\mathcal{F},$ we introduce the notation $\mbox{\textsc{\#Perfect Matching}}(\mathsf{Excl}(\mathcal{F}))$ problem restricted to graphs excluding all graphs in $\mathcal{F}$ as minors.
We stress that every minor-closed graph class can be characterized by the minor-exclusion of some {\em finite} $\mathcal{F},$ because of Robertson and Seymour's theorem~\cite{RobertsonS04wagner}.
Observe also that $\mbox{\textsc{\#Perfect Matching}}$ is the same as $\mbox{\textsc{\#Perfect Matching}}(\emptyset),$ while $\mbox{\textsc{\#Perfect Matching}}$ on planar graphs is $\mbox{\textsc{\#Perfect Matching}}(\{K_{5},K_{3,3}\}).$
Using this notation, 
advances on the $\mbox{\textsc{\#Perfect Matching}}$ problem can be described as follows:
Valiant proved in \cite{Vazirani89ncalg} that $\mbox{\textsc{\#Perfect Matching}}(\{K_{3,3}\})$ is polynomially solvable,
Straub, Thierauf, and Wagner proved in \cite{StraubTW14count} that $\mbox{\textsc{\#Perfect Matching}}(\{K_{5}\})$ is polynomially solvable and later Curticapean in~\cite{Curticapean14count} and 
Eppstein and Vazirani in \cite{EppsteinV19ncalg} proved that $\mbox{\textsc{\#Perfect Matching}}(\mathcal{F})$ is polynomially solvable for every $\mathcal{F}$ containing a minor of a graph that admits a drawing in the plane with  at most one crossing. 

\begin{figure}[t] 
\hspace{-4.3cm}\scalebox{.96}{\input{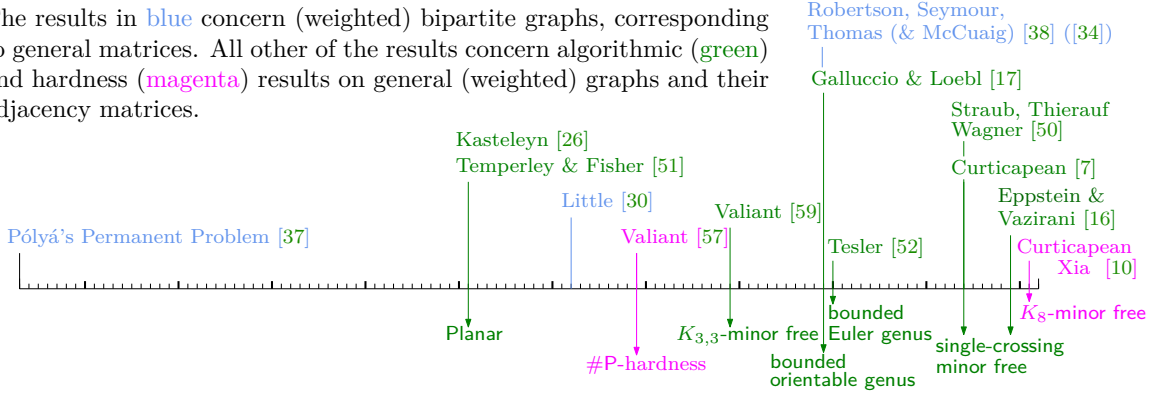}}
\caption{Timeline of the results on the complexity of the $\mbox{\textsc{\#Perfect Matching}}$ problem. 
} 
\label{@desecrated}
\end{figure}

Recently, Curticapean and Xia strengthened Valiant's
original hardness result by proving that $\mbox{\textsc{\#Perfect Matching}}(\{K_{8}\})$ is 
$\#P$-complete~\cite{CurticapeanX21param}. This means that the tractability horizon of $\mbox{\textsc{\#Perfect Matching}}$
does not include all minor-closed graph classes and it lies somewhere ``above'' single-crossing minor free graphs and ``below'' $K_{8}$-minor free graphs. In this paper we completely determine this tractability horizon.

\subsection{Our main result}

Let $G$ be a graph and let ${\mathbf{w}}:E(G)\rightarrow \Bbb{Z}$ be a function assigning weights to the edges of $G$ such that $\max\{|{\mathbf{w}}(e)|\mid e\in E(G)\}$ is bounded by some polynomial function of $|G|$ (that is the number of vertices of $G$).
We refer to such a pair $(G,{\mathbf{w}})$ as an \emph{edge-weighted graph} (or simply {\em weighted graph}). We use $\Perf{G}$ for the set of all perfect matchings of $G.$
We define the \emph{generating function of perfect matchings} of the weighted graph $(G,{\mathbf{w}})$
as

$$\GenerateMatchings{G,{\mathbf{w}}}\coloneqq \sum_{M\in\Perf{G}}\prod_{e\in M}x^{{\mathbf{w}}(e)}.$$

Notice that if ${\mathbf{1}}$ is the weighing function assigning unit weights on the edges of $G,$ then 
$\mathsf{perm}(A_{G})=\GenerateMatchings{G,{\mathbf{1}}}=c\cdot x^{|G|/2}$ where $c=|\Perf{G}|$ is the number of perfect matchings of $G.$ 
Therefore, any algorithm computing $\mathsf{PerfMatch}$ in polynomial time can also serve as a polynomial algorithm 
for the $\mbox{\textsc{\#Perfect Matching}}$ problem. Moreover, the computation of $\GenerateMatchings{G,{\mathbf{w}}}$ also permits to  solve the \textsc{Exact Perfect Matching} problem: given an edge-weighted graph and some non-negative integer  $k,$ decide whether there is a perfect matching of total weight exactly $k.$  \textsc{Exact Perfect Matching} 
was defined by Papadimitriou and Yannakakis in \cite{PapadimitriouY82theco}, has 
been extensively studied \cite{GurjarKMT17exact,ZhuLM08exact}, with  applications on DNA sequencing \cite{BlazewiczFKSW07apoly} and storage load balancing in blockchain networks \cite{liu2021efficient}.

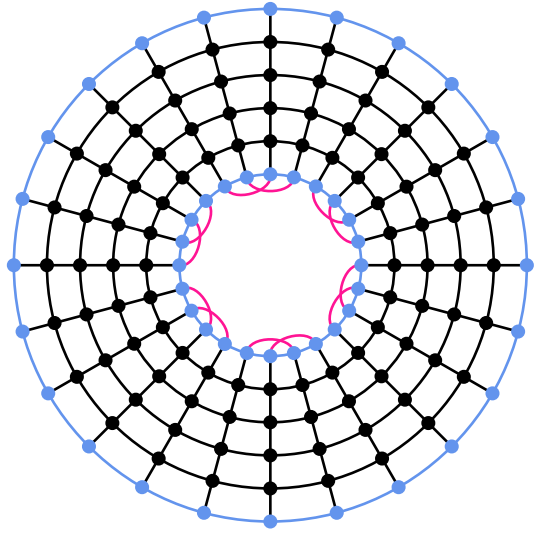
\begin{figure}[ht]
\centering 
{\begin{tikzpicture}[scale=1.1] 
\pgfdeclarelayer{background} 
\pgfdeclarelayer{foreground} 
\pgfsetlayers{background,main,foreground} 
 
\foreach \r in {2,...,5} 
\foreach \x in {1,...,24} 
{  
\pgfmathsetmacro\Radius{\r*4+7} 
\node[v:main] () at (\x*15: \Radius mm){}; 
}

\foreach \r in {1,6} 
\foreach \x in {1,...,24} 
{  
\pgfmathsetmacro\Radius{\r*4+7} 
\node[v:main, color=CornflowerBlue] () at (\x*15: \Radius mm){}; 
}

\begin{pgfonlayer}{background} 
\foreach \r in {2,...,5} 
{  
\pgfmathsetmacro\Radius{\r*4+7} 
\draw[e:main] (0,0) circle ({\Radius mm}); 
} 
\foreach \r in {1,6} 
{  
\pgfmathsetmacro\Radius{\r*4+7} 
\draw[e:main, color=CornflowerBlue] (0,0) circle ({\Radius mm}); 
}

\foreach \x in {1,...,24} 
{ 
\draw[e:main] (\x*15:11mm) to (\x*15:31mm); 
} 
 
\foreach \x in {1,...,6} 
{  
\pgfmathsetmacro\Position{4*15*(\x-1)} 
\draw[e:main,bend left=85, color=DeepPink] (\Position+15:11mm) to (\Position+45:11mm); 
\draw[e:main,bend left=85, color=DeepPink] (\Position+30:11mm) to (\Position+60:11mm); 
} 

\end{pgfonlayer} 
 
\end{tikzpicture} 
} 
\vspace{3mm} 
\caption{The shallow vortex grid $H_{6}$. The additional six pairs of crossed edges are depicted in \textcolor{DeepPink}{red}.
The two \textcolor{CornflowerBlue}{blue} cycles are the two extremal cycles of the $(6,24)$-cylindrical grid.}
\label{@industrial}
\end{figure}

\begin{definition}[Shallow vortex grids]\label{@duplicating}
    The \emph{shallow vortex grid} of \emph{order $k$} is the graph $\mathscr{S}_{k}$ obtained from the Cartesian product of a cycle $C=(u_1,u_2,\dots,u_{4k})$ on $4k$ vertices with a path $P= v_1,\dots,v_k$ on $k$ vertices (that we call \emph{$(k,4k)$-cylindrical grid}) by adding the edges $\{ (u_{4(i-1)+1},v_1),(u_{4(i-1)+3},v_1)\}$ and $\{ (u_{4(i-1)+2},v_1),(u_{4(i-1)+4},v_1)\}$ for every $i\in\{1,\dots,k\}.$
\end{definition}
 
In general, we see the \emph{shallow vortex grid} as  a \textsl{parametric graph}, that is the sequence $\mathscr{S}=\langle\mathscr{S}_{k}\rangle_{k\in\mathbb{N}}.$
To get a fairly good idea of $\mathscr{S}_{k},$ see \autoref{@industrial} for a drawing of $\mathscr{S}_{6}.$
We also define ${\mathcal{S}}$ as the graph class consisting of all minors of shallow vortex grids, i.e., all minors of the graphs in $\mathscr{S}.$ (For more on parametric graphs, see \cite{paul2023universal}.)

Our main result is a complexity dichotomy for $\mbox{\textsc{\#Perfect Matching}}(\mathcal{F}),$ based on the class ${\mathcal{S}}.$

\begin{theorem} 
\label{@controllers} 
Let $\mathcal{F}$ be some finite set of graphs. Then $\mbox{\textsc{\#\emph{Perfect Matching}}}(\mathcal{F})$ can be solved in polynomial time if $\mathcal{F}\cap {\mathcal{S}}\neq\emptyset$; otherwise it is $\#\mathsf{ P}$-complete.
\end{theorem}

As an example of an application of \autoref{@controllers}, $K_{7}$ is a minor of $\mathscr{S}_{18}$ (see~\autoref{@inevitably}), therefore $K_{7}\in \mathcal{S}.$ This implies that $\mbox{\textsc{\#Perfect Matching}}(\{K_{r}\})$ is polynomially solvable for $r\leq 7$ (which already answers the open question in~\cite{CurticapeanX21param}) and, as proved in~\cite{CurticapeanX21param}, is $\#P$-complete for $r\geq 8.$ For the general minor-exclusion of a finite set $\mathcal{F},$ containing graphs of size at most $h,$ one needs to check whether some graph in $\mathcal{F}$ is one of the graphs in $\mathcal{S}$ with at most $h$ vertices.\smallskip

In order to give some intuition why the minors of graphs as the one in \autoref{@industrial} provide the correct dichotomy criterion, we need a brief introduction to the Graph Minors Structure Theorem, proven by Robertson and Seymour in~\cite{RobertsonS03a} (in this paper 
we use the notation used in the simpler proof of this theorem that was proposed recently by Kawarabayashi, Thomas, and Wollan in~\cite{kawarabayashiTW21quickly} – see also \cite{KawarabayashiTW18anew}).

\subsection{The {\texttt{vga}}-hierarchy}

A \emph{graph parameter} is a function mapping graphs to non-negative integers. 
Let ${\mathsf{p}}$ and ${\mathsf{p}}'$ be two graph parameters. We write ${\mathsf{p}}\preceq {\mathsf{p}}'$
if there is a function $f:\Bbb{N}\to\Bbb{N}$ such that, for every graph
$G,$ it holds that ${\mathsf{p}}(G)\leq f({\mathsf{p}}'(G)).$ We also say that ${\mathsf{p}}$ and ${\mathsf{p}}'$
are \emph{asymptotically equivalent} if ${\mathsf{p}}\preceq {\mathsf{p}}'$ and ${\mathsf{p}}'\preceq {\mathsf{p}}.$

\begin{definition}[Treewidth]\label{@publications} 
Let $G$ be a graph. 
A \emph{tree decomposition} of $G$ is a tuple $(T,\beta )$ where $T$ is a tree and $\beta \colon\V{T}\rightarrow 2^{V(G)}$ is a function, called the \emph{bags} of $(T,\beta )$, such that 
\begin{enumerate}[label=\textit{\roman*})] 
\item $\bigcup_{t\in\V{T}}\Fkt{\beta }{t}=\V{G},$ 
\item for every $e\in\E{G}$ there exists $t\in\V{t}$ with $e\subseteq\Fkt{\beta }{t},$ and 
\item for every $v\in\V{G}$ the set $\CondSet{t\in\V{T}}{v\in\Fkt{\beta }{t}}$ induces a subtree of $T.$ 
\end{enumerate} 
The \emph{width} of a tree decomposition is defined as $\max_{t\in V(T)}\Abs{\Fkt{\beta }{t}}-1$ and the \emph{treewidth} of $G,$ denoted by $\tw{G},$ is the minimum width over all tree decompositions for $G.$ 
The \emph{adhesion} of $(T,\beta )$ is $\max_{dt\in E(T)}\Abs{\Fkt{\beta }{d}\cap\Fkt{\beta }{t}}.$
For a vertex $t\in\V{T},$ the \emph{torso of $G$ at $t$} is the graph $G_{t}$ obtained from $\InducedSubgraph{G}{\Fkt{\beta }{t}}$ by turning the sets $\Fkt{\beta }{t}\cap\Fkt{\beta }{d}$ into cliques, for all neighbors $d\in V(T)$ of $t.$
\end{definition}

\begin{definition}
\label{@historiography}
 
Given a graph $G,$ we define $\mathsf{p}_{\mbox{\scriptsize \texttt{v\!\!\! g\!\!\! a}}}(G)$ as the minimum $k$ such that $G$ has a tree decomposition ${\mathcal{D}}=(T,\beta )$ where, for every torso $G_{t}$ of ${\mathcal{D}},$ if $|G_{t}|>k,$ then the following holds: 
There is a set $A\subseteq V(G_{t}),$ called \emph{apex set}, a surface $\Sigma,$ and a $\Sigma$-decomposition $ \Delta$ of $G_{t}-A$ such that: 
\begin{itemize} 
\item[({\texttt{v}})] $ \Delta$ has at most $k_{{\texttt{v}}}$ vortices, each of depth at most $k_{{\texttt{v}}},$
\item[({\texttt{g}})] $\Sigma$ has Euler-genus at most $k_{{\texttt{g}}},$ and  
\item[({\texttt{a}})] $|A|\leq k_{{\texttt{a}}}$,
\end{itemize}
where $\max\{ k_{{\texttt{v}}},k_{{\texttt{g}}},k_{{\texttt{a}}}\}\leq k.$
We postpone the formal definition of a $\Sigma$-decomposition of a graph as well the definition of a vortex and its depth  to \autoref{@constraint} (see \autoref{@definition}).
Intuitively, $G_t$ has a $\Sigma$-decomposition with $k$ vortices if $G_t=G^{(0)}\cup G^{(1)}\cup\cdots\cup G^{(k)}$ where $G^{(0)}$ is a graph embedded in $\Sigma$ 
and each vortex $G^{(i)}$ is a graph of bounded pathwidth ``attached around'' some of the vertices of some face of the embedding of $G^{(0)}.$ 
\end{definition}  
 
\autoref{@historiography} has several variants ${\mathsf{p}}_{\mathsf{w}}.$
Where $\texttt{w}$ is string of length three obtained from $\texttt{vga}$ by replacing some of ``{\texttt{v}}'', ``{\texttt{g}}'', or ``{\texttt{a}}'' with ``{\texttt{-}}''.
Given such a string $\mathsf{w}$ the parameter ${\mathsf{p}}_{\mathsf{w}}$ is defined via altering \autoref{@historiography} by replacing for each $x\in\{ \texttt{v},\texttt{g},\texttt{a}\},$ where $x$ has been replaced with ``\texttt{-}'' in $\mathsf{w},$ the number $k_x$ from the corresponding condition in \autoref{@historiography} with $0.$
That is, if a letter appears in the string, then the corresponding object is allowed to appear in the decomposition.
Otherwise the appearance of the corresponding object is prohibited.

In addition, we derive one more parameter from \autoref{@historiography}.
That is the parameter $\mathsf{p}_{\mbox{\scriptsize \texttt{false}}}$ which we obtain by replacing whatever follows the statement ``the following holds:'' by some false statement.
We wish to stress here that ${\mathsf{p}}_{\mbox{\scriptsize \texttt{-\!\!\! -\!\!\! -}}}$ and $\mathsf{p}_{\mbox{\scriptsize \texttt{false}}}$ are different parameters.
To see this observe that ${\mathsf{p}}_{\mbox{\scriptsize \texttt{-\!\!\! -\!\!\! -}}}$ is zero on all planar graphs while $\mathsf{p}_{\mbox{\scriptsize \texttt{false}}}(G)=\mathsf{tw}(G)+1$ for all graph $G.$
This generates nine variants of parameters whose  
relation with respect to $\preceq$ is depicted in \autoref{@assessment}. We refer to this hierarchy of parameters as the {\text{\texttt{vga}}\emph{-hierarchy}}. We will later fix our attention to ${\mathsf{p}}_{\mbox{\scriptsize \texttt{-\!\!\! g\!\!\! a}}}$ where  
vortices disappear and therefore $G_{t}-A$ is just required to be embeddable to a surface of Euler-genus at most $k.$

On a lower level, the eight parameters above $\mathsf{p}_{\mbox{\scriptsize \texttt{false}}}$ can be understood as the different variants for the structure of the ``area'' which is controlled by a large order \textit{tangle} in the graph.
To clarify, tangles are dual objects to \textit{branchwidth}, a parameter which is assymptotically equivalent to treewidth which was introduced in Graph Minors X by Robertson and Seymour \cite{RobertsonS91X} to identify substructures in the graph which obstruct small treewidth.
The Grid Theorem asserts that every tangle of large order ``controls'' a large grid minor and the subsequent structure theorems, in essence, describe how the rest of the graph attaches to such a grid minor.
This form of attachment happens in two ways; via the highly representative infrastructure of a surface, or through small order separations which are represented by the decompositions of vortices, the apex set, and the $\leq3$-clique sums generated by applications of the \textit{Two Paths Theorem}\footnote{We discuss these notions in greater detail in \autoref{@constraint}}.
So, in a sense, the \texttt{vga}-hierarchy classifies graphs with respect to the structure of their high order tangles.
This is another way how the special role of $\mathsf{p}_{\mbox{\scriptsize \texttt{false}}}$ becomes apparent as this measures the situation where $G$ does not have any high-order tangles.
For more information on the lower levels of the \texttt{vga}-hierarchy we refer the reader to \cite{thilikos2023approximating,thilikos2023excluding}.

In~\autoref{@assessment} we depict, in relation to the {\texttt{vga}}-hierarchy, two more parameters that, when bounded, allow for $\mbox{\textsc{\#Perfect Matching}}$ to be solved in polynomial time.
The first is ${\mathsf{apex}}$ where ${\mathsf{apex}}(G)$ is the minimum number of vertices whose removal from $G$ yields a planar graph and the second is ${\mathsf{genus}}$ where ${\mathsf{genus}}(G)$ us the minimum Euler-genus of a surface where $G$ can be embedded.

\begin{figure}[h] 
\begin{center} 
\scalebox{.98}{\includegraphics{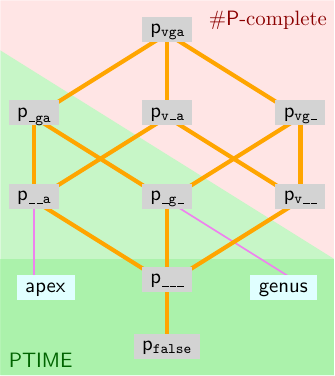}}
\end{center} 
\caption{The {\texttt{vga}}-hierarchy of parameters and the position of the parameters ${\mathsf{apex}}$ and ${\mathsf{genus}}$ in it. If ${\mathsf{p}}$ and ${\mathsf{p}}'$ are parameters in the above diagram then ${\mathsf{p}}\preceq {\mathsf{p}}'$ if and only if  there is a path between   ${\mathsf{p}}$ and ${\mathsf{p}}'$
that is ``above'' ${\mathsf{p}}'.$ 
The two \green{green}/{\textcolor{mypink}{pink}}-colored areas indicate the complexity of $\mbox{\textsc{\#Perfect Matching}}$ when  
restricted to graphs where each of the depicted parameters is bounded. 
The lower dark \green{green} area indicates the current state of the art on polynomial algorithms for $\mbox{\textsc{\#Perfect Matching}}.$} 
\label{@assessment}
\end{figure}

The \emph{Hadwiger number} of a graph $G$ is the maximum size of a clique minor of $G,$ denoted 
by ${\mathsf{hw}}(G).$
The Graph Minors Structure Theorem \cite{kawarabayashiTW21quickly,RobertsonS03a} (in short GMST) can be stated as follows.

\begin{proposition} 
\label{@naturalized}  
The graph parameters  
${\mathsf{hw}}$ and $\mathsf{p}_{\texttt{\emph{vga}}}$ are asymptotically equivalent.
\end{proposition}

 (For an alternative (non-parametric) statement of GMST, see \autoref{@guarantors}.) 
Notice that all parameters of the {\texttt{vga}}-hierarchy can be seen as generalizations of treewidth, starting
from the lowest level parameter ${\mathsf{p}}_{\mbox{\scriptsize\texttt{false}}},$ that is asymptotically equivalent to treewidth, to the highest level parameter ${\mathsf{p}}_{\mbox{\scriptsize \texttt{v\!\!\! g\!\!\! a}}},$ that is asymptotically equivalent to the Hadwiger number. 

In \cite{RobertsonS93exclu}, Robertson and Seymour proved that ${\mathsf{p}}_{\mbox{\scriptsize \texttt{-\!\!\! -\!\!\! -}}}(G)$ is asymptotically equivalent to the maximum size of an internally $4$-connected\footnote{A graph is \emph{internally $4$-connected} if it is $3$-connected and for every separation $(A,B)$ either $\Abs{A\setminus B}\leq 1$ or $\Abs{B\setminus A}\leq 1.$} single-crossing\footnote{A \emph{single-crossing graph} is one that can be drawn  in the plane with only one crossing.} minor of $G.$ Using this, the results 
in \cite{EppsteinV19ncalg} and in \cite{Curticapean14count} imply  an algorithm that, given a weighted graph $(G,{\mathbf{w}}),$
outputs ${\mathsf{GenPM}}(G,{\mathbf{w}})$ in time\footnote{We use the notation $h(k,n)=\mathcal{O}_{k}(g(n))$ to denote that $h(k,n)=\mathcal{O}(f(k)\cdot g(n)),$ for some function $f.$} $\mathcal{O}_{k}(|G|^{\mathcal{O}(1)}),$ where $k={\mathsf{p}}_{\mbox{\scriptsize \texttt{-\!\!\! -\!\!\! -}}}(G).$ 
This positions the results of \cite{EppsteinV19ncalg,Curticapean14count} to the second lower level of the {\texttt {vga}}-hierarchy (just above ${\mathsf{p}}_{\mbox{\scriptsize\texttt{false}}}$).
Apart from this result, the results in~\cite{GalluccioL99onthe,Tesler00match,galluccio2001optimization} imply that $\GenerateMatchings{G,{\mathbf{w}}}$ can be computed 
in time  $\mathcal{O}_{k}(|G|^{\mathcal{O}(1)})$  when $k=\genus{G}.$ Also, it is easy to see that $\GenerateMatchings{G,{\mathbf{w}}}$ 
can be computed 
in time $|G|^{\mathcal{O}(k)}$  when $k={\mathsf{apex}}(G)$ \cite{Curticapean19count,CurticapeanX15param}.
These three results are the best, so far, algorithmic results about when $\mbox{\textsc{\#Perfect Matching}}$ can be solved in polynomial time (corresponding to the dark green rectangle of the diagram in~\autoref{@assessment}). 

\subsection{Our approach}

We define a new parameter ${\mathsf{p}},$ based on  shallow vortex grids $\mathscr{S}_{k},$ as follows   $${\mathsf{p}}(G)=\max\{k\mid \mathscr{S}_{k}\text{ is a minor of } G\}.$$
Our main combinatorial result is a vortex-free refinement of the GMST (\autoref{@naturalized}).
In particular ${\mathsf{p}}$ can be seen as the vortex-free analogue of the Hadwiger number.

\begin{theorem} 
\label{@dialectical} 
The graph parameters $\mathsf{p}$ and $\mathsf{p}_{\texttt{\emph{-ga}}}$ are asymptotically equivalent.
\end{theorem}

\autoref{@dialectical} is a min-max theorem indicating that shallow vortex grids
can be seen as ``universal obstructions'' for the parameter ${\mathsf{p}}_{\mbox{\scriptsize \texttt{-\!\!\! g\!\!\! a}}}.$
Our proof implies that there is a function $f:\Bbb{N}\to\Bbb{N}$ and an algorithm that, given a graph $G$ and an integer $k,$ either finds the shallow vortex grid $\mathscr{S}_{k}$ as a minor of $G$
or outputs a tree decomposition of $G$ certifying that ${\mathsf{p}}_{\mbox{\scriptsize \texttt{-\!\!\! g\!\!\! a}}}(G)\leq f(k).$
Moreover, this algorithm runs in time $\mathcal{O}_{k}(|G|^{3})$  (see \autoref{@incinerated}) and, moreover, $f$ is a function that is exponential to some polynomial of $k$.

The proof of \autoref{@dialectical} is the main technical part of this paper. The fact that $\mathsf{ p}_{\mbox{\scriptsize \texttt{-\!\!\! g\!\!\! a}}}\preceq {\mathsf{p}}$ follows from the following result.

\begin{theorem}\label{@lovelessly} 
There exist functions $\alpha ,\gamma\colon\mathbb{N}\rightarrow\mathbb{N}$ such that every graph $G$ excluding some graph $H\in\mathcal{S}$ as a minor has a tree decomposition $(T,\beta )$ of adhesion at most $4\alpha (\Abs{V(H)})$ such that for every $t\in V(T),$ if $G_t$ is the torso of $G$ at $t,$ then there exists a set $A\subseteq V(G_t)$ with $\Abs{A}\leq \alpha (\Abs{V(H)})$ such that $G_t-A$ has Euler-genus at most $\gamma(\Abs{V(H)}).$ Moreover, $\gamma (x)=\mathsf{poly}(x),$ i.e., $\gamma $ is a polynomial function, while $\alpha(x)=2^{\mathsf{poly}(x)}.$
\end{theorem}

The proof of \autoref{@lovelessly} is presented in the first three  subsections of \autoref{@desexualized}. The proof of \autoref{@dialectical} is completed in \autoref{@inevitably}, where we show that 
$\mathsf{p}\preceq {\mathsf{p}}_{\mbox{\scriptsize \texttt{-\!\!\! g\!\!\! a}}}$ (\autoref{@prophecies}).

We prove \autoref{@lovelessly} by first using a local version of \autoref{@naturalized} and then applying the following key observation.
Indeed, one can get a stronger version of \autoref{@naturalized} which equips each vortex with a sequence of concentric cycles in the obtained $\Sigma$-decomposition.
These cycles provide quite a lot of infrastructure and are commonly referred to as a \emph{nest}.
Now consider the following sequence of observations.
\begin{enumerate}
\item If there are not many pairwise disjoint paths from the outer part of the nest to the boundary of the vortex itself, there exists a small set of vertices which may be added to the apex set in order to fully ``remove'' the vortex from the almost embedding by ``pushing'' the corresponding subgraph deeper into the tree-decomposition.
\item Hence, we may assume that there exist many pairwise disjoint paths which meet all cycles in the nest and end on the boundary of the vortex.
If now we cannot find many pairwise disjoint pairs of crossings within the vortex such that these crossings are arranged sequentially on the vortex-boundary we may use the fact that the vortex has a path-decomposition of bounded adhesion to delete a small number of vertices and remove all crossings on the boundary of the vortex which lie inside the vortex.
By doing so, however, the parts of the vortex that remain fully connected to the nest do not have any crossings and thus may be fully incorporated into the embdded part of the $\Sigma$-decomposition.
This procedure is another way of essentially removing the vortex.
\item On the other hand, if we find many pairwise disjoint paths meeting all cycles in the nest and ending on the vortex boundary such that there are many pairwise disjoint pairs of crossings on the endpoints of these paths which are arranged sequentially along the vortex boundary we have, essentially, found a minor model of a shallow vortex grid as depicted in \autoref{@industrial}.
\end{enumerate}
By applying these three observations to each of the vortices of the $\Sigma$-decomposition, we either find a way to enrich the apex set and thereby get rid of all vortices completely, or one of the vortices witnesses the third case above which provides us with the minor we excluded in the assumption.

Based on \autoref{@lovelessly}, we next prove our main algorithmic result.

\begin{theorem}\label{@unencumbered} 
There are an algorithm and a function $f:\Bbb{N}\to\Bbb{N}$ that, given a weighted graph $(G,{\mathbf{w}}),$ 
where the maximum absolute value of $\textbf{\emph{w}}$ is bounded by some polynomial in $|G|,$ outputs $\GenerateMatchings{G,{\mathbf{w}}}$ in time $\mathcal{O}(|G|^{f(k)}),$ where $k={\mathsf{p}}_{\mbox{\scriptsize \texttt{\emph{-\!\!\! g\!\!\! a}}}}(G).$ Our algorithm assumes that arithmetic operations are performed in constant time.
\end{theorem}

The algorithm of \autoref{@unencumbered} is presented in  \autoref{@hairdressers}.
It performs dynamic programming on the tree decomposition 
provided by \autoref{@lovelessly} and combines all the algebraic tools that have been employed so far around Pfaffian orientations for the $\mbox{\textsc{\#Perfect Matching}}$ problem~\cite{Kasteleyn61thest,Kasteleyn67graph,TemperleyF61dimer,GalluccioL99onthe,galluccio2001optimization,Tesler00match}.

Notice that ${\mathcal{F}}\cap {\mathcal{S}}\not=\emptyset$ if and only if 
there is a constant $c_{\mathcal{F}}$ such that for every ${\mathcal{F}}$-minor free graph $G$ it holds that ${\mathsf{p}}(G)\leq c_{\mathcal{F}}.$ This fact, combined with \autoref{@dialectical} and \autoref{@unencumbered} implies the positive 
part of \autoref{@controllers}. 

On the negative side, in \autoref{@stiidentenzeitung}, we prove the  following using as a departing point the complexity lower bound in \cite{curticapean2022parameterizing}.

\begin{theorem} 
\label{@winterhilfiwerlr} 
For every graph class ${\mathcal{G}}$ where ${\mathcal{S}}\subseteq {\mathcal{G}},$ $\mbox{\textsc{\#\emph{Perfect Matching}}}$ is \#{\textsf{\emph{P}}}-complete when its inputs are restricted to the graphs in ${\mathcal{G}}.$
\end{theorem}

Notice that ${\mathcal{F}}\cap {\mathcal{S}}=\emptyset$ if and only if all graphs in ${\mathcal{S}}$
are ${\mathcal{F}}$-minor free. This, combined with \autoref{@winterhilfiwerlr}, yields the negative part of \autoref{@controllers}.\medskip

One may ask whether the algorithm of \autoref{@unencumbered} can be improved to a {\textsl{fixed parameter}} one, that is an algorithm running in  time
$\mathcal{O}_{k}(|G|^{\mathcal{O}(1)})$  where $k={\mathsf{p}}_{\mbox{\scriptsize \texttt{-\!\!\! g\!\!\! a}}}(G)$ (we already mentioned that this is the case when $k={\mathsf{p}}_{\mbox{\scriptsize \texttt{-\!\!\! -\!\!\! -}}}(G)$ \cite{EppsteinV19ncalg,Curticapean14count} 
or when $k=\genus{G}$ \cite{GalluccioL99onthe,Tesler00match,galluccio2001optimization}). Unfortunately, this is not something to expect even for $k={\mathsf{p}}_{\mbox{\scriptsize \texttt{-\!\!\! -\!\!\! a}}}(G)$ (which is  
lower than ${\mathsf{p}}_{\mbox{\scriptsize \texttt{-\!\!\! g\!\!\! a}}}$ in the {\texttt{vga}}-hierarchy). Indeed, it was proved in \cite{CurticapeanX15param} that $\mbox{\textsc{\#Perfect Matching}}$
is \#{\textsf{W}}[1]-hard when parameterized by ${\mathsf{apex}}(G).$ As ${\mathsf{p}}_{\mbox{\scriptsize \texttt{-\!\!\! g\!\!\! a}}}\preceq {\mathsf{p}}_{\mbox{\scriptsize \texttt{-\!\!\! -\!\!\! a}}}\preceq {\mathsf{apex}},$ this hardness result 
holds also when the parameter is $k={\mathsf{p}}_{\mbox{\scriptsize \texttt{-\!\!\! -\!\!\! a}}}(G)$ or, even more, $k={\mathsf{p}}_{\mbox{\scriptsize \texttt{-\!\!\! g\!\!\! a}}}(G).$ This indicates that, under standard computational complexity assumptions, the algorithm of \autoref{@unencumbered} is optimal from the parameterized complexity point of view in the sense that the dependency of the degree of the polynomial on the excluded minor cannot be removed.
We stress that this dependency, while it cannot be removed, could however possibly be improved.

\section{Definitions and preliminary results}
\label{@exceptional}

We denote by $\Bbb{Z}$ the set of integers and by $\Bbb{R}$ the set of reals.
Given two integers $a,b\in\Bbb{Z}$ we denote the set $\CondSet{z\in\Bbb{Z}}{a\leq z\leq b}$ by $[a,b].$
In case $a>b$ the set $[a,b]$ is empty. For an integer $p≥ 1,$ we set $[p]=[p]$ and $\Bbb{N}_{≥ p}=\Bbb{N}\setminus [0,p-1].$
Whenever we need a closed interval over the reals we use $[x,y]_{\Bbb{R}}$ to avoid ambiguity.
Please note that this only happens on rare occasions.
\smallskip

All graphs considered in this paper are undirected, finite, and without loops or multiple edges.
We use standard graph-theoretic notation and we refer the reader to
\cite{Diestel10grap} for any undefined terminology.

\subsection{The Graph Minors Structure Theorem}\label{@constraint}

We rely heavily on the more refined versions of the GMST from \cite{kawarabayashiTW21quickly} and \cite{diestel2012excluded} instead of the original result proven by Robertson and Seymour in the Graph Minors series.
As the involved definitions and concepts are highly technical we dedicate this section almost entirely just to their statements and some explanations.
In the proof of our main theorem we actually need a stronger version of GMST which has, implicitly, already been proven in \cite{kawarabayashiTW21quickly} and which resembles the main theorem of \cite{diestel2012excluded}.
However, for the purpose of our proofs, in particular regarding algorithmic applications, it is more convenient to derive a synthesis of the results in \cite{diestel2012excluded} and those in \cite{kawarabayashiTW21quickly}.
The resulting statement makes the structure of the vortices within the original structure theorem more accessible and this might be of use to applications other than our own as well.

Our proof strategy is to ``exclude'' a shallow vortex grid as a minor.
This means that whenever we find a clique minor on the same number of vertices as our shallow vortex grid we have found a contradiction.
Hence, we may immediately apply the GMST to obtain a, so to speak, preliminary structure which we then can refine.
The following definitions come from the framework of   \cite{kawarabayashiTW21quickly,KawarabayashiTW18anew}.
Their main purpose is to provide a formal environment in which we can apply our refinement strategy.
This is necessary specifically because we will need to work both inside the vortices and outside of them, that is, in the parts of the graph that are properly embedded in some surface.

\paragraph{The Two Paths Theorem.}\label{@immaterial}
We begin by introducing the \textit{Two Paths Theorem}.
This theorem can be seen as the central bridging element that combines a notion of embeddability with a way disjoint paths can be routed within the graph.

Let $G$ be a graph and let $s_1,s_2,t_1,t_2\in\V{G}.$
The \textsc{Two Disjoint Paths Problem} (\textsc{TDPP}) with \emph{terminals} $s_1,s_2,t_1,t_2$ is the question for the existence of two paths $P_1$ and $P_2$ such that for both $i\in[2]$ $P_i$ joins $s_i$ and $t_i$ and $P_1$ and $P_2$ are vertex disjoint.
The characterization for the \yes-instances of the \textsc{TDPP} known as the \emph{Two Paths Theorem} plays an integral role in structural graph theory.
The statement of the Two Paths Theorem we present here makes use of the concept of the so called ``societies'' which play a focal role in \cite{kawarabayashiTW21quickly} and are used extensively in our proofs as well.

There exists a number of different forms for stating the Two Paths Theorem.
The version we present here is the one used by Kawarabayashi, Thomas, and Wollan in their proof of the GMST in \cite{kawarabayashiTW21quickly} and it is based on the concept of societies.
Roughly speaking, a society is a graph $G$ together with a vertex set $X\subseteq(G)$ which is equipped with a cyclic ordering $\Omega$.
This cyclic ordering is meant to indicate that we are interested in some particular topological properties of the set $X$ with regard to its connectivity within $G.$
Through the lens of the Two Paths Theorem, a society is meant to encode a specific subgraph which we would like to embed in a (closed) disk $ \Delta$ while drawing exactly the vertices of $X$ onto the boundary of $ \Delta.$
If this is possible, then $X$ forms a noose in the resulting embedding of the entire graph, otherwise the society represents an obstruction to embedability which, if such a thing occurs to often, would allow us to build a large clique minor.

\begin{definition}[Society]\label{@deflecting}
Let $\Omega$ be a cyclic ordering of the elements of some set which we denote by $\V{\Omega}.$
A \emph{society} is a pair $(G,\Omega),$ where $G$ is a graph and $\Omega$ is a cyclic ordering with $\V{\Omega}\subseteq\V{G}$.
A \emph{cross} in a society $(G,\Omega)$ is a pair $(P_1,P_2)$ of disjoint paths\footnote{When we say two paths are \emph{disjoint} we mean that their vertex sets are disjoint.} in $G$ such that $P_i$ has endpoints $s_i,t_i\in\V{\Omega}$ and is otherwise disjoint from $\V{\Omega},$ and the vertices $s_1,s_2,t_1,t_2$ occur in $\Omega$ in the order listed.
\end{definition}
Hence, $(G,s_1,s_2,t_1,t_2)$ is a \yes-instance of \textsc{TDPP} if and only if the society $(G,\Omega),$ where $\V{\Omega}=\Set{s_1,s_2,t_1,t_2}$ and the vertices occur in $\Omega$ in the order listed, has a cross.
See \autoref{@predominantly} for an illustration of a society with a cross.

\begin{figure}[h] 
    \begin{center} 
    \scalebox{0.75}{\includegraphics{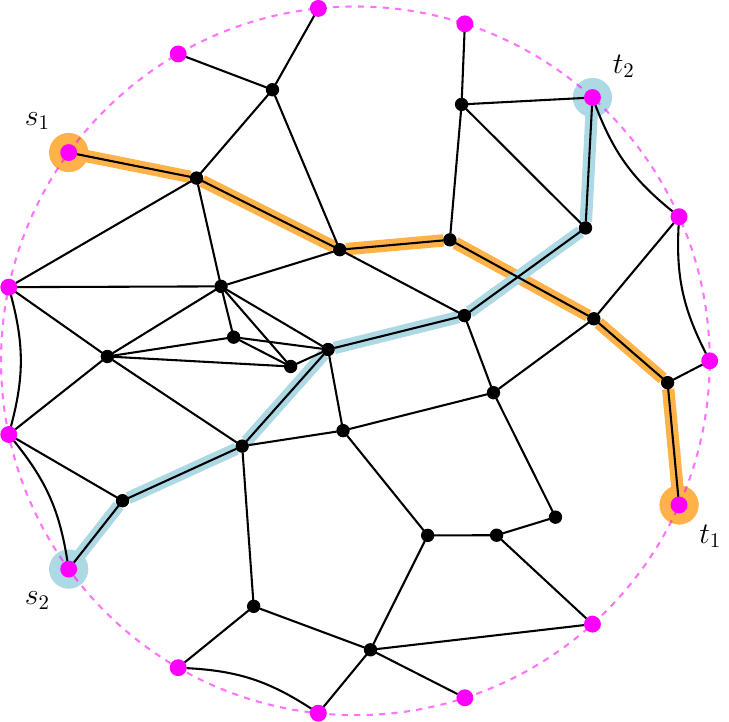}}
    \end{center}
    \caption{A society $(G,\Omega)$ with a cross. The \textcolor{BrightMagenta}{magenta} vertices are the vertices of $V(\Omega).$} 
    \label{@predominantly} 
\end{figure}

We will make use of this concept later when we start refining the vortices of our almost embeddings as follows:
The entire vortex will be encoded as a society $(G,\Omega)$ equipped with some additional infrastructure.
We will move along $\Omega$ trying to isolate pairwise disjoint crosses over ``segments'' of consecutive vertices of $\Omega.$
In case we can find many such segments, we can use these crosses together with the afore mentioned infrastructure to construct a shallow vortex grid.
Otherwise, only few such segments can be found which will allow us to remove all crosses on $(G,\Omega)$ by deleting a small vertex set.
The Two Paths Theorem will then allow us to embed the remaining parts of $(G,\Omega)$ in a way that is compatible with the already embedded part of our graph.

To fully present the Two Paths Theorem, we need to introduce some topological concepts as well.

By a \emph{surface} we mean a compact $2$-dimensional manifold with or without boundary.
By the classification theorem of surfaces \cite{Dyck1888Beitraege}, every surface is homeomorphic to the sphere with $h$ handles and $c$ cross-caps added, and the interior of $d$ disjoint closed disks $ \Delta_1,\dots, \Delta_d$ removed, in which case the \emph{Euler-genus} of the surface is defined to be $2h+c.$
We call the union of the boundaries of the disks $ \Delta_i$ the \emph{boundary} of the surface and each such boundary is a \emph{boundary component} of the surface.
For more information on surfaces and the Classification Theorem see for example \cite{Maunder1996Algebraic}.

\begin{definition}(Drawing on a surface)\label{@unarticulated}
A \emph{drawing} (with crossings) \emph{in a surface $\Sigma$} is a triple $\gamma=(U,V,E)$ such that
\begin{itemize} 
\item $V$ and $E$ are finite, 
\item $V\subseteq U\subseteq\Sigma,$ 
\item $V\cup\bigcup_{e\in E}e=U$ and $V\cap (\bigcup_{e\in E}e)=\emptyset,$ 
\item for every $e\in E,$ either $e=h([0,1]_{\mathbb{R}})\setminus\{h(0),h(1)\},$ where $h\colon[0,1]_{\Bbb{R}}\rightarrow U$ is a homeomorphism onto its image with $h(0),h(1)\in V,$ or $e=h(\mathbb{S}^2-(1,0)),$ where $h\colon\mathbb{S}^2\rightarrow U$ is a homeomorphism onto its image with $h((0,1))\in V,$ and 
\item if $e,e'\in E$ are distinct, then $\Abs{e\cap e'}$ is finite.
\end{itemize}
We call the set $V,$ sometimes referred to by $V(\gamma),$ the \emph{vertices of $\gamma$} and the set $E,$ referred to by $E(\gamma),$ the \emph{edges of $\gamma$}.
If $G$ is graph and $\gamma=(U,V,E)$ is a drawing with crossings in a surface $\Sigma$ such that $V$ and $E$ naturally correspond to $V(G)$ and $E(G)$ respectively, we say that $\gamma$ is a \emph{drawing of $G$ in $\Sigma$ (possibly with crossings)}. In the case where no two edges in $E(\gamma)$ have a common point, we 
say that $\gamma$ is a \emph{drawing of $G$ in $\Sigma$ without crossings}. In this last case, the connected components of $\Sigma\setminus U,$ are the \emph{faces} of $\gamma.$
\end{definition}

Next we need a more intricate formalization of what an ``almost embedding'' of a graph is supposed to be.
The main issue is that large portions of the graph that we are working with, might be hidden behind $3$-clique sums produced by repeated applications of the Two Paths Theorem.
To fully encode the almost embedding of a part of the graph $G,$ while keeping the entirety of $G$ accessible, Kawarabayashi, Thomas, and Wollan  developed $\Sigma$-decompositions in \cite{kawarabayashiTW21quickly,KawarabayashiTW18anew}.

\begin{definition}[$\Sigma$-Decomposition]\label{@custodians} 
Let $\Sigma$ be a surface. 
A \emph{$\Sigma$-decomposition} of a graph $G$ is a pair $ \Delta=(\gamma,\mathcal{D}),$ where $\gamma$ is a drawing of $G$ in $\Sigma$ with crossings, and $\mathcal{D}$ is a collection of closed disks, each a subset of $\Sigma$ such that 
\begin{enumerate}[label=\textit{\roman*})]
\item the disks in $\mathcal{D}$ have pairwise disjoint interiors, 
\item the boundary of each disk in $\mathcal{D}$ intersects $\gamma$ in vertices only, 
\item if $ \Delta_1, \Delta_2\in\mathcal{D}$ are distinct, then $ \Delta_1\cap \Delta_2\subseteq\V{\gamma},$ and
\item every edge of $\gamma$ belongs to the interior of one of the disks in $\mathcal{D}.$ 
\end{enumerate} 
Let $N$ be the set of all vertices of $\gamma$ that do not belong to the interior of any of the disks in $\mathcal{D}.$ 
We refer to the elements of $N$ as the \emph{nodes} of $ \Delta.$ 
If $ \Delta\in\mathcal{D},$ then we refer to the set $ \Delta-N$ as a \emph{cell} of $ \Delta.$ 
We denote the set of nodes of $ \Delta$ by $N( \Delta)$ and the set of cells by $C( \Delta).$ 
For a cell $c\in C( \Delta)$ the set of nodes that belong to the closure of $c$ is denoted by $\widetilde{c}.$ 
For a cell $c\in C( \Delta)$ we define $ \sigma_{ \Delta}(c),$ or $ \sigma(c)$ if $ \Delta$ is clear from the context, to be the subgraph of $G$ consisting of all vertices and edges drawn in the closure of $c.$ 
 
We define $\pi_{ \Delta}\colon N( \Delta)\rightarrow\V{G}$ to be the mapping that assigns to every node in $N( \Delta)$ the corresponding vertex of $G.$

For illustrations of $\Sigma$-decompositions consider \autoref{@completeness} and \autoref{fig_sigma_dec}.
In both figures the vertices in magenta and black are the nodes while the grey vertices are drawn in the interiors of cells which means they either sit behind $h$-clique sums for $h\leq 3$ or belong to a vortex.
The cells are depicted as blue shapes and every edge of $G$, including those between nodes, is drawn within a cell.
 
Isomorphisms between two $\Sigma$-decompositions are defined in the natural way. That is given $\Sigma$-decompositions $ \Delta=(\gamma,\mathcal{D})$ and $ \Delta'=(\gamma',\mathcal{D}'),$ the drawing $\gamma=(U,V,E)$ is mapped to a drawing $\gamma'=(U',V',E')$ where the elements of $V$ are in bijection with the elements of $V',$ similarly for $E$ and $E'$ such that the map between the corresponding graphs is a graph isomorphism, and the elements of $\mathcal{D}$ are mapped to the disks in $\mathcal{D}'$ while agreeing with the map between $\gamma$ and $\gamma'.$
\end{definition}

Notice that, in the definition of a $\Sigma$-decomposition $ \Delta,$ the cells $c$ of $ \Delta$ with $\widetilde{c}\neq\emptyset$ correspond to the hyperedges of a hypergraph with vertex set $N( \Delta)$ where $\widetilde{c}$ is the set of vertices incident with $c.$
Moreover, this hypergraph can be embedded in $\Sigma$ such that hyperedges only intersect in common vertices and all vertices are drawn on the boundaries of their hyperedges.
As examples of this consider \autoref{@completeness} and \autoref{fig_sigma_dec}.
In both figures, the \textcolor{CornflowerBlue}{blue} areas mark the cells which become the hyperedges of some hypergraph.
Please note that, in order to not overload the figures, the graphs in the interiors of the cells are chosen to be somewhat minimal with the property that, together with the remaining planar part of the graph, they form an obstruction to planarity.
In general no such restriction exists, the graphs within cells can also be planar or arbitrarily complex.

\begin{definition}[Vortex]\label{@definition}
Let $G$ be a graph, $\Sigma$ be a surface and $ \Delta=(\gamma,\mathcal{D})$ be a $\Sigma$-decomposition of $G.$
A cell $c\in C( \Delta)$ is called a \emph{vortex} if $\Abs{\widetilde{c}}\geq 4.$
Moreover, we call $ \Delta$ \emph{vortex-free} if no cell in $C( \Delta)$ is a vortex.
\end{definition}

See \autoref{fig_sigma_dec} for an illustration of a part of some $\Sigma$-decomposition which includes a vortex.

The reason why the threshold for the boundary size of a vortex is four  lies hidden in the \textit{Two Paths Theorem}  which we present below (\autoref{@postscripts}). Whenever the society defined by the boundary of a cell has at most three vertices, it is impossible to have a cross. This means, in particular, that, if a $\Sigma$-decomposition $ \Delta$ of a graph $G$ has no vortex, one could forget about the interiors of the cells of $ \Delta,$ for each cell of $ \Delta$ transform the vertices drawn on its boundary into a clique, and thereby obtain a graph on the vertex set $\pi(N( \Delta))$ which is drawn in $\Sigma$ without crossings.
This is, roughly, the intuition how $\Sigma$-decompositions encode a torso.

\begin{definition}[Rendition]\label{@contradicted}
Let $(G,\Omega)$ be a society, and let $\Sigma$ be a surface with one boundary component $B.$
A \emph{rendition} in $\Sigma$ of $G$ is a $\Sigma$-decomposition $\rho$ of $G$ such that the image under $\pi_{\rho}$ of $N(\rho)\cap B$ is $\V{\Omega}$ and $\Omega$ is one of the two cyclic orderings of $V(\Omega)$ defined by the way the points of $\pi_{\delta}(V(\Omega))$ are arranged in the boundary $Β$.
\end{definition}

\begin{figure}[h] 
    \begin{center} 
    \scalebox{0.8}{\includegraphics{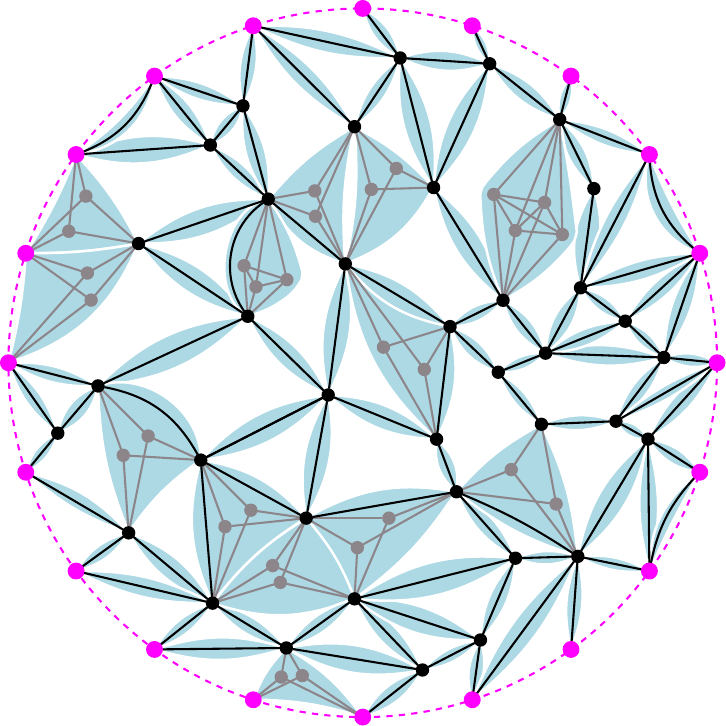}}
    \end{center} 
    \caption{A vortex-free rendition of a society $(G,\Omega)$ in the disk.} 
    \label{@completeness} 
\end{figure}

These technical definitions allow us the state the Two Paths Theorem in the general context of the Graph Minors Structure Theorem as follows.
For an illustration of a vortex-free rendition and the absence of a cross on a society $(G,\Omega),$ even when $G$ is non-planar, see \autoref{@completeness}.

\begin{proposition}[Two Paths Theorem, \cite{jung1970verallgemeinerung,seymour1980disjoint,shiloach1980polynomial,thomassen19802,robertson1990graph2}]\label{@postscripts}
A society $(G,\Omega)$ has no cross if and only if it has a vortex-free rendition in a disk.
\end{proposition}

\paragraph{The Flat Wall Theorem.}\label{@participant}

This section is dedicated to a weaker version of the structure theorem for $K_t$-minor-free graphs known as the \emph{Flat Wall Theorem} \cite{robertson1995graph,KawarabayashiTW18anew}.
While we do not need the Flat Wall Theorem directly in this paper, it is useful to keep in mind as it is the primary source for the infrastructure we need to construct our shallow vortex grid after processing the vortices of a $\Sigma$-decomposition.
Indeed, it acts as the base of the construction of a rendition with a bounded number of bounded depth vortices for any $K_t$-minor-free graph of large treewidth and, as such, it is needed for the statements of the slightly altered versions of lemmas and theorems that we extract from \cite{kawarabayashiTW21quickly}.

A \emph{separation} in a graph $G$ is a pair $(A,B)$ of vertex sets such that $A\cup B=\V{G}$ and there is no edge in $G$ with one endpoint in $A\setminus B$ and the other in $B\setminus A.$
The \emph{order} of $(A,B)$ is $\Abs{A\cap B}.$

An \emph{$(n\times m)$-grid} is the graph $G_{n,m}$ which is the product of a path $P= u_1,u_2,\dots,u_n$ with $n$ vertices and a path $Q=v_1,v_2,\dots,v_m$ with $m$ vertices.
We call the copies of $Q$ in $G_{n,m}$ the \emph{rows} and the copies of $P$ the \emph{columns}.
If $L$ is a row of the form $(u_i,v_1),(u_i,v_2),\dots,(u_n,v_m)$ we call it the \emph{$i$th row} and for $j\in[m]$ we say that $(u_i,v_j)$ is the \emph{$j$th vertex of the $i$th row} while the edge $\{(u_i,v_j)(u_i,v_{j+1})\}$ is the \emph{$j$th edge of the $i$th row}.
For columns we define analogous terminology.
An \emph{elementary $k$-wall}, $k\geq 3,$ is obtained from a $(k\times 2k)$-grid by deleting every odd edge in every odd column and every even edge in every even column.
An elementary $k$-wall $W$ has a unique face whose boundary contains more than six vertices.
The \emph{perimeter} of an elementary $k$-wall is defined to be the subgraph of $W$ induced by all vertices that lie on the unique face with more than six vertices.
A \emph{$k$-wall} $W'$ is a subdivision of an elementary $k$-wall $W.$ In other words, 
$W'$ is obtained by the $k$-wall $W$ after subdividing each edge of $W$ an arbitrary (possibly zero) number of times.
The \emph{perimeter} of $W',$ denoted by $\Perimeter{W'},$ is the subgraph of $W'$ induced by the vertices of the perimeter of $W$ together with the subdivision vertices of the edges of the perimeter of $W.$

\begin{figure}[h] 
\begin{center} 
\scalebox{0.13}{\includegraphics{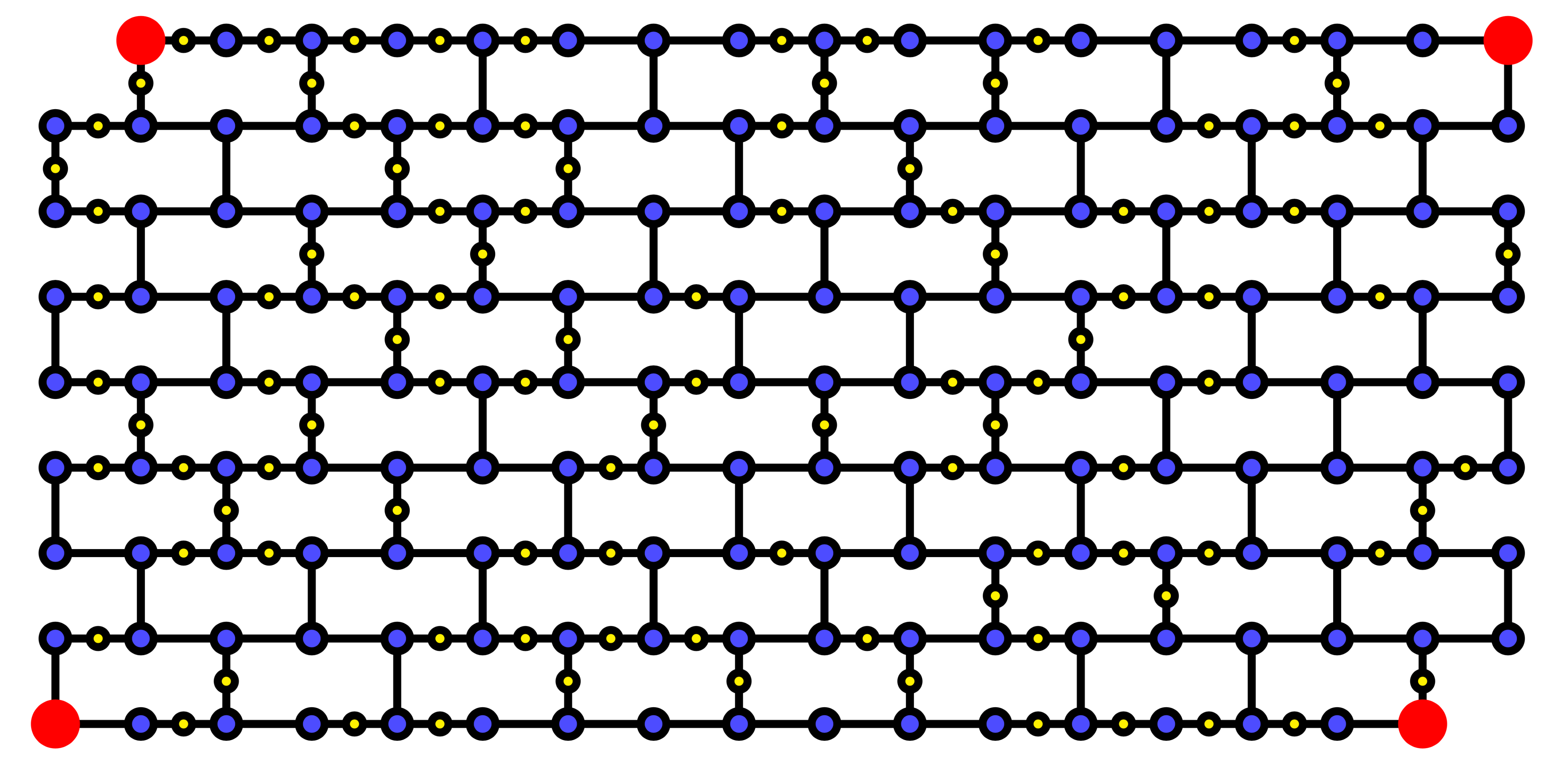}}
\end{center} 
\caption{A $9$-wall. The subdivision vertices are in yellow and the four corners are in red.} 
\label{@terminologie} 
\end{figure}

Let $G$ be a graph and $W⊆ G$ be a wall (here $\subseteq$ denotes the subgraph relation).
The \emph{compass} of $W$ in $G,$ denoted by $\CompassG{G}{W},$ is the subgraph of $G$ induced by the vertices of $\Perimeter{W}$ together with the vertices of the unique component of $G-\Perimeter{W}$ that contains $W-\Perimeter{W}.$
The \emph{corners} of an elementary $r$-wall $W$ are first and the last vertices of the first and the last row respectively.
Notice that if $W$ is an $r$-wall its corners are not uniquely determined.
However, we assume to always be given some choice of corners which we identify with the names from above
(see for instance the red vertices in \autoref{@terminologie}). 
The \emph{corner society} of an $r$-wall $W$ is $(\CompassG{G}{W},\Omega_W)$ where $\Omega_W$ is a cyclic ordering of the corners that agrees with a cyclic ordering of $\Perimeter{W}.$ 

\begin{definition}[Flat Wall]\label{@providence}
Let $r\geq 2$ be an integer.
Let $G$ be a graph and $W$ be an $r$-wall in $G.$
We say that $W$ is \emph{flat} if $(\CompassG{G}{W},\Omega_W)$ has a vortex-free rendition in the disk.
\end{definition}

We wish to stress that this definition is not exactly the definition of flatness used in \cite{kawarabayashiTW21quickly} and \cite{KawarabayashiTW18anew}.
However, if $W$ is a flat wall in a graph $G,$ then $W-\Perimeter{W}$ can be seen to satisfy the stronger requirements for the flatness in \cite{kawarabayashiTW21quickly,KawarabayashiTW18anew} and any wall which is flat in this stronger sense also must be flat in our sense.

To state the Flat Wall Theorem in the terminology of  \cite{kawarabayashiTW21quickly,KawarabayashiTW18anew} we need to define what it means for a minor to be attached to the infrastructure provided by a given wall.

Let $G$ and $H$ be graphs.
If $H$ is a minor of $G,$ then it is possible to find connected subgraphs $X_v$ of $G$ for each $v\in\V{H}$ such that $X_u\cap X_v=\emptyset$ if $u\neq v,$ and if $uv\in\E{H},$ there exists an edge $e$ in $G$ with one end in $X_u$ and the other in $X_v.$
We say that $\CondSet{X_v}{v\in\V{H}}$ for a \emph{minor model} (or simply a \emph{model\,}) of $H$ in $G.$

Let $W$ be a wall in $G.$
We say that $W$ \emph{grasps an $H$ minor} if there exists a model $\CondSet{X_v}{v\in\V{H}}$ of $H$ in $G$ together with indices $i_v,j_v$ for each $v\in\V{H}$ such that $X_v$ meet the intersection of the $i_v$th row and $j_v$th column of $W.$

\begin{proposition}[Flat Wall Theorem \cite{robertson1995graph,KawarabayashiTW18anew}] 
Let $r,t\geq 1$ be integers, $R\coloneqq  49152t^{24}(40t^2+r),$ let $G$ be a graph and let $W$ be an $R$-wall in $G.$ 
Then either $G$ has a \hyperref[@attachements]{model} of a $K_t$-minor grasped by $W,$ or there exist a set $A\subseteq\V{G}$ of size at most $12288t^{24}$ and an $r$-subwall $W'\subseteq W-A$ which is flat in $G-A.$
\end{proposition}

\paragraph{Paths, Transactions, and Societies of Bounded Depth.}\label{@competition}
In order to be able to apply our refinement strategy to a given vortex together with its infrastructure provided by the refined version of the GMST  in \cite{kawarabayashiTW21quickly}, we require a particular situation for that infrastructure itself.
To achieve this situation we will first apply a preprocessing step which allows us to ``push'' the infrastructure as close to the actual vortex as possible.
For this, we need some additional terminology to handle large linkages between segments of a society.

If $P$ is a path and $x$ and $y$ are vertices on $P,$ we denote by $xPy$ the subpath of $P$ with endpoints $x$ and $y.$
Moreover, if $s$ and $t$ are the endpoints of $P$ and we have fixed an order of the vertices of $P,$ say $s$ is the first and $t$ the last vertex, then $xP$ denotes the path $xPt$ and $Px$ denotes the path $sPx.$
Let $P$ be a path from $s$ to $t$ and $Q$ be a path from $q$ to $p.$
If $x$ is a vertex in $\V{P}\cap\V{Q},$ then $PxQ$ is the path obtained from the union of $Px$ and $xQ.$
Let $X,Y\subseteq\V{G}.$
An \emph{$X$-path} is a path of length at least one with both endpoints in $X$ and internally disjoint from $X.$
In a society $(G,\Omega),$ we write $\Omega$-path as a shorthand for a $\V{\Omega}$-path.
A path is an \emph{$X$-$Y$-path} if it has one endpoint in $X$ and the other in $Y.$
Whenever we consider $X$-$Y$-paths we implicitly assume them to be ordered starting in $X$ and ending in $Y,$ expect if stated otherwise.

Let $(G,\Omega)$ be a society.
A \emph{segment} of $\Omega$ is a set $S\subseteq\V{\Omega}$ such that there do not exist $s_1,s_2\in S$ and $t_1,t_2\in\V{\Omega}\setminus S$ such that $s_1,t_1,s_2,t_2$ occur in $\Omega$ in the order listed, i.\@e.\@,  vertices of $V(\Omega)$ that appear consecutively in $\Omega.$
A vertex $s\in S$ is an \emph{endpoint} of the segment $S$ if there is a vertex $t\in\V{\Omega}\setminus S$ which immediately precedes or immediately succeeds $s$ in the order $\Omega.$
For vertices $s,t\in\V{\Omega}$ we denote by $s\Omega t$ the uniquely determined segment with first vertex $s$ and last vertex $t.$
In case $t$ immediately precedes $s,$ we define $s\Omega t$ to be the \emph{trivial segment} $\V{\Omega}.$

Let $G$ be a graph.
A \emph{linkage} in $G$ is a set of pairwise vertex disjoint paths.
In slight abuse of notation, if $\mathcal{L}$ is a linkage, we use $\V{\mathcal{L}}$ and $\E{\mathcal{L}}$ to denote $\bigcup_{L\in\mathcal{L}}\V{L}$ and $\bigcup_{L\in\mathcal{L}}\E{L}$ respectively.
Given two sets $A$ and $B$ we say that a linkage $\mathcal{L}$ is a \emph{$A$-$B$}-linkage if every path in $\mathcal{L}$ has one endpoint in $A$ and one endpoint in $B.$

Let $(G,\Omega)$ be a society. 
A \emph{transaction} in $(G,\Omega)$ is a linkage $\mathcal{L}$ of $\Omega$-paths in $G$ such that there exist disjoint segments $A,B$ of $\Omega$ where the paths in $\mathcal{L}$ are $A$-$B$-paths. 
We define the \emph{depth} of $(G,\Omega)$ as the maximum cardinality of a transaction in $(G,\Omega).$

Let $\mathcal{T}$ be a transaction in a society $(G,\Omega).$ 
We say that $\mathcal{T}$ is \emph{planar} if no two members of $\mathcal{T}$ form a cross in $(G,\Omega).$ 
An element $P\in\mathcal{T}$ is \emph{peripheral} if there exists a segment $X$ of $\Omega$ containing both endpoints of $P$ and no endpoint of another path in $\mathcal{T}.$ 
A transaction is \emph{crooked} if it has no peripheral element.

Finally we will need the following proposition.

\begin{proposition}[\!\!\cite{kawarabayashiTW21quickly}]\label{@thoroughly} 
Let $(G,\Omega)$ be a society and $p\geq 1,$ $q\geq 2$ positive integers. 
Let $\mathcal{P}$ be a transaction in $(G,\Omega)$ of order $p+q-2.$ 
Then there exists $\mathcal{P}'\subseteq\mathcal{P}$ such that $\mathcal{P}'$ is either a planar transaction of order $p$ or a crooked transaction of order $q.$
\end{proposition}

\paragraph{The Graph Minors Structure Theorem.}\label{@reproduces}

Next we present two different statements, both fit to capture the global structure of $H$-minor-free graphs.
The first one focusses on the structure relative to a wall and thus can be seen as a local extension of the \hyperref[thm_flatwall]{Flat Wall Theorem}, hence we call this one the \emph{Local Structure Theorem}.
The second one is the Graph Minors Structure Theorem that completely describes the structure of $H$-minor-free graphs in terms of graphs of bounded Euler-genus with a bounded number of bounded depth vortices, clique sums and apex vertices.
The corresponding graph parameter is asymptotically equivalent to $\mathsf{p}_{\mbox{\scriptsize \texttt{vga}}}$ and represents the global maximum of the $\texttt{vga}$-hierarchy (see \autoref{@assessment}).

\begin{definition}[Vortex Societies and Breadth and Depth of a $\Sigma$-Decomposition]\label{@phenomenon}
Let $\Sigma$ be a surface and $G$ be a graph.
Let $ \Delta=(\gamma,\mathcal{D})$ be a $\Sigma$-decomposition of $G.$
Every vortex $c$ defines a society $( \sigma(c),\Omega),$ called the \emph{vortex society} of $c,$ by saying that $\Omega$ consists of the vertices $\pi_{ \Delta}(n)$ for $n\in\widetilde{c}$ in the order given by $\gamma.$
(There are two possible choices of $\Omega,$ namely $\Omega$ and its reversal. Either choice gives a valid vortex society.).
The \emph{breadth} of $ \Delta$ is the number of cells $c\in C( \Delta)$ which are a vortex and the \emph{depth} of $ \Delta$ is the maximum depth of the vortex societies $( \sigma(c),\Omega)$ over all vortex cells $c\in C( \Delta).$
\end{definition}

Next we need to combine our definition of flat walls with the idea of $\Sigma$-decompositions.
This is a necessary step so to be able to relate flat renditions of walls to the drawings provided by $\Sigma$-decompositions and thus to impose additional structure onto these drawings.

Let $G$ be a graph and $W$ be a wall in $G$.
We say that $W$ is \emph{flat in a $\Sigma$-decomposition $ \Delta$ of $G$} if there exists a closed disk $ \Delta\subseteq\Sigma$ such that
\begin{itemize} 
\item the boundary of $ \Delta$ does not intersect any cell of $ \Delta,$ 
\item $\pi(N(\delta))\cap\Boundary{\Delta}\subseteq V(\Perimeter{W})$
\item for each degree-three vertex $v$ of $W$ such that $v$ is not mapped to a member of $N( \Delta)$ by $\pi,$ let $c_v\in C( \Delta)$ be the cell with $v\in V( \sigma(c_v)).$ Then, for all pairs of such distinct degree-three vertices $u$, $v$ of $W$ the cells $c_u$ and $c_v$ are disjoint and $c_v\subseteq \Delta$,
\item no cell $c\in C( \Delta)$ with $c\subseteq \Delta$ is a vortex, and 
\item $W-\V{\Perimeter{W}}$ is a subgraph of $\bigcup\CondSet{ \sigma(c)}{c\subseteq \Delta}.$
\end{itemize}

Let $G$ be a graph, let $r\geq 1$ be an integer, let $W'$ be an $r+2$-wall in $G,$ and let $W\subseteq W'$ be the unique $r$-subwall of $W'$ disjoint from its perimeter. 
If $(A,B)$ is a separation of $G$ of order at most $r-1,$ then exactly one of the sets $A\setminus B$ and $B\setminus A$ includes the vertex set of a column and a row of $W.$ 
If it is the set $A,$ we say that $A$ is the \emph{$W$-majority side of the separation $(A,B)$}; otherwise, we say that $B$ is the $W$-majority side.
For those readers familiar with the concept of tangles \cite{RobertsonS91X} (see \autoref{@explanation}), the orientation of all separations of order at most $r-1$ induced by the majority side of a wall is exactly the tangle induced by the wall.

To apply the intuition of this ``orientation'' of separations to $\Sigma$-decompositions recall that the boundary of every non-vortex cell contains at most three vertices.
Moreover, these boundaries define separations between the part of $G$ which is drawn in the interior of a single cell and the part of $G$ which is drawn outside.
To link a fixed wall $W$ in a graph $G$ with a $\Sigma$-decomposition we make use of this observation as follows.
Let $\Sigma$ be a surface and $ \Delta=(\gamma,\mathcal{D})$ be a $\Sigma$-decomposition of $G.$ 
We say that $ \Delta$ is \emph{$W$-central} if there is no cell $c\in C( \Delta)$ such that $\V{ \sigma(c)}$ includes the $W$-majority side of a separation of $G$ of order at most $r-1.$ 
Similarly, let $Z\subseteq\V{G},$ $\Abs{Z}\leq r-1,$ let $\Sigma'$ be a surface and $ \Delta'$ be a $\Sigma'$-decomposition of $G-Z.$ 
Then $ \Delta'$ is a $W$-central decomposition of $G-Z$ is for all separations $(A,B)$ of order at most $r-\Abs{Z}-1$ such that $B\cup Z$ is the majority side of the separation $(A\cup Z,Z\cup B)$ of $G,$ there is no cell $c\in C( \Delta')$ such that $\V{ \sigma_{ \Delta'}(c)}$ contains $B.$

\begin{proposition}[Local Structure Theorem \cite{RobertsonS03a,kawarabayashiTW21quickly}]\label{@artificially} 
Let $r,p\geq 0$ be integers, and let $R\coloneqq  49152p^{24}r+p^{10^7p^{26}}.$ 
Let $G$ be a graph and let $W$ be an $R$-wall in $G.$ 
Then either $G$ has a \hyperref[@attachements]{model} of a $K_p$-minor grasped by $W,$ or there exists a set $A\subseteq\V{G}$ of size at most $p^{10^7p^{26}},$ a surface $\Sigma$ of Euler-genus at most $p(p+1),$ a $W$-central $\Sigma$-decomposition $ \Delta$ of $G-A$ of depth at most $p^{10^7p^{26}}$ and breadth at most $2p^2,$ and an $r$-subwall $W'\subseteq W-A$ which is flat in $ \Delta.$
\end{proposition}

From the local structure theorem, a global version can be derived.
The way this is usually done is following a balanced separator argument.
That is, one fixes a small set of vertices $X$ and either finds a small balanced separator for it, which allows to continue the construction in each of the resulting components, or no such separator exists.
In the second case, $X$ witnesses large treewidth which allows one to deduce the existence of a big wall and thus, \autoref{@artificially} can be applied.
See the proof of \autoref{@diskussion} for a variant of this proof in full detail.

An \emph{$\alpha $-near embedding} of a graph $G$ in a surface $\Sigma$ of \emph{depth} $k$ and \emph{breadth} $t$ is a pair $( \Delta,A)$ such that $A\subseteq\V{G},$ $\Abs{A}\leq\alpha ,$ and $ \Delta$ is a $\Sigma$-decomposition of $G-A$ of depth\footnote{Please note that in \cite{kawarabayashiTW21quickly} there is a difference between the ``depth'' of a $\Sigma$-decomposition and the ``width'' of an $\alpha $-near embedding. This difference arises from the fact that after resolving the clique sums, each vortex can be decomposed into a path decomposition of bounded width. This width however is related to the depth of the vortex by a small constant factor.} $k$ and breadth $t$ such that for every $c\in C( \Delta)$ which is not a vortex, $\V{ \sigma(c)}$ induces a clique in $G$ and $\V{ \sigma(c)}\subseteq\pi(N( \Delta)).$

\begin{proposition}[Global Structure Theorem \cite{RobertsonS03a,kawarabayashiTW21quickly}]\label{@guarantors} 
There exists a constant $c$ that satisfies the following. 
Let $p\geq 1$ be a positive integer and let $G$ be a graph which does not contain $K_p$ as a minor. 
Let $\alpha \coloneqq  p^{18\cdot 10^7p^{26}+c}.$ 
Then $G$ has a tree decomposition $(T,\beta )$ of adhesion at most $4\alpha $ such that for all $t\in\V{T},$ if $G'$ is the torso of $G$ at $t$ then $G'$ has an $\alpha $-near embedding of breadth at most $2p^2$ and depth at most $\alpha $ in a surface of Euler-genus at most $p(p+1).$
\end{proposition}

\subsection{A refined version of \autoref{@artificially}}\label{@abbreviated}

The next step is to derive slightly refined versions of \autoref{@artificially}.
Towards this goal we first introduce additional definitions from \cite{kawarabayashiTW21quickly} and then describe in some detail how the refined versions follow from the proofs in \cite{kawarabayashiTW21quickly}.
A similar variant as the one we state below has been proven by Diestel, Kawarabyashi, M\"uller, and Wollan in \cite{diestel2012excluded}.
While their theorem exposes the vortices and the corresponding infrastructure in a similar way, for our purposes \autoref{@mouthpiece} below is more convenient to work with.
One reason for this is that the new statement is better suited to be incorportated into the balanced separator argument we touched upon above.

\paragraph{Societies and Nests.}

The proof of \autoref{@artificially} in \cite{kawarabayashiTW21quickly} is based on the systematic study of societies.
We start by introducing further definitions.

\begin{definition}[Cylindrical Rendition]\label{@positively} 
Let $(G,\Omega)$ be a society, $\rho=(\gamma,\mathcal{D})$ be a rendition of $(G,\Omega)$ in a disk, and let $c_0\in C(\rho)$ be such that no cell in $C(\rho)\setminus\Set{c_0}$ is a vortex. 
In those circumstances we say that the triple $(\gamma,\mathcal{D},c_0)$ is a \emph{cylindrical rendition} of $(G,\Omega)$ around $c_{0}.$
\end{definition}

Let $\rho=(\gamma,\mathcal{D})$ be a rendition of a society $(G,\Omega)$ in a surface $\Sigma.$
For every cell $c\in C(\rho)$ with $\Abs{\widetilde{c}}=2,$ we select one of the components of $\Boundary{c}-\widetilde{c}.$
This selection will be called a \emph{tie-breaker in $\rho$}, and we will assume that every rendition comes equipped with a tie-breaker.
Let $Q$ be either a cycle or a path in $G$ that uses no edge of $ \sigma(c)$ for every vortex $c\in C(\rho).$
We say that $Q$ is \emph{grounded} in $\rho$ if either $Q$ is a non-zero length path with both endpoints in $\pi_{\rho}(N(\rho)),$ or $Q$ is a cycle and it uses edges of $ \sigma(c_1)$ and $ \sigma(c_2)$ for two distinct cells $c_1,c_2\in C(\rho).$
If $Q$ is grounded we define the \emph{trace} of $Q$ as follows.
Let $P_1,\dots,P_k$ be distinct maximal subpaths of $Q$ such that $P_i$ is a subgraph of $ \sigma(c)$ for some cell $c.$
Fix an index $i.$
The maximality of $P_i$ implies that its endpoints are $\pi(n_1)$ and $\pi(n_2)$ for distinct nodes $n_1,n_2\in N(\rho).$
If $\Abs{\widetilde{c}}=2,$ define $L_i$ to be the component of $\Boundary{c}-\Set{n_1,n_2}$ selected by the tie-breaker, and if $\Abs{\widetilde{c}}=3,$ define $L_i$ to be the component of $\Boundary{c}-\Set{n_1,n_2}$ that is disjoint from $\widetilde{c}.$
Finally, we define $L_i'$ by pushing $L_i$ slightly so that it is disjoint from all cells in $C(\rho).$
We define such a curve $L_i'$ for all $i,$ maintaining that the curves intersect only at a common endpoint.
The \emph{trace} of $Q$ is defined to be $\bigcup_{i\in[k]}L_i'.$
So the trace of a cycle is the homeomorphic image of the unit circle, and the trace of a path is an arc in $ \Delta$ with both endpoints in $N(\rho).$

See \autoref{fig_sigma_dec} for an illustration of some part of a $\Sigma$-decomposition.
In this figure we have marked three cycles together with their traces which are closed curves, depicted as dashed red lines.
The same figure also illustrates the idea of the tie-breaker as one can see that the traces of the three cycles may alter between ``inside'' and ``outside'' the actual drawing of their respective cycle.

\begin{definition}[Nest]\label{@evangelistic} 
Let $\rho=(\gamma,\mathcal{D})$ be a rendition of a society $(G,\Omega)$ in a surface $ \Sigma$ and let $ \Delta\subseteq \Sigma$ be an arcwise connected set\footnote{That is, for every pair of points $x$ and $y$ in $ \Delta$ there exists a curve $\zeta$ which lies completely in $ \Delta$ and which contains $x$ and $y.$}.
A \emph{nest in $\rho$ around $ \Delta$ of order $s$} is a sequence $\mathcal{C}=(C_1,C_2,\dots,C_s)$ of disjoint cycles in $G$ such that each of them is grounded in $\rho,$ and the trace of $C_i$ bounds a closed disk $ \Delta_i$ in such a way that $ \Delta\subseteq \Delta_1\subsetneq \Delta_2\subsetneq\dots\subsetneq \Delta_s\subseteq \Sigma.$
\end{definition}

For an simple example of a {cylindrical rendition} one may readily consider
as the shallow vortex grid $\mathscr{S}_{6}$ depicted
in  \autoref{@industrial}. The ``outer face'' of the $(6,24)$-cylindrical grid of $\mathscr{S}_{6}$ 
is defining a cyclic ordering $\Omega$ of the vertices in its boundary.
The closure of the face where the crossings are drawn, is the (unique) 
vortex cell $c_0$ of the cylindrical rendition $\rho,$ where 
other cells correspond to the edges of the $(6,24)$-cylindrical grid
around $c_{0}.$ The disjoint cycles of this cylindrical grid are the cycles 
of a nest in $\rho$ around $c_0$ of order six.
The depth of the society defined by the cell $c_0$ is two.
Notice that this is essentially the minimum possible depth of a vortex as the society defined by a vortex of depth less than two has already a vortex-free rendition in a disk by the Two Paths Theorem (\autoref{@postscripts}).
This justifies the term ``shallow'' in the term shallow vortex grid that we use.

A more involved example of a nest can be found in \autoref{fig_sigma_dec} where we depict a nest of order three around the vortex cell $c_0$.

\begin{figure}[t]
\vspace{-0mm}
    \begin{center} 
    \scalebox{0.625}{\includegraphics{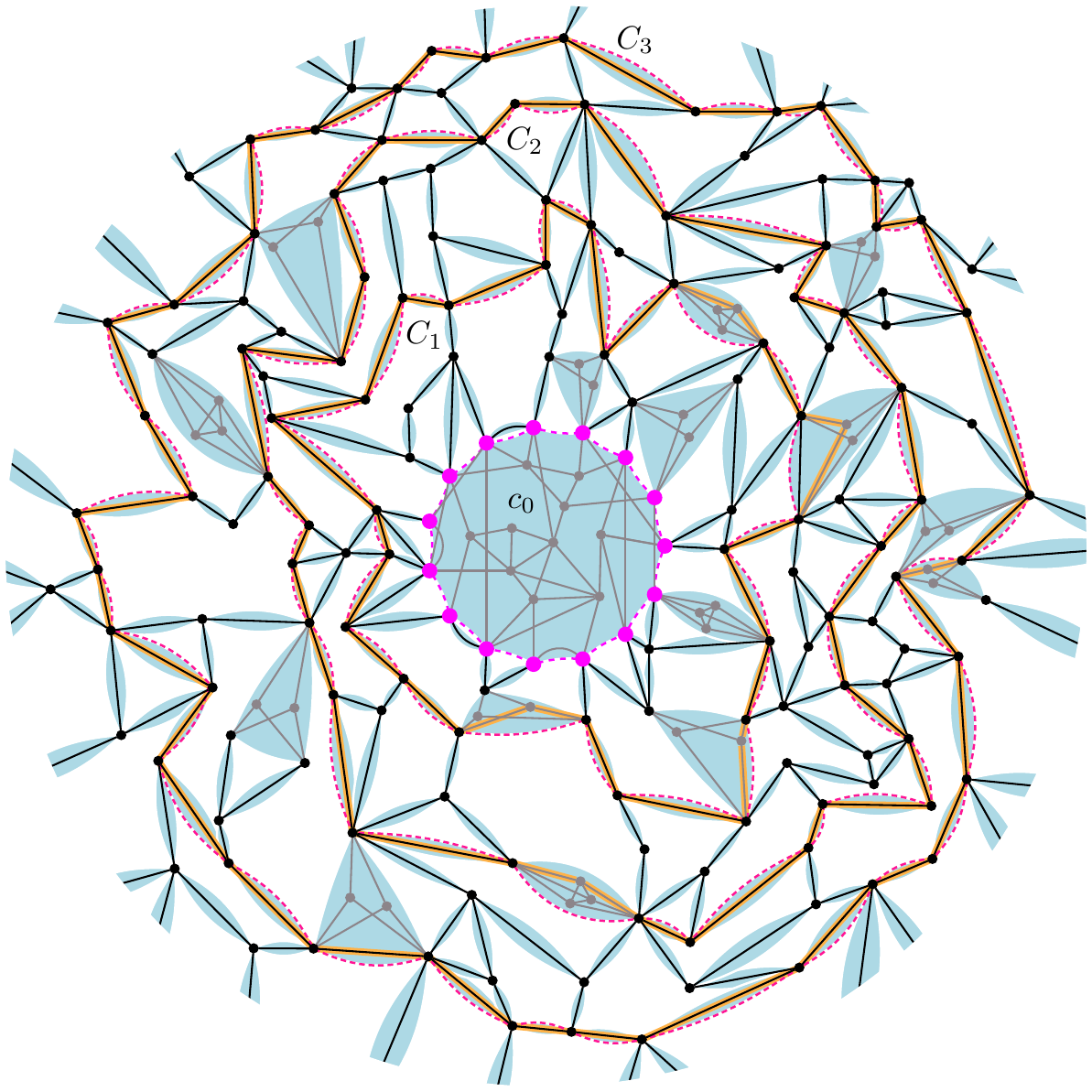}}
    \end{center} 
    \vspace{-4mm}
    \caption{A partial depiction of a $\Sigma$-decomposition focussing on a disk around a single vortex cell $c_0$ together with a nest of order three.
    Each of the three cycles $C_1$, $C_2$, and $C_3$ of the nest comes with its trace (the dashed \red{red} lines ``mirroring'' the behavior of the cycles on the boundary of the cells).} 
    \label{fig_sigma_dec}
\end{figure}

We are finally ready to state the refined version of \autoref{@artificially}.
The purpose of this refinement is to establish the nest structure around the vortices in \autoref{@artificially} in such a way that we are able to guarantee large nests for all vortices while also keeping these nests disjoint.
A similar statement can be found in \cite{diestel2012excluded} presented there in different terminology without explicit bounds, and with a slightly different definition for the depth of a vortex.

\begin{proposition}[Local Structure Theorem (with vortex infrastructure)\cite{RobertsonS03a,kawarabayashiTW21quickly}]\label{@mouthpiece} 
Let $r,p\geq 0$ be integers, and let $R\coloneqq  49152p^{24}r+p^{10^7p^{26}}.$ 
Let $G$ be a graph and let $W$ be an $R$-wall in $G.$ 
Then either $G$ has a \hyperref[@attachements]{model} of a $K_p$-minor grasped by $W,$ or there exists  
\begin{enumerate}[label=\roman*)]
\item a set $A\subseteq\V{G}$ of size at most $p^{10^7p^{26}},$ 
\item a surface $\Sigma$ of Euler-genus at most $p(p+1),$ 
\item a $W$-central $\Sigma$-decomposition $ \Delta=(\gamma,\mathcal{D})$ of $G-A$ of depth at most $p^{10^7p^{26}}$ and breadth at most $2p^2,$ and
\item an $r$-subwall $W'\subseteq W-A$ which is flat in $ \Delta.$
\end{enumerate}
Moreover, for each vortex cell $c\in C( \Delta),$
\begin{enumerate}[label=\roman*)]
\setcounter{enumi}{4}
\item there exists a nest $\mathcal{C}_c$ of order $10^{21}p^{100}$ in $ \Delta$ around the unique disk $ \Delta\in\mathcal{D}$ corresponding to $c,$ and 
\item let $ \Delta_c\subseteq\Sigma$ be the disk bounded by the trace of $C_{10^{21}p^{100}}\in\mathcal{C}_c,$ then for each pair of distinct vortex cells $c,c'\in C( \Delta)$ we have that $ \Delta_c\cap  \Delta_{c'}=\emptyset.$ 
\end{enumerate}
\end{proposition}

The reason why \autoref{@mouthpiece} is more convenient to work with than the main theorem of \cite{diestel2012excluded} lies in the fact that the latter was formulated for graphs of high treewidth only, while the above theorem only requires the existence of a large wall.
While high treewidth and the existence of large walls are famously asymptotically equivalent, it is much easier to derive a global version for all graphs using \autoref{@mouthpiece}.
Thus, if one is after a global description of graphs excluding a certain minor while maintaining some kind of infrastructure for each vortex, the above theorem is a bit more convenient.

\section{Combinatorial results} 
\label{@desexualized}

In \autoref{@exceptional} we presented all concepts and theorems that are necessary for the proof of our main combinatorial result, that is \autoref{@dialectical}. In this section we prove this result. The direction ${\mathsf{p}}_{\mbox{\scriptsize \texttt{-\!\!\! g\!\!\! a}}}\preceq \mathsf{ p}$ follows from \autoref{@lovelessly}, whose proof spans the first three subsections and the 
proof of ${\mathsf{p}}\preceq \mathsf{p}_{\mbox{\scriptsize \texttt{-\!\!\! g\!\!\! a}}}$ is given in \autoref{@inevitably}.

\subsection{Advancing through a nest}\label{@prejudiced}

In this subsection we are going to further refine \autoref{@mouthpiece}.
Our goal is to describe the structure surrounding a given vortex.

\begin{definition}[Rendition of a Nest]\label{@corresponding}
Let $(G,\Omega)$ be a society with a \hyperref[def_cylindricalrendition]{cylindrical rendition} $\rho=(\gamma,\mathcal{D},c_0)$ and let $ \Delta\in\mathcal{D}$ be the disk corresponding to $c_0.$
Let $\mathcal{C}=\Set{C_1,\dots,C_s}$ be a \hyperref[def_nest]{nest} in $\rho$ around $ \Delta$ of order $s\geq 1.$
We say that $(\rho,G,\Omega)$ is the \emph{rendition of $\mathcal{C}$ around $c_0$ in $\rho$} if $\V{C_s}=\V{\Omega}.$

Suppose $ \Delta=(\gamma',\mathcal{D}')$ is a \hyperref[def_sigmadecomposition]{$\Sigma'$-decomposition} of some graph $G'$ and $c\in C( \Delta)$ is a vortex cell with nest $\mathcal{C}'=\Set{C'_1,\dots,C'_h}$ such that the disk $ \Delta'$ containing $c$ defined by the trace of $C'_h$ does not contain a vortex cell besides $c.$
We denote by $( \Delta_{\mathcal{C}'},H_{\mathcal{C}'},\Omega_{\mathcal{C}'})$ the rendition of $\mathcal{C}$ around $c$ \emph{induced by $ \Delta'$ in $ \Delta$}.
\end{definition}

\begin{definition}[Tight Nests]\label{@subjective}
Let $\theta$ be a positive integer.
Let $(\rho,G,\Omega)$ be a rendition of a nest $\mathcal{C}$ around a vortex $c$ of depth at most $\theta$ in a disk $ \Delta.$
We say that $(\rho,G,\Omega)$ is \emph{$\theta$-tight} if one of the following is true
\begin{enumerate}[label=\textit{\roman*})] 
\item\label{@miraculous} there exists a set $Z\subseteq\V{G}$ such that $\Abs{Z}\leq \theta$ and every $\V{\Omega}$-$\widetilde{c}$-path in $G$ intersects $Z,$ or 
\item\label{@idealities} there exists a disk $ \Delta'\subseteq \Delta$ whose boundary intersects $\gamma$ only in nodes with $X=\Boundary{ \Delta'}\cap N(\rho)$ such that 
\begin{itemize} 
\item {$ \Delta'$ contains $c,$} 
\item there exists a family $\mathcal{P}$ of pairwise disjoint $\V{\Omega}$-$\widetilde{c}$-paths with $\Abs{\mathcal{P}}=\Abs{\pi(X)}\geq \theta,$ and 
\item if $\Omega'$ is an ordering of $\pi(X)$ induced\footnote{This means that $\Omega'$ is either the ordering obtained by following along the boundary of $ \Delta'$ in clockwise order, or its reversal.} by $ \Delta',$ then the society $(\InducedSubgraph{G}{\V{G}\cap  \sigma_{\rho}( \Delta')},\Omega')$ has \hyperref[def_transactionanddepth]{depth} at most $3\theta,$ where $ \sigma_{\rho}( \Delta')$ denotes $\bigcup_{c\in C(\rho),c'\subseteq \Delta'}V( \sigma_{\rho}(c')).$ 
\end{itemize}
\end{enumerate}
\end{definition}

\begin{figure}[h] 
\begin{center} 
\scalebox{.62}{\includegraphics{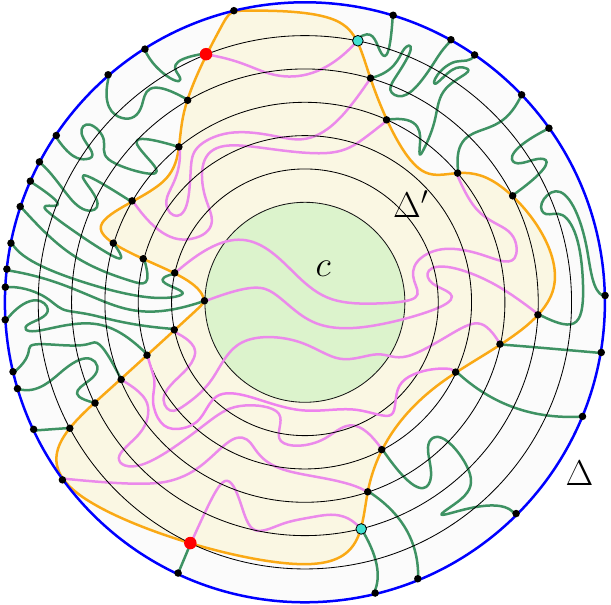}}
\end{center} 
\caption{A visualization of the proof of \autoref{@transformation} which finds a \hyperref[@maidservants]{tight} nest. The \darkgreen{green} paths are the portions of the paths in $\mathcal{P},$ each joining some vertex of $\Omega$ with some vertex of $\Omega'.$ The disk $ \Delta'$ is \darkorange{orange} and, in the first case, the set $Z$ can be seen as the vertices in its boundary. 
In the second case, the graph drawn inside the closed disk $ \Delta',$ along with the ordering of $\Omega',$ should form a society of depth at most $3\theta.$ Towards a contradiction we assume the existence of $3\theta+1$ violet paths forming a linkage between two disjoint segments of $\Omega'$ (the one segment is between the two \red{red} vertices and the other between the two \blue{blue} vertices).} 
\label{@axiomatically} 
\end{figure}

The remainder of this section is dedicated to the proof of the following theorem.

\begin{theorem}[Local Structure Theorem (with tight vortex infrastructure)] \label{@transformation}
Let $r,p\geq 0$ be integers, and let $R\coloneqq  49152p^{24}r+p^{10^7p^{26}}.$ 
Let $G$ be a graph and let $W$ be an $R$-wall in $G.$ 
Then either $G$ has a \hyperref[@attachements]{model} of a $K_p$-minor grasped by $W,$ or the second outcome of \autoref{@mouthpiece} applies along with the following additional condition: 
 
\begin{enumerate}[label=\textit{\roman*})] 
\item[vii)] for each vortex cell $c\in C( \Delta)$ the rendition $( \Delta_{\mathcal{C}_c},H_{\mathcal{C}_c},\Omega_{\mathcal{C}_c})$ of the nest $\mathcal{C}_c$ is $p^{10^7p^{26}}$-\hyperref[@maidservants]{tight}. 
\end{enumerate}
\end{theorem}

Towards a proof of \autoref{@transformation}, we first show a slightly more general lemma that allows us to choose a new nest ``closer'' to the vortex whenever none of the two cases of the definition of a \hyperref[@maidservants]{tight nest} holds.
\autoref{@transformation} will then follow from iterative applications of this lemma.
Moreover, the lemma is constructive in the sense that is only uses some re-routing arguments and Menger's Theorem, hence it can be applied in polynomial time.
Therefore, from the polynomial time algorithm in \cite{kawarabayashiTW21quickly} it follows that the $\Sigma$-decomposition of \autoref{@transformation} can be found in polynomial time.

Let $\rho=(\gamma,\mathcal{D},c_0)$ be a \hyperref[def_cylindricalrendition]{cylindrical rendition} of a society $(G,\Omega)$ and let $\mathcal{C}=\Set{C_1,\dots,C_s}$ be a nest around $c_0$ in $\rho.$
We associate a vector $\mathbf{v}_{\mathcal{C}}\in\Bbb{N}^s$ with $\mathcal{C}$ as follows.
For each $i\in[0,s-1]$ let $\mathbf{v}_i$ be the number of nodes of $\rho$ which are contained in the disk that contains $c_0$ and is bounded by the trace of $C_{i+1}$.
The vector $\mathbf{v}_{\mathcal{C}'}$ is defined analogously.

Let $\mathcal{C}'=\Set{C_1',\dots,C_s'}$ be a nest around $c_0$ in $\rho.$ 
We now write $\mathcal{C}<\mathcal{C'}$ if $\mathbf{v}_{\mathcal{C}}<_{\operatorname{lex}}\mathbf{v}_{\mathcal{C}'},$ that is if $\mathbf{v}_{\mathcal{C}}$ is lexicographically smaller than $\mathbf{v}_{\mathcal{C}'}.$

\begin{lemma}\label{@metaphysician} 
Let $\theta\geq 2$ and $s$ be positive integers with $s\leq \frac{\theta}{2}.$ 
Let $(\rho=(\gamma,\mathcal{D},c_0),G,\Omega)$ be a rendition of the nest $\mathcal{C}$ of order $s$ around the vortex $c_0$ of depth at most $\theta$ in $\rho.$ 
Then either $\mathcal{C}$ is \hyperref[@maidservants]{$\theta$-tight}, or there exists a nest $\mathcal{C}'<\mathcal{C}$ of order $s$ around $c_0$ within $G.$
\end{lemma}

\begin{proof} 
Let $\mathcal{P}$ be a maximum family of pairwise disjoint $\V{\Omega}$-$\widetilde{c_0}$-paths in $G.$ 
If $\Abs{\mathcal{P}}\leq\theta$ then, by Menger's Theorem, there exists a set $Z\subseteq\V{G}$ of size at most $\theta$ which meets all $\V{\Omega}$-$\widetilde{c_0}$-paths in $G$ and thus $(\rho,G,\Omega)$ is  
\hyperref[@maidservants]{$\theta$-tight}. 
 
Hence, we may assume $\Abs{\mathcal{P}}\geq \theta+1.$ 
Again by Menger's Theorem we can find a set $X\subseteq\V{G}$ with $\Abs{X}=\Abs{\mathcal{P}}$ since $\mathcal{P}$ such that $X$ contains a vertex of every path in $\mathcal{P}$ since $\mathcal{P}$ is maximum. 
In particular, we may assume $X\subseteq \pi(N(\rho)),$ that is, no vertex of $X$ is drawn in the interior of a cell. 
In the following, in a slight abuse of notation, we will identify the sets $X$ and $\pi^{-1}(X).$
Since $G-c_0$ has a vortex-free rendition in the disk and there is no $\V{\Omega}$-$\widetilde{c_0}$-path in $G-X-c_0$ we may find a disk $ \Delta$ whose boundary intersects $N(\rho)$ exactly in $X,$ is otherwise disjoint from $\gamma,$ and $c_0\subseteq  \Delta.$ 
Let $\Omega'$ be an ordering of the vertices in $X$ obtained by traversing the boundary of $ \Delta$ in an arbitrarily chosen direction. 
If the society $(\InducedSubgraph{G}{\V{G}\cap \sigma_{\rho}( \Delta)},\Omega')$ has depth at most $3\theta$ we are done since this would mean that $(\rho,G,\Omega)$ is \hyperref[@maidservants]{$\theta$-tight}. 
 
Hence, we may assume that there exist disjoint segments $I_1'$ and $I_2'$ of $\Omega'$ and a family $\mathcal{L}$ of at least $3\theta+1$ pairwise disjoint $I_1'$-$I_2'$-paths in $G'\coloneqq  \InducedSubgraph{G}{\V{G}\cap  \sigma_{\rho}( \Delta)}$ (see \autoref{@axiomatically}). 
Suppose $\mathcal{L}$ contains a crooked transaction of order $\theta+1.$ 
This means that we can select subpaths of some paths in $\mathcal{L}$ to obtain a \hyperref[def_crooked]{crooked transaction} on the \hyperref[def_vortexsociety]{vortex society} $( \sigma(c_0),\Omega_0)$ and therefore contradicts the assumption that $c_0$ is of depth at most $\theta.$ 
Hence, by \autoref{@thoroughly}, $\mathcal{L}$ contains a planar transaction $\mathcal{L}_1$ of order at least $2\theta+1.$ 
Moreover, if $\theta+1$ paths of $\mathcal{L}_1$ contain an edge of $ \sigma(c_0),$ then we can find a transaction of order $\theta+1$ in $( \sigma(c_0),\Omega_0).$ 
Hence, there exists a family $\mathcal{L}_2\subseteq\mathcal{L}_1$ of size at least $\theta+1$ which does not contain an edge of $ \sigma(c_0)$ nor a vertex of $ \sigma(c_0-\widetilde{c_0}).$ 
Note that $\mathcal{L}_2$ must be a planar transaction on $(\InducedSubgraph{G}{\V{G}\cap V( \sigma( \Delta))},\Omega')$ since $G-c_0$ has a vortex-free rendition in the disk. 
Moreover, using the paths from $\mathcal{P}$ and the fact that each vertex of $X$ is an endpoint of some $\V{\Omega}$-$X$-path that is a subpath of some path in $\mathcal{P}$ we can extend $\mathcal{L}_2$ to be a planar transaction $\mathcal{L}_3$ of size $h\geq\theta+1$ of $(G,\Omega).$ 
 
We now describe how to use $\mathcal{L}_3$ to obtain the nest $\mathcal{C}'.$ 
Let $I_1$ and $I_2$ be the two segments of $\Omega$ projected from $I_1'$ and $I_2'$ as follows. 
For each vertex of $\Omega'$ there is a unique vertex of $\Omega$ such that these two are linked via a subpath of some path in $\mathcal{P}.$ 
Given some segment $I'$ of $\Omega'$ we say its \emph{projection to $\Omega$} is the smallest segment $I$ of $\Omega$ containing all vertices of $\Omega$ which are linked to $I'$ by subpaths of paths in $\mathcal{P}.$ 
Let us number the paths in $\mathcal{L}_3=\Set{L_1,L_2,\dots,L_h}$ according to the appearance of their starting point on $I_1.$ 
From here on we declare the \emph{inside} of a cycle $C\in\mathcal{C}$ to be everything drawn by $\gamma$ onto the closed disk bounded by the trace of $C$ that contains $c_0,$ but none of the vertices or edges of $C$ itself. 
Now let $i\in[s]$ be the smallest number such that the inside of $C_i$ contains a subpath $P$ of some path of $\mathcal{L}_3$ such that that all of $P$ except its endpoints is drawn on the inside of $C_i,$ both endpoints of $P$ lie on $C_i,$ and $P$ is disjoint from all cycles in $\mathcal{C}\setminus\Set{C_i}.$ 
 
Suppose the number $i$ exists. 
Then the graph $C_i\cup P$ contains a unique cycle that uses edges of $P$ and whose trace bounds a disk that contains $c_0.$ 
Let $C_i'$ be this cycle. 
Note that, since $G$ is a simple graph, there must exist a vertex $v\in\V{C_i}$ which is not contained in $C_i'$ and also not on the inside of $C_i'.$ 
Hence, $\mathcal{C}'=\Set{C_1,\dots,C_{i-1},C_i',C_{i+1},\dots,C_s}<\mathcal{C}$ and we are done. 
 
So we may assume that $i$ does not exist. 
Let $k\in[h]$ be any number with $L_k\in\mathcal{L}_3$ and let $j_k\in[s]$ be the smallest integer such that $L_k$ contains vertices of $C_{j_k}.$ 
By assumption $L_k$ cannot contain any edges or vertices drawn on the inside of $C_{j_k}$ since otherwise the number $i$ would exist. 
 
\medskip 
 
We claim that either the paths in $\CondSet{L_z}{z\in[k-1]}$ or the paths in $\CondSet{L_z}{z\in[k+1,h]}$ are disjoint from $C_{j_k}.$ 
Suppose the claim is false and there exist $a\in[k-1]$ and $b\in[k+1,h]$ such that $L_a$ and $L_b$ both meet $C_{j_k}.$ 
Let $x_1$ be the first vertex of $L_k$ on $C_{j_k}$ and let $x_2$ be the last vertex of $L_k$ on $C_{j_k}.$ 
Then $C_{j_k}$ is divided into two paths, say $Q_1$ and $Q_2,$ both with endpoints $x_1$ and $x_2.$ 
Since $\mathcal{L}_3$ is a planar transaction in a vortex-free cylindrical rendition of $(G-c_0,\Omega)$ exactly one of the paths, say $L_a,$ can intersect $Q_1$ and thus $L_b$ must intersect $Q_2.$ 
Let $Q_1'$ be a maximal subpath of $L_a\cap Q_1$ and let $Q_2'$ be a maximal subpath of $L_b\cap Q_2.$ 
Observe that $L_k$ must avoid both $Q_1'$ and $Q_2'.$ 
Moreover, $L_k$ always stays ``in between'' $L_a$ and $L_b,$ meaning that the graph $C_h\cup L_a\cup L_b$ contains a unique cycle containing both paths $L_a$ and $L_b$ whose trace bounds a disk that fully contains $L_k.$ 
It follows that $L_k$ must contain a non-trivial subpath that lies on the inside of $C_{j_k}$ contradicting our observation above. 
So our claim follows. 
 
For each $j\in[s]$ let $\mathcal{R}_j\subseteq\mathcal{L}_3$ be the collection of all paths $L_k\in\mathcal{L}_3$ such that $j=j_k.$ 
From the previous discussion it follows that $\Abs{\mathcal{R}_j}\leq 2$ for all $j\in[s].$ 
However, since $2s\leq\theta$ and $\Abs{\mathcal{L}_3}\geq\theta+1$ there must exist some $\ell\in[s]$ such that $\Abs{\mathcal{R}_{\ell}}\geq 3.$ 
As this is a contradiction, the case where the number $i$ does not exist cannot occur and our claim follows.
\end{proof}

\begin{proof}[Proof of \autoref{@transformation}] 
Let $G$ be a graph and $W$ be an $R$-wall in $G.$ 
By \autoref{@mouthpiece} we know that either $G$ contains a $K_p$-minor \hyperref[@attachements]{model} grasped by $W,$ or there exists a set $A\subseteq\V{G},$ a surface $\Sigma$ and a $W$-central $\Sigma$-decomposition $ \Delta$ satisfying items \emph{i.}--\emph{vi.} of our theorem. 
We may assume the nests $\mathcal{C}_c$ of all vortex cells $c\in C( \Delta)$ to be chosen such that $\mathcal{C}_c\leq\mathcal{C}$ for all other nests of order $10^{21}p^{100}$ around $c$ that satisfy point vi) of the theorem. 
 
Suppose there exists $c\in C( \Delta)$ such that $\mathcal{C}_c=\Set{C_1,\dots,C_{p^{10^7p^{26}}}}$ is not $p^{10^7p^{26}}$-tight. 
Then \autoref{@metaphysician} yields the existence of a nest $\mathcal{C}<\mathcal{C}_c$ around $c$ which is contained in the disk bounded by the trace of $C_{p^{10^7p^{26}}}$ that contains $c.$ 
Since $\CondSet{\mathcal{C}_x}{x\in C( \Delta)\text{ and }x\text{ is a vortex}}$ satisfies vi) this means that $\CondSet{\mathcal{C}_x}{x\in C( \Delta)\text{, }x\neq c\text{, and }x\text{ is a vortex}}\cup\Set{\mathcal{C}}$ must also satisfy vi). 
Hence, we have reached a contradiction to our previous assumption.
\end{proof}

\subsection{The death of a vortex}\label{@supervision}

To create a minor model of a shallow vortex grid we need to be able to find a large number of crosses on the society of a nest such that these crosses are pairwise disjoint and occur in a sequential fashion on the society.
In this section we discuss how to either separate the non-planar part of a vortex completely from its nest by a bounded size set of vertices, and therefore ``kill'' the vortex, or to find these crosses to construct our minor.

\begin{definition}[Cross Over a Segment]\label{@forgctfiilncss} 
Let $(G,\Omega)$ be a society and $S$ be a segment of $\Omega.$ 
We say that a pair of $V(G)$-paths $(P_1,P_2)$ form a \emph{cross over $S$} if they form a cross over $(G,\Omega)$ and all of their four endpoints lie in $S.$
\end{definition}
\begin{definition}[Consecutive Crosses]\label{@understandable} 
Let $(\rho,G,\Omega)$ be a rendition of a nest $\mathcal{C}$ around a vortex $c$ in a disk $ \Delta$ in $\rho$.
A family $\mathcal{C}=\Set{(L_1,R_1),\dots,(L_h,R_h)}$ of crosses over $(G,\Omega)$ is said to be \emph{consecutive} if there exist segments $I_1,\dots, I_h$ of $\Omega$ such that 
\begin{enumerate}[label=\textit{\roman*})] 
\item for each $i\in[h-1]$ the last vertex of $I_i$ comes before the first vertex of $I_{i+1}$ and $I_1\cap I_h=\emptyset,$ 
\item $\bigcup_{i\in[h]}\Set{L_i,R_i}$ is a family of pairwise disjoint paths, and 
\item for each $i\in[h],$ the pair $(L_i,R_i)$ is a cross over $I_i.$ 
\end{enumerate}
\end{definition}

The main result of this section provides a duality that certifies for every vortex with a large enough nest around it that either this vortex can be completely separated from the properly embedded part of the graph with a small set of vertices or we find the desired minor.

\begin{lemma}\label{@constitution} 
Let $t\leq \theta$ be positive integers. 
There exists a positive universal constant $\mathsf{c}$ such that, if $(\rho=(\gamma,\mathcal{D},c),G,\Omega)$ is a \hyperref[@maidservants]{$\theta$-tight} rendition of a nest $\mathcal{C},$ with $\Abs{\mathcal{C}}\geq 12t^2+\mathsf{c},$ around a vortex $c$ of depth at most $\theta$ in a disk $ \Delta,$ then one of the following holds.
\begin{enumerate}[label=\textit{\roman*})] 
\item There exists a separation $(A,B)$ of order at most $12\theta(t-1)$ with $V(\Omega)\cap B\subseteq A\cap B,$ such that, if $\Omega'$ is the restriction of $\Omega$ to $A\setminus B$ then $(\InducedSubgraph{G}{A\setminus B},\Omega')$ has a vortex-free rendition in the disk, or 
\item $G$ contains the shallow vortex grid of order $t$ as a minor.
\end{enumerate} 
\end{lemma}

\begin{definition}[Patches]\label{@intransigent} 
Let $\theta$ be a positive integer. 
Let $(\rho=(\gamma,\mathcal{D},c),G,\Omega)$ be a \hyperref[@maidservants]{$\theta$-tight} rendition of a nest $\mathcal{C}$ around a vortex $c$ of depth at most $\theta$ in a disk $ \Delta$ in $\rho.$ 
Moreover, assume that there exists a disk $ \Delta'\subseteq \Delta$ with $c\subseteq  \Delta',$ whose boundary intersects $\gamma$ in a set $X$ of nodes only such that 
\begin{itemize} 
\item there exists a family $\mathcal{P}$ of pairwise disjoint $V(\Omega)$-$\widetilde{c}$-paths with $\Abs{\mathcal{P}}=\Abs{\pi(X)}\geq\theta,$ and 
\item the society $(G'\coloneqq  \InducedSubgraph{G}{V(G)\cap V( \sigma( \Delta'))},\Omega')$ has depth at most $3\theta$ where $\Omega'$ is an ordering of $\pi(X)$ induced by $ \Delta'.$ 
\end{itemize}
Finally, let $S$ be a segment of $\Omega'$ and $Z\subseteq V(G').$
Let $Y$ be the collection of all vertices in $V(G')$ which are contained in a connected component of $G'-Z$ with a vertex of $S,$ we call $Y$ the \emph{patch of $(G',\Omega')$ cut at $S$ by $Z$}.
Let $\hat{G}\coloneqq  \InducedSubgraph{G}{V(G)\setminus V(G'-Y)}$ and let $\hat{\Omega}$ be the cyclical ordering of $V(\Omega)\cap V(\hat{G})$ induced by $\Omega.$
We call $(\hat{G},\hat{\Omega})$ the \emph{society from $(G,\Omega)$ cut by $Z$ and patched at $S$} (\emph{with the patch $Y$}).

\begin{figure}[h] 
\begin{center} 
\scalebox{.62}{\includegraphics{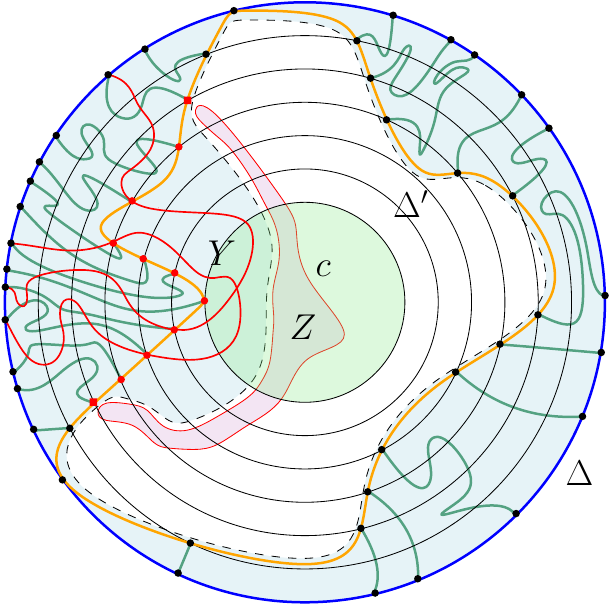}}
\end{center} 
\caption{A visualization of the concepts of \autoref{@intransigent}. The \green{green} paths are the portions of the paths in $\mathcal{P},$ each joining some vertex of $\Omega$ with some vertex of $\Omega'.$ The graph $G'$ is what is cropped by the disk $ \Delta'$ whose boundary is the \darkorange{orange}  cycle. The segment $S$ contains the 10 red vertices. The graph $\hat{G}$ contains everything that is not inside the dashed cycle (that is, the graph inside the \textcolor{celestialblue}{light blue} area). The two \red{red} paths form a cross of the segment $S$ at $Z.$} 
\label{@nationalists} 
\end{figure}

We say that a segment $S$ of $\Omega'$ \emph{has a cross at a set $Z\subseteq V(G')$} if the society from $(G,\Omega)$ cut by $Z$ and \hyperref[def_patch]{patched} at $S$ has a cross. (See \autoref{@nationalists} for a visualization of the above defined concepts.)

Let $\theta$ be a positive integer. 
Let $(\rho=(\gamma,\mathcal{D},c),G,\Omega)$ be a \hyperref[@maidservants]{$\theta$-tight} rendition of a nest $\mathcal{C}$ around a vortex $c$ of depth at most $\theta$ in a disk $ \Delta$ in $\rho.$ 
Moreover, assume that there exists a disk $ \Delta'\subseteq \Delta$ whose boundary intersects $\gamma$ only in vertices with $X=\Boundary{ \Delta'}\cap N(\rho)$ such that
\begin{itemize} 
\item there exists a family $\mathcal{P}$ of pairwise $V(\Omega)$-$\widetilde{c}$-paths with $\Abs{\mathcal{P}}=\Abs{\pi(X)}\geq\theta,$ and 
\item the society $(G'\coloneqq  \InducedSubgraph{G}{V(G)\cap V( \sigma( \Delta'))},\Omega')$ has depth at most $3\theta$ where $\Omega'$ is an ordering of $\pi(X)$ induced by $ \Delta'.$ 
\end{itemize}
We say that the tuple $(\theta,\rho,G,\Omega,\mathcal{C},\Delta,\Delta',\mathcal{P})$ is a \emph{$\theta$-suspension} of $c$ in $(\rho,G,\Omega)$

\end{definition}

\begin{lemma}\label{@swineherds} 
Let $\theta$ be a positive integer. 
Let $(\rho=(\gamma,\mathcal{D},c),G,\Omega)$ be a \hyperref[@maidservants]{$\theta$-tight} rendition of a nest $\mathcal{C}$ around a vortex $c$ of depth at most $\theta$ in a disk $ \Delta$ in $\rho$ and let $(\theta,\rho,G,\Omega,\mathcal{C},\Delta,\Delta',\mathcal{P})$ be a $\theta$-suspension of $c$ in $(\rho,G,\Omega)$.
 
Now let $S_1,\dots,S_{\ell}$ be pairwise disjoint segments of $\Omega'$ such that for each $i\in[\ell]$ there exists a set $Z_i\subseteq V(G')$ separating $S_i$ from $V(\Omega')\setminus S_i.$ 
For each $i\in[\ell]$ let $(G_i',\Omega_i')$ be the society from $(G',\Omega')$ cut by $Z\coloneqq \bigcup_{i\in[\ell]}Z_i$ and \hyperref[def_patch]{patched} at $S_i$ with the patch $Y_i.$
 
If for each $i'\in[\ell]$ the society $(G_{'i'}',\Omega_{'i'}')$ has a cross, then for every $i\in[\ell]$ there exists a segment $I_i$ of $\Omega$ and a pair of $\V{\Omega}$-paths $(L_i,R_i)$ such that 
\begin{enumerate}[label=\textit{\roman*})] 
\item if $i\neq j\in[\ell]$ then $I_i\cap I_j=\emptyset,$ 
\item if $i\neq j\in[\ell]$ then $(V(L_i)\cup V(R_i))\cap (V(L_j)\cup V(R_j))=\emptyset,$ 
\item $(L_i,R_i)$ forms a cross over $I_i,$ and 
\item $L_i$ and $R_i$ intersect every cycle in $\mathcal{C}.$ 
\end{enumerate}
\end{lemma}
See \autoref{@mimeographic} for an illustration of the statement of the above lemma.

\begin{figure}[h] 
\begin{center} 
\scalebox{.62}{\includegraphics{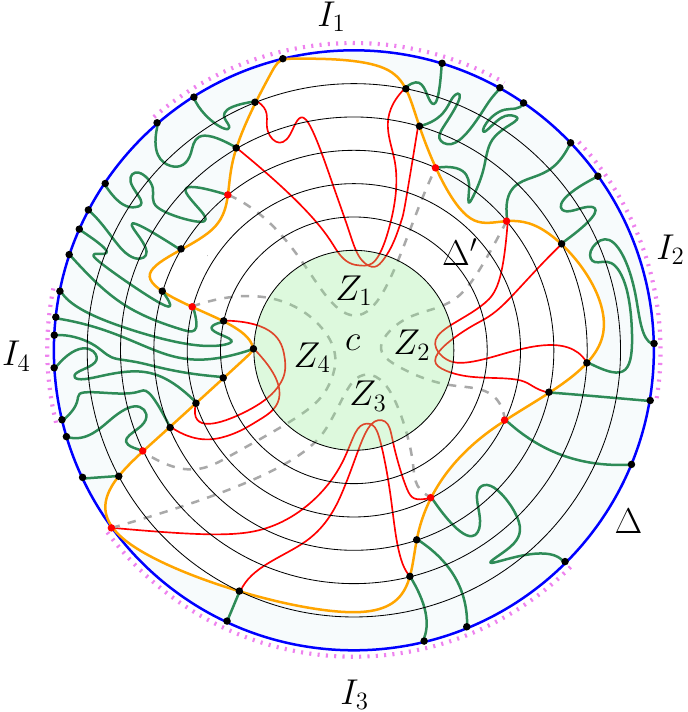}}
\end{center} 
\caption{A visualization of \autoref{@swineherds} for $\ell=4.$ The \green{green}  paths are the portions of the paths in $\mathcal{P},$ each joining some vertex of $\Omega$ with some vertex of $\Omega'.$ 
The segments of $\Omega$ are delimited by the \red{red} vertices, the dashed grey lines represent the separating sets $Z_{1},\ldots,Z_{4},$ the crossing paths in $(G_i',\Omega_i'), i\in[4]$ are red and the \darkmagenta{magenta} dotted lines indicate the segments $I_{1},\ldots,I_{4}.$} 
\label{@mimeographic} 
\end{figure}

\begin{proof} 
Since $(G_i',\Omega_i')$ has a rendition in the disk with a single vortex, for every $i\in[\ell]$ we find paths $Q_1^i$ and $Q_2^i$ for both $i\in[2]$ such that $(Q_1^i,Q_2^i)$ forms a cross over $(G_i',\Omega_i').$ 
Note that this means that $Q_1^i $ and $Q_2^i$ both contain an edge that is drawn in the interior of $c.$  
Let $J_i'$ be the smallest segment of $\Omega'_i$ such that $(Q_1^i,Q_2^i)$ is a cross over $J_i'$ and let $J_i$ be the smallest segment of $\Omega'$ containing $J_i'$ which is disjoint from all $J_j',$ $i\neq j\in[\ell].$ 
Finally, for each $i\in[\ell]$ and each $j\in[2]$ we denote by $x_{i,j}$ and $y_{i,j}$ the two endpoints of $Q_j^i$ such that $x_{i,j}$ comes before $y_{i,j}$ in the order $\Omega'.$ 
Moreover, we may assume $Q_1^i$ and $Q_2^i$ to be numbered such that $x_{i,1}$ comes before $x_{i,2}$ on $\Omega'.$ 
 
Notice that $\Abs{J_i}\geq 4$ for all $i\in[\ell].$ 
Moreover, since $(G,\Omega)$ is {$\theta$-tight}, from every vertex $v\in V(\Omega')=X$ there is a path $P_v\in\mathcal{P}$ joining $v$ to a vertex of $V(\Omega)=V(C_{\Abs{\mathcal{C}}})$ and all of these paths are internally disjoint from $G'.$ 
 
For every $i\in[\ell]$ let us choose the paths $\mathcal{P}_i\coloneqq \Set{P_{x_{i,1}},P_{y_{i,1}},P_{x_{i,2}},P_{y_{i,2}}}\subseteq\mathcal{P}.$ 
Notice that $\bigcup_{i\in[\ell]}\mathcal{P}_i$ is still a family of pairwise disjoint paths. 
Moreover, $Y_i\cap Y_j=\emptyset$ for all $i\neq j\in[\ell]$ by the definition of the $Y_i$ and thus $\bigcup_{i\in[\ell]}\Set{Q_1^i,Q_2^i}$ is also a family of pairwise disjoint paths. 
At last, notice that the endpoints of $P_{x_{i,1}}$ and $P_{y_{i,1}}$ on $V(\Omega)$ define two segments of $\Omega,$ each of which containing an endpoint of one of the two paths $P_{x_{i,2}}$ and $P_{y_{i,2}}.$ 
 
Now let $I_i$ be the segment of $\Omega$ with endpoints $x_{i,1}$ and $y_{i,2}$ that contains $y_{i,1}$ and $x_{i,2}.$ 
Then let $L_i\coloneqq  P_{x_{i,1}}x_{i,1}Q_1^iy_{i,1}P_{y_{i,1}}$ and $R_i\coloneqq  P_{x_{i,2}}x_{i,2}Q_2^iy_{i,2}P_{y_{i,2}}.$ 
From the observations above it is now clear that each $(L_i,R_i)$ forms a cross over $I_i$ and that $\bigcup_{i\in[\ell]}\Set{L_i,R_i}$ is a family of pairwise disjoint paths. 
Moreover, each $L_i$ and each $R_i$ must contain an edge that is drawn in the interior of $c.$ 
Since $\rho$ is a rendition in the disk and both $L_i$ and $R_i$ have their endpoints on $\Omega$ this implies that each such path meets all cycles of $\mathcal{C}.$ 
At last, observe that the segments $J_i$ and $J_j$ on $\Omega'$ are disjoint if $i\neq j.$ 
Suppose there are $i\neq j\in[\ell]$ such that $I_i\cap I_j\neq\emptyset.$ 
Without loss of generality let us assume that $x_{i,1}$ comes before $x_{j,1}$ on $\Omega.$ 
Then we have $y_{i,2}\in I_j$ and $x_{j,1}\in I_i.$ 
Let $F_1$ be the path obtained by starting in $y_{i,2},$ following along $P_{y_{i,2}}$ until we meet $C_1\in\mathcal{C}$ in the vertex $u_1$ for the first time, then following along $C_1$ within $Y_i\cup Z$ until we meet a vertex $v_1$ of $P_{x_{i,1}}$ and from here following along $P_{x_{i,1}}$ until we meet $x_{i,1}.$ 
By the discussion above $F_1$ must be completely disjoint from the paths $L_j$ and $R_j$ since we only used vertices in $Y_i\cup Z$ and from the paths $L_i$ and $L_j.$ 
Let us construct the path $F_2$ in a similar way by starting in $x_{j,1}$ and ending in $y_{j,2}$ and let $u_2$ and $v_2$ be defined analogously to $u_1$ and $v_1.$ 

If follows that any vertex of $V(F_1)\cap V(F_2)$ must belong to $Z\cap V(C_1).$
To see that this is impossible let us consider two subpaths of $C_1.$ 
Let $B_1$ be the $u_1$-$v_1$ subpath of $C_1$ that does not contain a vertex of $\bigcup_{h\in[\ell]\setminus\Set{i}}Y_h.$ 
Similarly let $B_2$ be the $u_2$-$v_2$ subpath of $C_1$ that does not contain a vertex of $\bigcup_{h\in[\ell]\setminus\Set{j}}Y_h.$ 
These paths must exist since there always exists an $u_h$-$v_h$-path on $C_1$ that lies in $Y_h\cup Z$ for both $h\in[2].$ 
The only way where these paths can intersect is, if their endpoints appear on $C_1$ in the order $u_1,u_2,v_1,v_2$ or $u_1,v_2,v_1,u_2.$ 
In both cases, however, it would follow that $Y_i\cap Y_j\neq\emptyset$ which contradicts our assumptions. 
Hence, $F_1$ and $F_2$ are disjoint. 
This implies the existence of a cross over $\Omega$ which does not contain a single edge drawn in the interior of $c,$ contradicting the fact that $\rho$ is a rendition of $(G,\Omega)$ in the disk. 
Hence, $I_i\cap I_j=\emptyset$ follows.
\end{proof}

The next step would be to show that, given some integer $t,$ we can either find $t$ crosses as in the lemma above or completely remove the vortex by deleting a number of vertices bound by a function of $t$ and $\theta.$
That is, we want to show that in the situation where we cannot find $t$ pairwise disjoint patches, each hosting a cross, we find another set of vertices of bounded size whose deletion allows us to ``flatten'' all patches.
If this case does not occur and we find many patches, each hosting a cross, we would like to ``project'' these crosses onto the original society $(G,\Omega)$ while maintaining vertex-disjointness.
The following lemma is the key tool towards the structural results of this paper.

\begin{lemma}\label{@regulating} 
Let $t\leq \theta$ be positive integers. 
Let $(\rho=(\gamma,\mathcal{D},c),G,\Omega)$ be a \hyperref[@maidservants]{$\theta$-tight} rendition of a nest $\mathcal{C}$ around a vortex $c$ of depth at most $\theta$ in a disk $ \Delta$ in $\rho.$ 
Then either 
\begin{enumerate}
\item there exists a separation $(A,B)$ of order at most $6\theta(t-1)$ such that $(\InducedSubgraph{G}{A\setminus B},\Omega^*)$ has a vortex-free rendition in the disk, where $\Omega^*$ is the restriction of $\Omega$ to $A\setminus B$ and $V(\Omega)\cap B\subseteq A\cap B,$ or 
\item there exists a $\theta$-suspension $(\theta,\rho,G,\Omega,\mathcal{C},\Delta,\Delta',\mathcal{P})$ of $c$ in $(\rho,G,\Omega)$ and a consecutive family $\Set{\mathcal{Q}_1,\dots,\mathcal{Q}_2}$ of $t$ crosses over $(G,\Omega)$ together with $t$ pairwise disjoint segments $S_1,\dots,S_t$ of $\Omega'$ and a set $Z'\subseteq V(G')$ such that for each $i\in[t]$ the cross $\mathcal{Q}_i$ intersects $G'$ exactly in a cross of $\Omega'$ at $Z'.$
\end{enumerate}
\end{lemma}

\begin{proof} 
As $(\rho,G,\Omega)$ is \hyperref[@maidservants]{$\theta$-tight}, one of two cases may hold. 
Let us first assume that case \ref{@miraculous} holds. 
Then there exists a set $Z\subseteq V(G)$ with $\Abs{Z}\leq \theta$ such that $Z$ separates $V(\Omega)$ from $\widetilde{c}.$ 
Let $A$ be the collection of all vertex sets of components of $G-Z$ that contain a vertex of $V(\Omega)$ and let $B\coloneqq  V(G)\setminus (A\cup Z).$ 
Moreover, let $\Omega^*$ be the restriction of $\Omega$ to $A.$ 
Then the restriction of $\rho$ to $\InducedSubgraph{G}{A}$ is a vortex-free rendition of $(\InducedSubgraph{G}{A},\Omega^*)$ in a disk and $(A\cup Z, Z\cup B)$ is a separation of order at most $\theta\leq 6\theta(t-1).$ 
Hence, in this case we are done. 
 
So from now on we may assume that case \ref{@idealities} from the definition of \hyperref[def_tightnest]{$\theta$-tightness} holds. 
Hence, we may assume that there exists a disk $ \Delta'\subseteq \Delta$ whose boundary intersects $\gamma$ only in vertices with $X=\Boundary{ \Delta'}\cap N(\rho)$ such that 
\begin{itemize} 
\item there exists a family $\mathcal{P}$ of pairwise $V(\Omega)$-$\widetilde{c}$-paths with $\Abs{\mathcal{P}}=\Abs{\pi(X)}\geq\theta,$ and 
\item the society $(G'\coloneqq  \InducedSubgraph{G}{V(G)\cap V( \sigma( \Delta'))},\Omega')$ has depth at most $3\theta$ where $\Omega'$ is an ordering of $\pi(X)$ induced by $ \Delta'$ 
\end{itemize}
as in the definition of $\theta$-suspensions.
 
In what follows we will iteratively construct segments $J_i$ of $\Omega'$ together with sets $Z_i\subseteq V(G')$ and $S_i\subseteq V(G')$ such that 
\begin{itemize} 
\item for all $j\in[i]$ we have $\Abs{Z_j},\Abs{S_j}\leq 3\theta,$ 
\item for each $j\in[i-1]$ the last vertex of $J_j$ comes before the first vertex of $J_{j+1}$ and $J_1\cap J_i=\emptyset,$ 
\item for all $j\in[i]$ the set $Z^j\coloneqq  \bigcup_{h\in[j]}Z_h$ separates $J_j$ from $V(\Omega')\setminus J_j$ in $G',$ 
\item for all $j\in[i]$ the society from $(G',\Omega')$ cut by $Z^i$ and \hyperref[def_patch]{patched} at $J_j$ has a cross, 
and 
\item if $J^j$ is the segment of $\Omega'$ whose first vertex is the first vertex of $J_1$ and whose last vertex is the last vertex is the last vertex of $J_j,$ while $S^j\coloneqq \bigcup_{h\in[j]}S_h,$ then for all $j\in[i]$ the society from $(G',\Omega')$ cut by $Z^j\cup S^j$ and \hyperref[def_patch]{patched} at $J^j$ has a vortex-free rendition in the disk. 
\end{itemize}
In case this iterative process stops before we complete the step $i=t$ we will find the required separation. 
Otherwise, an application of \autoref{@swineherds} yields the required $t$ consecutive crosses. 
 
Let $λ$ be a linear ordering of $V(\Omega')$ obtained from $\Omega'$ and let us denote by $x_1$ the smallest vertex with respect to $λ.$ 
If $λ'$ is any restriction of $λ$ to some vertex set $U$ we write $V(λ')$ for the set $U.$ 
For any restriction $λ'$ of $λ$ and any property defined for a vertex we say that a vertex $u$ is the \emph{smallest} with this property in $λ'$ if no vertex $v\in V(λ ')$ with $λ'(v)<λ'(u)$ has the property, but $u$ does. 
For any two vertices $u,v\in V(λ)$ we define $J_{u,v}\coloneqq  \CondSet{w\in V(λ)}{λ(u)\leqλ(w)\leqλ(v)}.$ 
Note that $J_{u,v}$ is empty if $λ(v)<λ(u).$ 
Moreover, $J_{u,v}$ always defines a segment of $\Omega'.$ 
 
Since $(G',\Omega')$ has depth at most $3\theta,$ for any segment $I$ of $\Omega'$ there exists a set $Y$ of size at most $3\theta$ separating $I$ from $V(\Omega')\setminus I$ in $G'.$ 
For any pair $u,v\in V(\Omega')$ and an already given set $Z^h,$ where $h\in\mathbb{Z},$ we denote by $Y^h_{u,v}$ a set of order at most $3\theta$ such that $Z^h\cup Y_{u,v}$ separates $J_{u,v}$ from $V(\Omega')\setminus J_{u,v}$ and $Y^h_{u,v}$ only contains vertices of $G'-Z^h$ that belong to a component of $G'-Z^h$ which contains a vertex of $J_{u,v}.$ 
 
Let us fix $Z^0\coloneqq \emptyset.$ 
Now let $z_1$ be the smallest vertex of $V(λ)$ such that the society from $(G',\Omega')$ cut by $Y^0_{x_1,z_1}$ and \hyperref[def_patch]{patched} at $J_{x_1,z_1}$ has a cross and let $y_1$ be the immediate predecessor of $z_1$ with respect to $λ.$ 
Note that in case $z_1$ does not exist $(G',\Omega')$ has a vortex-free rendition in the disk and thus so does $(G,\Omega)$ which would complete our proof with the separation $(V(G),\emptyset).$ 
Let us now set $J_1\coloneqq  J_{x_1,z_1},$ $Z_1\coloneqq  Y^0_{x_1,z_1},$ and $S_1\coloneqq  Y^0_{x_1,y_1}.$ 
It follows from the minimality of $z_1$ that the society from $(G',\Omega')$ cut by $Z_1\cup S_1$ and \hyperref[def_patch]{patched} at $J_1$ has a vortex-free rendition in the disk. 
Hence, the first iteration is complete. 
 
Now assume for some $\ell\leq t-1$ the vertices $x_i,y_i,z_i$ together with the segments $J_{x_i,z_i}$ and the sets $Z_i=Y^{i-1}_{x_i,z_i}$ and $S_i=Y^{i-1}_{x_i,y_i}$ have already been constructed for all $i\in[\ell]$ meeting the requirements from above. 
Let $\Complement{J}$ be the segment of $\Omega'$ obtained by deleting the vertices of the segment $J^{\ell}.$ 
 
Suppose the society from $(G',\Omega')$ cut by $Z^{\ell}\cup S^{\ell}$ and \hyperref[def_patch]{patched} at $\Complement{J}$ has no cross. 
Let $A$ be the collection of all vertices of $G$ that are contained in a component of $G-(Z^{\ell}\cup S^{\ell})$ that contains a vertex of $V(\Omega),$ let $\Omega^*$ be the restriction of $\Omega$ to $A,$ and let $B\coloneqq  V(G)\setminus(A\cup Z^{\ell}\cup S^{\ell}).$ 
As $\ell\leq t-1$ and $\Abs{Z_i},\Abs{S_i}\leq 3\theta$ for all $i\in[\ell]$ we have $\Abs{Z^{\ell}\cup S^{\ell}}\leq 6\theta(t-1).$ 
Hence, $(A\cup Z^{\ell}\cup S^{\ell},Z^{\ell}\cup S^{\ell}\cup B)$ is a separation of order at most $6\theta(t-1)$ and $(\InducedSubgraph{G}{A},\Omega^*)$ has a vortex-free rendition in the disk.
Thus, we are done with this case. 
 
\begin{figure}[h] 
\begin{center} 
~~~~~~~~~~~~~\scalebox{1}{\includegraphics{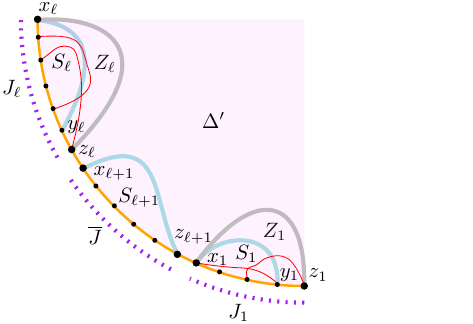}}
\end{center} 
\caption{A visualization of \textbf{Rule 2} in the proof of \autoref{@regulating}. The fat grey and \textcolor{celestialblue}{light blue} lines represent the separators $Z_{i}$ and $S_{i},$ and the \darkmagenta{magenta} dotted lines are the corresponding segments $J_{i}.$} 
\label{@neoliberal} 
\end{figure} 
 
From now on we may assume that the society from $(G',\Omega')$ cut by $Z^{\ell}\cup S^{\ell}$ and \hyperref[def_patch]{patched} at $\Complement{J}$ has a cross. 
 
Let $x_{\ell+1}$ be the immediate successor of $z_{\ell}$ with respect to $λ$ and let $z_{\ell+1}$ be chosen according to one of the following rules: 
\begin{description} 
\item[Rule 1] Let $z_{\ell+1}$ be the smallest vertex of $V(λ)$ with $λ(x_{\ell+1})\leqλ(z_{\ell+1})$ such that the society from $(G',\Omega')$ cut by $Y^{\ell}_{x_{\ell+1},z_{\ell+1}}\cup Z^{\ell}$ and \hyperref[def_patch]{patched} at $J_{x_{\ell+1},z_{\ell+1}}$ has a cross. 
 
\item[Rule 2] If a choice according to \textbf{Rule 1} is not possible select $z_{\ell+1}$ such that $J_{x_{\ell+1},z_{\ell+1}}=\Complement{J}$ and the society from $(G',\Omega')$ cut by $Z^{\ell}\cup S^{\ell}$ and \hyperref[def_patch]{patched} at $\Complement{J}$ has a cross.  
\end{description} 
By our assumption, we may consecutively apply \textbf{Rule 1} and, when this is not any more possible, either the procedure stops, which means that the society from $(G',\Omega')$ cut by $Z^{\ell}\cup S^{\ell}$ and \hyperref[def_patch]{patched} at $\Complement{J}$ has a vortex free rendition in the disk, or it stops after applying \textbf{Rule 2} once. In the later case, 
the society from $(G',\Omega')$ cut by $(Z^{\ell}\cup S^{\ell})\cup Y^{\ell}_{x_{\ell+1},z_{\ell+1}}$ and \hyperref[def_patch]{patched} at $V(\Omega')$ has a vortex free rendition in the disk (see \autoref{@neoliberal}).  
 
Either way let $y_{\ell+1}$ be the immediate predecessor of $z_{\ell+1}$ under $λ.$ 
 
In case $z_{\ell}$ was chosen by \textbf{Rule 2} we set $Z_{\ell+1}\coloneqq  \emptyset$ and $S_{\ell+1}\coloneqq  Y^{\ell}_{x_{\ell+1},z_{\ell+1}}.$ 
Note that in this case $J_{\ell+1}=V(λ).$ 
Hence, we satisfy all five rules of our iterative process simply by assumption and may continue with step $\ell+2$ or we are in the case $\ell+1=t$ which will be treated later. 
 
So we may further assume that $z_{\ell+1}$ was chosen according to \textbf{Rule 1}. 
Now we set $J_{\ell+1}\coloneqq  J_{x_{\ell+1},z_{\ell+1}},$ $Z_{\ell+1}\coloneqq  Y^{\ell}_{x_{\ell+1},z_{\ell+1}},$ and $S_{\ell+1}\coloneqq  Y^{\ell}_{x_{\ell+1},y_{\ell+1}}.$ 
We then have, by the discussion above, $\Abs{Z_{\ell+1}},\Abs{S_{\ell+1}}\leq 3\theta$ and $J_{\ell+1}$ is disjoint from all other segments $J_i,$ $i\in[\ell].$ 
Moreover, $Z^{\ell+1}$ separates $J_{\ell+1}$ from $V(\Omega')\setminus J_{\ell+1}$ within $G'$ by choice of $Z_{\ell+1}.$ 
By \textbf{Rule 1} the society from $(G',\Omega')$ cut by $Z^{\ell+1}$ and \hyperref[def_patch]{patched} at $J_{x_{\ell+1},z_{\ell+1}}$ has a cross and finally, by the minimality of $z_{\ell+1}$ we have that the society from $(G',\Omega')$ cut by $Z^{\ell+1}\cup S^{\ell+1}$ and \hyperref[def_patch]{patched} at $J^{\ell+1}$ has a vortex-free rendition in the disk. 
Thus, all five requirements for our iteration are met and we may continue or have entered the case where $\ell+1=t.$ 
 
To complete the proof let us discuss what happens in the case where we have successfully completed iteration round $t$ without finding the separation. 
We need to show that this implies the existence of $t$ consecutive crosses over $(G,\Omega).$ 
Since we have successfully completed round $t$ we have found segments $J_1,\dots,J_t$ such that the first vertex of $J_{i+1}$ comes before the last vertex of $J_i$ for all $i\in[t-1]$ and $J_1\cap J_t=\emptyset.$ 
For each $i\in[t]$ let $Y_i$ be the \hyperref[def_patch]{patch} of $(G',\Omega')$ cut at $J_i$ by $Z^t$ and let $(G_i,\Omega_i)$ be the restriction of $G$ and $\Omega$ to the graph $G-(V(G')\setminus Y_i).$ 
Since the society from $(G',\Omega')$ cut by $Z^t$ at $J_i$ has a cross, the society $(G_i,\Omega_i)$ cannot have a vortex-free rendition in the disk and thus, by \autoref{@postscripts}, it must have a cross. 
Therefore, we may apply \autoref{@swineherds} to obtain our $t$ consecutive crosses.
\end{proof}

So far we have seen that, in the second case of \hyperref[@maidservants]{$\theta$-tightness}, we can find disjoint crosses over our society that meet all cycles in the nest or we can completely eliminate the vortex.
To construct a shallow vortex minor however, we need the paths of these crosses to be particularly well behaved with a large portion of the nest.

Our goal is to obtain a still large number of concentric cycles which behaves in an ``orthogonal'' fashion with respect to the $t$-consecutive crosses \autoref{@regulating} provides us with in case we cannot remove a vortex with a small set of vertices.
To realise this plan we will take several steps.
Let $(\rho,G,\Omega)$ be a rendition of a nest $\mathcal{C}$ around a vortex $c$ and let $(P_1,P_2)$ be a cross over $(G,\Omega).$
Clearly each of the two paths $P_i$ needs to contain an edge that is drawn in the interior of $c.$
However, we can not necessarily guarantee a bound on the number of components in $(P_1\cup P_2)-\widetilde{c}.$
That is, the paths $P_i$ may leave and re-enter the vortex $c$ many times before eventually making progress towards $V(\Omega).$
This poses an issue as any $\widetilde{c}$ subpath of one of the $P_i$ whose interior is disjoint from the interior of $c$ cannot necessarily be used for a re-routing of the cycles in $\mathcal{C}$ that eventually achieves ``orthogonality''.
To cope with this issue we proceed to prove that there exists some constant $\mathsf{c}$ such that any cross $(P_1,P_2)$ can be rerouted to a cross $(P_1',P_2')$ with the same endpoints as $(P_1,P_2)$ but every $\widetilde{c}$-subpath of a path $P_i'$ is disjoint from all cycles $C_j\in\mathcal{C}$ with $j>\mathsf{c}.$ For instance this property does not hold for the crossing paths in the left-side of \autoref{@recollection} while it holds for the right-side one.
To achieve this we need some further definitions.
 
\begin{figure}[h] 
\hspace{-2.6cm}\scalebox{.89}{\input{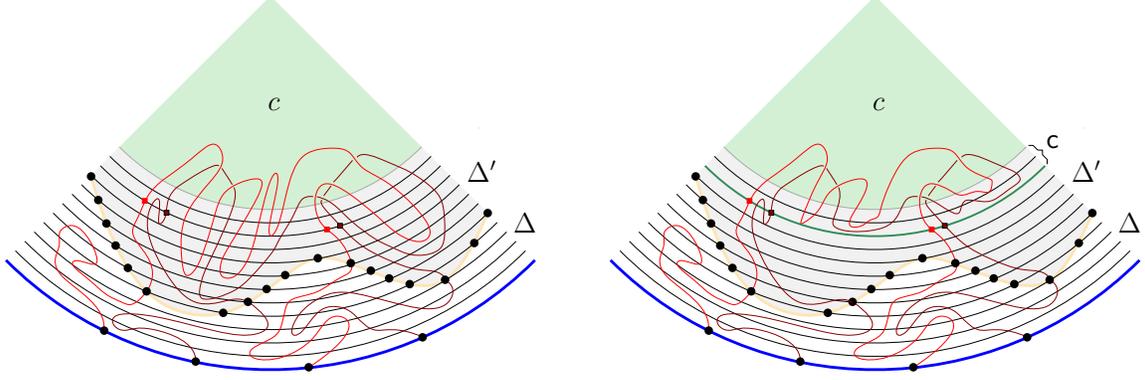}}
\caption{Two ways crossing paths $P_{1},P_{2}$ (in red and dark red) may be rooted inside $ \Delta'.$ Notice that in the picture on the right every $\widetilde{c}$-subpath of a path $P_i$ is disjoint from all cycles $C_j\in\mathcal{C}$ with $j>\mathsf{c}$ while this is not the case in the picture on the left.} 
\label{@recollection} 
\end{figure}
 
In the following we will identify a linkage $\mathcal{P}$ with the graph $\bigcup_{P\in\mathcal{P}}P.$
Let $G$ be a graph and let $\mathcal{P}$ be a linkage in $G.$
The set $\CondSet{\Set{s,t}}{\text{some path in }\mathcal{P}\text{ has endpoints }s\text{ and }t}$ is called the \emph{pattern} of $\mathcal{P}.$
Two linkages $\mathcal{Q}_1,$ $\mathcal{Q}_2$ are said to be \emph{equivalent} if they have the same pattern.
A linkage $\mathcal{Q}$ is said to be \emph{vital} if $V(\mathcal{Q})=V(G)$ and there exists no other linkage in $G$ equivalent to $\mathcal{Q}.$
The famous \emph{unique linkage theorem} of Robertson and Seymour \cite{robertson2009graph,kawarabayashi2010shorter} states that, if $\mathcal{P}$ is vital in a graph $G,$ then the treewidth of $G$ is bounded in some function of $\Abs{\mathcal{P}}.$
Instead of the unique linkage theorem however, we make use of a slightly more convenient version from \cite{golovach2020hitting}.

\begin{definition}[LB-Pair]\label{@impression} 
Given a graph $G,$ a LB-pair of $G$ is a pair $(\mathcal{L},B)$ where $B$ is a subgraph of $G$ with maximum degree two and $\mathcal{L}$ is a linkage in $G.$ 
We define $\Disagreement{\mathcal{L},B}\coloneqq  \Abs{E(\mathcal{L})\setminus E(B)}$ to be the number of edges from $\mathcal{L}$ on which $\mathcal{L}$ and $B$ disagree.
\end{definition}

\begin{proposition}[\!\!\cite{golovach2020hitting}]\label{@liquidating} 
There exists a function $\UniqueLinkage\colon\Bbb{N}\rightarrow\Bbb{N}$ such that for all graphs $G$ and all LB-pairs $(\mathcal{L},B)$ of $G,$ if $\tw{\mathcal{L}\cup B}>\UniqueLinkage(\Abs{\mathcal{L}})$ then $G$ contains a linkage $\mathcal{R}$ such that 
\begin{enumerate}[label=\textit{\roman*})] 
\item $\Disagreement{\mathcal{R},B}<\Disagreement{\mathcal{L},B},$ 
 
\item $\mathcal{L}$ and $\mathcal{R}$ are equivalent, and 
 
\item $\mathcal{R}\subseteq \mathcal{L}\cup B.$ 
\end{enumerate}
\end{proposition}

Given an LB-pair $(\mathcal{L},B)$ of a graph $G$ we say that $(\mathcal{L},B)$ is of \emph{maximum consensus} if for all families of pairwise disjoint paths $\mathcal{R}$ which are equivalent to $\mathcal{L}$ we have $\Disagreement{\mathcal{L},B}\leq \Disagreement{\mathcal{R},B}.$ 
It follows from \autoref{@liquidating} that $\tw{\mathcal{L}\cup B}\leq \UniqueLinkage(\Abs{\mathcal{L}})$ for all LB-pairs $(\mathcal{L},B)$ of maximum consensus.

\begin{lemma}\label{@bankruptcy}
Let $\theta$ be a positive integer.
Let $(\rho=(\gamma,\mathcal{D},c),G,\Omega)$ be a \hyperref[@maidservants]{$\theta$-tight} rendition of a nest $\mathcal{C},$ with $\Abs{\mathcal{C}}\geq 2\cdot \UniqueLinkage(2)+4,$ around a vortex $c$ of depth at most $\theta$ in a disk $ \Delta$ in $\rho.$
Moreover, let $(\theta,\rho,G,\Omega,\mathcal{C},\Delta,\Delta',\mathcal{P})$ be a $\theta$-suspension of $c$ in $(\rho,G,\Omega)$
Let $Z\subseteq V(G')$ be a set of vertices and let $S\subseteq V(\Omega')$ be a segment of $\Omega'$ such that $S$ has a cross $(L,R)$ at $Z.$
Then there exists a pair of paths $(L',R')$ such that
\begin{enumerate}[label=\textit{\roman*})] 
\item $(L',R')$ is a cross of $S$ at the set $Z,$ 
\item $L'$ and $L$ have the same endpoints, 
\item $R'$ and $R$ have the same endpoints, and 
\item any subpath of $L'$ or $R'$ with both endpoints in $\widetilde{c}$ intersects at most $\mathsf{c}\coloneqq  2\UniqueLinkage(2)+2$ many cycles of $\mathcal{C}.$
\end{enumerate}
\end{lemma}

Towards a proof for \autoref{@bankruptcy} we need a way to certify large treewidth.
To do this we make use of the well known concept of brambles.

Let $G$ be a graph and $H_1,$ $H_2$ be two connected subgraphs of $G.$
We say that $H_1$ and $H_2$ \emph{touch} if $\V{H_1}\cap\V{H_2}\neq\emptyset,$ or there is an edge $uv$ with $u\in\V{H_1}$ and $v\in\V{H_2}.$
A set $S\subseteq\V{G}$ is a \emph{hitting set} or \emph{cover} for a family $\mathcal{S}$ of subgraphs of $G,$ if $\V{H}\cap S\neq\emptyset$ for all $H\in\mathcal{S}.$

\begin{definition}[Bramble] 
Let $G$ be a graph. 
A \emph{bramble} $\mathcal{B}=\Set{B_1,B_2,\dots,B_{\ell}}$ of $G$ is a family of connected and pairwise touching subgraphs $B_i$ of $G.$ 
The \emph{order} of $\mathcal{B}$ is the size of a minimum hitting set for $B.$
\end{definition}

\begin{proposition}[\!\!\cite{seymour1993graph}]\label{@nondeducible} 
Let $G$ be a graph, and $k\in\Bbb{N}$ a positive integer. 
Then $G$ contains a bramble of order $k$ if and only if $\tw{G}\geq k-1.$
\end{proposition}

\begin{proof}[Proof of \autoref{@bankruptcy}]
Let $(\hat{G},\hat{\Omega})$ be the society from $(G',\Omega')$ cut by $Z$ and \hyperref[def_patch]{patched} at $S.$
Let $U$ be some subpath of $L\cup R$ with both endpoints in $\widetilde{c}$ that is internally disjoint from $\widetilde{c}$ and let $U'$ be its trace.
We call such a subpath of $L\cup R$ an \emph{arc} of $(L,R).$
Note that $U'$ together with the boundary of $c$ divides $ \Delta'$ into three different areas, the disk bounded by $c$ and two disks which are subsets of $ \Delta'-(c-\Boundary{c}),$ exactly one of these disks contains all vertices of $L$ and $R$ that lie on the boundary of $ \Delta'.$
Let $ \Delta_U$ be the remaining disk.
Observe that every maximal subpath of $L\cup R$ which is drawn in $ \Delta_U$ must have both of its endpoints in $\widetilde{c}.$
We call $ \Delta_U$ a \emph{mountain} if there does not exist another arc $W$ of $(L,R)$ such that $ \Delta_U\subset  \Delta_W,$ moreover, if $ \Delta_U$ is a mountain, we call $U$ its \emph{outline}.

Now every subpath of $L\cup R$ that intersects some cycle from $\mathcal{C}$ and has both endpoints in $\widetilde{c}$ must be completely contained in $\widetilde{c}$ together with the union of all arcs of $(L,R).$
Hence, every arc of $(L,R)$ must either be an outline of some mountain or drawn in the interior of some mountain.
Moreover, if $U$ and $W$ are distinct outlines, then their corresponding mountains are disjoint with the possible exception for the endpoints of $U$ and $W.$
For every arc $W$ of $(L,R)$ let $z_W$ be the maximum number of cycles from $\mathcal{C}$ met by $W.$
Let $n=\Abs{\mathcal{C}}.$
Given a mountain $ \Delta_U,$ we may associate a vector $\mathsf{v}(U)\in\mathbb{N}^n$ such that for all $i\in[n],$ $$ \mathsf{v}(U)_{n-i+1}=\Abs{\CondSet{W}{W\text{ is an arc drawn in } \Delta_U\text{ with }z_W=i}}.$$
We call $\mathsf{v}(U)$ the  \emph{characteristic vector} of the mountain $ \Delta_U.$
Note that $\mathsf{v}(U)_i$ is the number of arcs drawn in $ \Delta_U$ that meet the cycle $C_{n-i+1}.$
We now proceed to show that we can find a cross $(L',R')$ with properties \textit{i., ii.,} and \textit{iii.}, such that no mountain of $(L',R')$ contains a vertex of $\mathsf{c}+1$ cycles.

Let $(L',R')$ be chosen to first minimize the number of mountains, second minimize the number of edges on which the arcs within a mountain disagree with the cycles from $\mathcal{C},$ and thirdly to lexicographically minimize the characteristic vectors of all mountains.
Suppose there exists an outline $U$ such that the corresponding mountain $ \Delta_U$ meets at least $h\geq \mathsf{c}+1$ cycles.
For each $i\in[h]$ let $B_i$ be the subpath of $C_i$ which is completely drawn into the disk $ \Delta_U,$ and let $B$ be the union of all of these subpaths.
Let $\mathcal{U}$ be the union of all arcs of $(L,R)$ which are drawn in $ \Delta_U.$
Note that the trace of any arc $W\in\mathcal{U},$ distinct from $U,$ separates $ \Delta_U$ into two disks, one of them containing all vertices of $U,$ let us call this disk the \emph{upper part of $W$}, the other disk is called the \emph{lower part of $W$}.
Moreover, every other arc is completely contained in one of these two disks.
Let us assume $\mathcal{U}=\Set{U_1,\dots,U_{\ell}}$ is numbered such that $U_1=U$ and for all $i\in[2,\ell]$ the upper part of $U_i$ contains the arc $U_{i-1}.$
Please note that this ordering is not uniquely determined.

In the following we will define a bramble as follows.
Let $K_1\coloneqq  B_h\cup U_1$ and let $x_1, y_1$ be the two endpoints of $U_1.$
Now suppose $U_2$ does not meet $B_{h-1}.$
In this case let $P_1$ be the shortest $x_1$-$B_{h-1}$-subpath of $U_1$ and let $P_2$ be the shortest $B_{h-1}$-$y_1$-subpath of $U_1.$
Finally, let $P_3$ be the shortest subpath of $B_{h-1}$ that joins the two endpoints of $P_1$ and $P_2.$
We may now exchange the path $U_1$ in $L'\cup R'$ with the path $P_1\cup P_3\cup P_2.$
This yields an immediate contradiction to the choice of $(L',R')$ as this exchange yields a new pair of paths whose mountains are the same as the mountains of $(L',R')$ with the sole exception of $U,$ here the new pair has a mountain whose characteristic vector is lexicographically smaller than the one of $U$.
Moreover, we claim that the numbering of the $U_i$ can be chosen such that for any $i\in[2,h]$ the path $U_i$ meets $B_{h-i+1}.$
To see this, suppose there is some $i$ such that for all $j<i$ the choice was possible, but for $i$ itself it is not.
This means that for all $j\in[2,i-1]$ the arc $U_j$ meets $B_{h-i+2}$ and every arc in the lower part of $U_j$ avoids $B_{h-i+1}.$ 
Hence, we may use $B_{h-i+1}$ to obtain a cross with the same number of mountains as $(L',R'),$ but which contradicts the lexicographic minimality of the characteristic vectors as before.
Hence, we may define $K_i\coloneqq  B_{h-i+1}\cup U_i$ for all $i\in h.$

Now observe that for any $j\in[\ell]$ and $i\in[h],$ if $U_j$ meets $B_i,$ then $U_j$ intersects all $B_{i'}$ with $i'\leq i.$
Hence, $K_i$ and $K_j$ intersect for all choices of $i,j\in[h].$
However, since the $B_i$ are pairwise disjoint and the $U_j$ are pairwise disjoint, no vertex can be contained in 
more than two of the $K_i$ at once.
Hence, any cover of $\mathcal{K}\coloneqq \CondSet{K_i}{i\in[h]}$ must have size at least $\Ceil{\frac{h}{2}}\geq \UniqueLinkage(2)+2,$ making it a bramble of order at least $\UniqueLinkage(2)+1.$
Hence, by \autoref{@nondeducible} the treewidth of $\CondSet{U_i}{i\in[h]}\cup B$ is at least $\UniqueLinkage(2)+1.$
By \autoref{@liquidating} this is a contradiction to the second minimization parameter we used for the choice of $(L',R')$ and thus our proof is complete.
\end{proof}

Towards extracting the structure of a shallow vortex grid, we need that all crossing paths implied by \autoref{@regulating} traverse  the cycles of the nest ``orthogonally''.
This was formalized in \cite{kawarabayashiTW21quickly} with the following definition.

\begin{definition}[Orthogonal Linkage and Crosses]\label{@nihilistic} 
Let $(\rho,G,\Omega)$ be a rendition of a nest $\mathcal{C}$ around a vortex $c$ in a disk $ \Delta$ in $\rho$ and let $\mathcal{P}$ be a $\V{\Omega}$-$\widetilde{c}$ linkage. 
We say that $\mathcal{P}$ is \emph{orthogonal} to $\mathcal{C}$ if for every $C\in\mathcal{C}$ and every $P\in\mathcal{P}$ the graph $P\cap C$ is a, possibly trivial, path.
 
Let $\mathcal{U}$ be a collection of crosses over $(G,\Omega)$ and let $\mathcal{Q}$ be the collection of all $V(\Omega)$-$\widetilde{c}$-subpaths of the paths of the crosses in $\mathcal{C}.$ 
We say that $\mathcal{U}$ is \emph{orthogonal} to $\mathcal{C}$ if $\mathcal{Q}$ is orthogonal to $\mathcal{C}.$
\end{definition}

Please notice that the demand for the intersection of some path $P$ from $\mathcal{P}$ and some cycle $C$ from $\mathcal{C}$ cannot be further restricted to be a single vertex since this definition needs to apply, in particular, to graphs with maximum degree three.

\begin{lemma}\label{@suppressed}
Let $t\leq \theta$ be positive integers.
Let $(\rho=(\gamma,\mathcal{D},c),G,\Omega)$ be a \hyperref[@maidservants]{$\theta$-tight} rendition of a nest $\mathcal{C},$ with $\Abs{\mathcal{C}}\geq 8t^2+2\UniqueLinkage(2)+4,$ around a vortex $c$ of depth at most $\theta$ in a disk $ \Delta$ in $\rho.$
Moreover, assume that $(\theta,\rho,G,\Omega,\mathcal{C},\Delta,\Delta',\mathcal{P})$ is $\theta$-suspension of $c$ in $(\rho,G,\Omega)$ and that there is a consecutive family $\mathcal{Q}\coloneqq \Set{\mathcal{Q}_1,\dots,\mathcal{Q}_t}$ of $t$ crosses over $(G,\Omega)$ together with $t$ pairwise disjoint segments $S_1,\dots,S_t$ of $\Omega''$ and a set $Z''\subseteq V(G'')$ such that for each $i\in[t]$ the cross $\mathcal{Q}_i$ intersects $G''$ exactly in a cross of $\Omega''$ at $Z''.$

Then there exists a family $\mathcal{C}'$ of $t$ pairwise vertex disjoint cycles whose traces all separate $c$ from $V(\Omega)$ and such that $\mathcal{Q}$ is \hyperref[def_orthogonal]{orthogonal} to $\mathcal{C}'.$
\end{lemma}

\begin{proof}
Let $\mathcal{C}=\Set{C_1,\dots,C_{\ell}}.$
By \autoref{@bankruptcy} we may assume that no arc of $\mathcal{Q}_i$ meets the cycle $C_{2\UniqueLinkage(2)+3}$ for any $i\in[t].$
For every $i\in[t]$ let $\mathcal{Q}_i=(L_i,R_i),$ and let $s_i^L,s_i^R,t_i^L,t_i^R$ be the four endpoints of the paths $L_i$ and $R_i$ respectively in the order they appear on $\Omega.$
Moreover, for every $i\in[t]$ and every $W\in\Set{L,R}$ let $x_i^W$ be the first vertex of $\widetilde{c}$ encountered when traversing $W_i$ starting in $s_i^W,$ and let $y_i^W$ be the first vertex of $\widetilde{c}$ encountered when traversing $W_i$ starting in $t_i^W.$
Finally, let $\mathcal{P}_i\coloneqq  \Set{s_i^LL_ix_i^L,y_i^LL_it_i^L,s_i^RR_ix_i^R,y_i^RR_it_i^R},$ $\mathcal{P}\coloneqq  \bigcup_{i=1}^t\mathcal{P}_i,$ $T=\bigcup_{i\in[t]}\Set{x_i^L,y_i^L,x_i^R,y_i^R},$ $\widetilde{\Omega}$ be the cyclic ordering of the vertices in $T$ obtained by traversing the boundary of $c$ in clockwise direction, let $Q\coloneqq  \Set{C_{2\UniqueLinkage(2)+5},\dots,C_{2\UniqueLinkage(2)+4+8t^2}}\cup\mathcal{P},$ and $(Q,\widetilde{\Omega})$ be the society defined by $Q$ and $\widetilde{\Omega}.$
Note that it follows from our assumptions that $(Q,\widetilde{\Omega})$ has a vortex-free rendition in the disk.

For better readability let us rename $O_i\coloneqq  C_{2\UniqueLinkage(2)+4+i}$ for every $i\in[4t^2].$
Note that for each $i\in[t]$ the vertices $x_i^L,x_i^R,y_i^L,y_i^R$ appear on $\widetilde{\Omega}$ in the order listed.
Moreover, the vertices of $V(\Omega)$ which are present in $Q$ appear in the rendition of $(Q,\widetilde{\Omega})$ in the same order as they appear on $\Omega.$
Now, for every $i\in[t-1]$ and $j\in[t]$
\begin{itemize} 
\item let $B_{i,j,L,1}$ be a shortest $s_i^LL_ix_i^L$-$s_i^RR_ix_i^R$ subpath of the cycle $O_{t(j-1)+4(i-1)+1},$ 
\item let $B_{i,j,R,1}$ be a shortest $s_i^RR_ix_i^R$-$t_i^LL_iy_i^L$ subpath of the cycle $O_{t(j-1)+4(i-1)+2},$ 
\item let $B_{i,j,L,2}$ be a shortest $t_i^LL_iy_i^L$-$t_i^RR_iy_i^R$ subpath of the cycle $O_{t(j-1)+4(i-1)+3},$ and 
\item let $B_{i,j,R,2}$ be a shortest $t_i^RR_iy_i^R$-$s_{i+1}^LL_{i+1}y_{i+1}^L$ subpath of the cycle $O_{t(j-1)+4(i-1)+4}.$
\end{itemize}
Moreover,
\begin{itemize} 
\item let $B_{t,j,L,1}$ be a shortest $s_t^LL_tx_t^L$-$s_t^RR_tx_t^R$ subpath of the cycle $O_{t(j-1)+4(t-1)+1},$ 
\item let $B_{t,j,R,1}$ be a shortest $s_t^RR_tx_t^R$-$t_t^LL_ty_t^L$ subpath of the cycle $O_{t(j-1)+4(t-1)+2},$ 
\item let $B_{t,j,L,2}$ be a shortest $t_t^LL_ty_t^L$-$t_t^RR_ty_t^R$ subpath of the cycle $O_{t(j-1)+4(t-1)+3},$ and 
\item let $B_{t,j,R,2}$ be a shortest $t_t^RR_ty_t^R$-$s_1^LL_1y_1^L$ subpath of the cycle $O_{t(j-1)+4t}.$
\end{itemize}
Note that the paths of the form $B_{i,j,W,h}$ are pairwise internally disjoint and for each $j\in[t],$ every path of $P\in \mathcal{P}$ meets exactly two paths of the form $B_{i,j,W,h}.$
Indeed, both of these paths are met exactly in their endpoints.
These endpoints can be joined by a unique subpath of $P$ each, resulting in total in $2t$ pairwise disjoint cycles $D_1,\dots,D_{t}.$
Note that it immediately follows that the trace of each $D_i$ separates $c$ from $V(\Omega)$ in $\rho.$

By the choice of the $O_i$ and our assumption from before it follows that each $B_{i,j,W,h}$ is internally disjoint from all paths from the crosses of $\mathcal{Q}$ and thus $\mathcal{Q}$ is \hyperref[def_orthogonal]{orthogonal} to $\Set{D_1,\dots,D_{2t}}.$
\end{proof}

We are finally ready to prove the main result of this section.

\begin{proof}[Proof of \autoref{@constitution}] 
Let $\mathsf{c}\coloneqq  2\UniqueLinkage(2)+4.$ 
We start by applying \autoref{@regulating}. 
This either yields the desired separation or we find a $\theta$-suspension $(\theta,\rho,G,\Omega,\mathcal{C},\Delta,\Delta',\mathcal{P})$ of $c$ in $(\rho,G,\Omega)$.
Moreover, if $(G''\coloneqq  \InducedSubgraph{G}{V(G)\cap  \Delta'},\Omega'')$ is the society defined by $\Delta'$ in the definition of $\theta$-suspensions, then there exists a consecutive family $\mathcal{Q}=\Set{\mathcal{Q}_1,\dots,\mathcal{Q}_{2t}}$ of $2t$ crosses over $(G,\Omega)$ together with $2t$ pairwise disjoint segments $S_1,\dots,S_{2t}$ of $\Omega''$ and a set $Z''\subseteq V(G'')$ such that for each $i\in[2t]$ the cross $\mathcal{Q}_i$ intersects $G''$ exactly in a cross of $\Omega''$ at $Z''.$ 
 
Now an application of \autoref{@suppressed} yields the existence of a family $\mathcal{D}$ of $\frac{1}{8}\cdot8\cdot\sqrt{4t^2}=2t$ pairwise disjoint cycles, all of which have traces that separate $V(\Omega)$ from $c$ such that $\mathcal{D}$ is \hyperref[def_orthogonal]{orthogonal} to $\mathcal{Q}.$ 
Hence, the graph consisting of the cycles of $\mathcal{D}$ together with the paths of the crosses from $\mathcal{Q}$ is a \hyperref[@attachements]{minor model} of a shallow vortex grid of order $2t$ as desired. 
\end{proof}

\subsection{Excluding a shallow vortex minor}\label{@harmonized}

In the previous section we have seen that, under the absence of a shallow vortex grid, any vortex can be completely separated from the rest of a $\Sigma$-decomposition with a small set of vertices.
Towards a complete structural description of graphs excluding a fixed shallow vortex minor we need to push this idea one step further.
One way to do this would be to directly evoke the global structure theorem for $H$-minor-free graphs of Robertson and Seymour and afterwards remove all vortices while slightly increasing the apex set.
However, this approach has the problem that the process of ``killing'' the vortices has the potential to change the decomposition within the collection of subtrees that attached to the linear decomposition of a vortex.
To avoid the backtracking such an approach would require, it is more convenient to instead give a slightly altered version of the proof of the global structure theorem from \cite{kawarabayashiTW21quickly} which builds the desired decomposition inductively.

What remains of this section is now dedicated to the proof of \autoref{@guarantors}.
As described above, we heavily lean on the proof of the global structure theorem from \cite{kawarabayashiTW21quickly} which itself is based on a similar proof in \cite{diestel2012excluded} and inspired by a preliminary version of the original theorem that can be found in \cite{robertson1991graph}.
The central tool for these proofs is an object dual to small treewidth called a \emph{tangle}.

\begin{definition}[Tangle]\label{@explanation} 
Let $G$ be a graph and $k$ be a positive integer. 
We denote by $\mathcal{S}_k$ the collection of all tuples $(A,B)$ where $A,B\subseteq V(G)$ and $(A,B)$ is a separation of order $<k$ in $G.$ 
An \emph{orientation} of $\mathcal{S}_k$ is a set $\mathcal{O}$ such that for all $(A,B)\in\mathcal{S}_k$ exactly one of $(A,B)$ and $(B,A)$ belongs to $\mathcal{O}.$ 
 
A \emph{tangle} of order $k$ in $G$ is an orientation $\mathcal{T}$ of $\mathcal{S}_k$ such that for all $(A_1,B_1),(A_2,B_2),(A_3,B_3)\in\mathcal{T}$ we have $A_1\cup A_2\cup A_3\neq V(G).$
\end{definition}

Let $G$ and $H$ be graphs and let $p\coloneqq  \Abs{V(H)}.$
Let $\mathcal{T}$ be a tangle in $G.$
A \hyperref[@attachements]{minor model} of $H$ in $G$ is \emph{controlled} by $\mathcal{T}$ if for every separation $(A,B)\in\mathcal{T}$ with $\Abs{A\cap B}<p$ there does not exist a branch set of the model contained in $A\setminus B.$
For a set $Z\subseteq V(G)$ we denote by $\mathcal{T}-Z$ the tangle in $G-Z$ formed by the separations $(A,B)$ in $G-Z$ such that $(A\cup Z,B\cup Z)\in\mathcal{T}.$
Given a surface $\Sigma$ and a $\Sigma$-decomposition $ \Delta$ of $G,$ we say that the decomposition is \emph{$\mathcal{T}$-central} if for every separation $(A,B)\in\mathcal{T},$ there does not exist a cell $c\in C( \Delta)$ such that $B\subseteq V( \sigma_{ \Delta}(c)).$

We will need the following theorem.

\begin{proposition}[\!\!\cite{kawarabayashiTW21quickly}]\label{@senselessly} 
There exists an absolute constant $c$ which satisfies the following. 
Let $p\geq 1$ be a positive integer. 
Let $\mathcal{T}$ be a tangle of order $\Theta$ in a graph $G$ with 
\begin{align*} 
\Theta\geq p^{18\cdot 10^7p^{26}+c}. 
\end{align*} 
Then $\mathcal{T}$ either controls a $K_p$-minor or there exists $A\subseteq V(G),$ $\Abs{A}\leq 5p^2\cdot p^{10^7p^{26}},$ a surface $\Sigma$ of Euler-genus at most $p(p+1),$ and a $\Sigma$-decomposition $ \Delta$ of $G-A$ of breadth at most $2p^2$ and depth at most $p^{10^7p^{26}}$ which is $(\mathcal{T}-A)$-central. 
\end{proposition}

Instead of \autoref{@lovelessly}, we prove a slightly stronger statement which will imply our main theorem.

\autoref{@diskussion} is stronger than \autoref{@lovelessly} in two ways.
First, it gives much more intricate information on the different parameters of the tree decomposition produced including the information that any adhesion set can avoid the apex set of any of its neighbouring bags in at most three vertices.
Second, it essentially shows that it does not matter where we begin to construct the decomposition.
This is facilitated by the free choice of the set $Z$ which acts as a way to ``anchor'' the decomposition.
In each step of the inductive proof, we try to find a balanced separator for $Z$ of small order.
If we can find such a separator we may split the graph and apply induction on each of the pieces, otherwise $Z$ does not have a small balanced separator.
In the latter case we get a tangle of large order from $Z$ which allows us to use \autoref{@senselessly} to either find a large clique minor or a $\Sigma$-decomposition central to the tangle.
If we find the clique we are done as it contains a shallow vortex grid, otherwise we apply \autoref{@constitution} to each of the vortices in order to obtain a vortex-free $\Sigma$-decomposition $ \Delta.$
By definition of $\Sigma$-decompositions, after deleting the resulting apex set, any non-node vertex can be separated from the nodes of $ \Delta$ by deleting at most three vertices.
This allows us to continue the induction within each of the cells with non-empty interiors.
Such an approach to constructing decompositions for global decompositions has already been employed for the proof of Theorem 11.3 in \cite{RobertsonS91X} and was also used for the proof of the global structure theorem in \cite{kawarabayashiTW21quickly}.

\begin{theorem}\label{@diskussion} 
Let $t\geq 1$ be a positive integer and let $H$ be a minor of the shallow vortex grid of order $t.$ 
Let $G$ be a graph that does not contain $H$ as a minor. 
Let $\alpha \coloneqq  t^{18\cdot 10^8t^{78}+c}+96t^{10^8t^{78}+5}+1$ and $\gamma=4t^4+2t^2.$ 
Let $Z\subseteq V(G)$ be such that $\Abs{Z}\leq3\alpha .$ 
Then $G$ has a tree decomposition $(T,\beta )$ with a root $r\in V(T)$ and adhesion at most $\alpha $ such that for every $d\in V(T),$ the torso $G_d$ of $G$ at $d$ has a set $A_d\subseteq V(G_d)$ of size at most $4\alpha $ for which the graph $G_d-A_d$ has Euler-genus at most $\gamma.$ 
Moreover, we have $Z\subseteq A_r,$ for every $(d_1,d_2)\in E(T)$ we have $\Abs{(\beta (d_1)\setminus A_{d_1})\cap(\beta (d_2)\setminus A_{d_2})}\leq 3,$ and if $\Abs{(\beta (d_1)\setminus A_{d_1})\cap(\beta (d_2)\setminus A_{d_2})}=3$ and $\beta (d_1)$ is larger than $4\alpha ,$ then $(\beta (d_1)\setminus A_{d_1})\cap(\beta (d_2)\setminus A_{d_2})$ induces a triangle in $G_{d_1}-A_{d_1}$ which bounds a face.
\end{theorem}

\begin{proof} 
Let us assume the assertion is false. 
We fix $G,$ $t$ and $Z$ to form a counter example minimizing $\Abs{G}+\Abs{G\setminus Z}.$
 
Observe that, in case $t=1$ we have that $H$ is a single crossing minor and thus the claim follows immediately from the main result of \cite{RobertsonS93exclu}. 
Hence, we may assume $t\geq 2.$ 
Furthermore, we may assume $\Abs{G}>4\alpha $ as otherwise the trivial tree decomposition on a tree with a single vertex $r$ and $\beta (r)=V(G)$ and $A_r=V(G)$ would satisfy our claim. 
Moreover, from our minimality assumptions it follows that $\Abs{Z}=3\alpha $ since otherwise we could add another vertex to $Z$ and find a smaller counter example. 
 
\textbf{Claim 1.} For all separations $(X_1,X_2)$ of order less than $\alpha $ we have $\Abs{X_i\cap Z}\leq \Abs{X_1\cap X_2}$ for exactly one $i\in[2].$ 
 
\emph{Proof of Claim 1.} Let $(X_1,X_2)$ be a separation of order less than $\alpha .$ 
If $\Abs{X_i\cap Z}\leq \Abs{X_1\cap X_2}<\alpha $ we would have $\Abs{Z}<2\alpha $ contradicting our assumption that $\Abs{Z}=3\alpha .$ 
 
Hence, it suffices to show that $\Abs{X_i\cap Z}>\Abs{X_1\cap X_2}$ cannot hold for both $i\in[2].$ 
Towards a contradiction let us assume that the inequality holds for both $i\in[2].$ 
Note that we must have that $(X_1,X_2)$ is non-trivial. 
For each $i\in[2]$ let $Z_i\coloneqq  (Z\cap X_i)\cup (X_1\cap X_2).$ 
By minimality the theorem holds for the two graphs $G_i\coloneqq  \InducedSubgraph{G}{X_i}$ with the set $Z_i$ respectively. 
So for each $i\in[2]$ there exists a tree decomposition $(T_i,\beta _i)$ with root $r_i$ such that for every $d\in V(T_i)$ the torso $G_{i,d}$ of $G_i$ at $d$ has an apex set $A_{i,d}$ of size at most $4\alpha $ such that $G_{i,d}-A_{i,d}$ has Euler-genus at most $\gamma.$ 
Let $T$ be the tree formed by introducing a new vertex $r$ to $T_1\cup T_2$ and joining it with edges to the vertices $r_1$ and $r_2.$ 
Let $\beta (r)\coloneqq  Z\cup(X_1\cap X_2)$ and $\beta (d)\coloneqq  \beta _i(d)$ if $d\in V(T_i)$ for all $d\in V(T_1)\cup V(T_2).$ 
Note that $Z_i\subseteq \beta (r_i)$ by assumption. 
As $\Abs{X_1\cap X_2}<\alpha $ and thus $\Abs{\beta (r)}<4\alpha $ it follows that $(T,\beta )$ is a tree decomposition for $G$ of adhesion at most $\alpha $ such that the torso at every vertex of $T$ has Euler-genus at most $\gamma$ after deleting at most $4\alpha $ vertices. 
This means that $G$ could not have been a counter example to our assertion in the first place and we have obtained a contradiction.\hfill$\blacksquare$ 
 
We may now consider the collection $\mathcal{T}$ of all separations $(X,Y)$ of order less than $\alpha $ in $G$ such that $\Abs{Z\cap Y}>\alpha .$ 
Recall that $\mathcal{S}_{\alpha }$ is the family of all separations of order less than $\alpha $ in $G.$ 
It follows from Claim 1 that $\mathcal{T}$ is in fact an orientation of $\mathcal{S}_{\alpha }.$ 
Moreover, given three separations $(X_1,Y_1),(X_2,Y_2),(X_3,Y_3)\in\mathcal{T},$ as each $Y_i$ contains at most $\alpha -1$ vertices of $Z$ and $\Abs{Z}=3\alpha $ we have $V(G)\neq Y_1\cup Y_2\cup Y_3$ and thus $\mathcal{T}$ is a tangle of order $\alpha .$ 
 
As $H$ is a minor of the shallow vortex grid of order $t,$ which itself has $2t^2$ vertices and thus is contained in $K_{2t^2},$ $G$ cannot have a $K_{2t^2}$-minor. 
Hence, by \autoref{@senselessly} there exists a surface $\Sigma$ of Euler-genus at most $2t^2(2t^2+1)=\gamma,$ a subset $A'\subseteq V(G)$ with $\Abs{A'}\leq5(2t^2)^2\cdot(2t^2)^{10(2t^2)^7(2t^2)^{26}},$ and a $\Sigma$-decomposition $\rho$ of $G-A'$ of breadth at most $2(2t^2)^2=8t^4$ and depth at most $(2t^2)^{10^7(2t^2)^{26}}\leq t^{10^8t^{78}}$ which is $\mathcal{T}-A'$-central. 
Moreover, by applying \autoref{@mouthpiece} we may further assume that for each vortex cell $c\in C(\rho)$ there exists a nest $\mathcal{C}_c$ in $\rho$ around the unique disk $ \Delta\in\mathcal{D}$ corresponding to $c$ of order $10^{21}(2t^2)^{100}.$ 
Also, for each vortex cell $c\in C(\rho),$ if $ \Delta_c\subseteq\Sigma$ is the disk bounded by the trace of $C_{10^{21}(2t^2)^{100}}\in\mathcal{C}_c,$ then for each pair of distinct vortex cells $c,c'\in C(\rho)$ we have that $ \Delta_c\cap  \Delta_{c'}=\emptyset.$ 
If we use \autoref{@transformation} we get, in addition to the above properties, that for each vortex cell $c\in C(\rho)$ the rendition $(\rho_{\mathcal{C}_c},H_{\mathcal{C}_c},\Omega_{\mathcal{C}_c})$ of the nest $\mathcal{C}_c$ is \hyperref[@maidservants]{$(2t^2)^{10^7(2t^2)^{26}}\leq t^{10^8t^{78}}$-tight}. 
 
We now break $G$ into subgraphs based on $\rho.$ 
By the minimality of $G,$ each subgraph has a tree decomposition with the desired properties that can be attached to the part of $G$ that is properly drawn on $\Sigma.$ 
The only difference between this proof and the one from \cite{kawarabayashiTW21quickly} is the way we treat vortices. 
In the end, all decompositions we obtain for the subgraphs will be combined into one tree decomposition for the entirety of $G,$ resulting in a contradiction to $G$ being a counter example. 
 
Let us start by treating the vortices. 
Let $c\in C(\rho)$ be an arbitrary vortex cell. 
Then there exists a nest $\mathcal{C}_c=\Set{C_1,\dots,C_{10^{21}(2t^2)^{100}}}$ of order $10^{21}(2t^2)^{100}$ and a disk $ \Delta_{\mathcal{C}_c}$ bound by the trace of $C_{10^{21}(2t^2)^{100}}$ such that the rendition $(\rho_{\mathcal{C}_c},H_{\mathcal{C}_c},\Omega_{\mathcal{C}_c})$ is \hyperref[@maidservants]{$t^{10^8t^{78}}$-tight}. 
By \autoref{@constitution} we either find a shallow vortex grid of order $t$ as a minor, and thus we obtain an $H$ minor in $G$ which is impossible, or there exists a separation $(A_c,B_c)$ of order at most $12\cdot t^{10^8t^{78}}(t-1)$ with $V(\Omega_{\mathcal{C}_c})\subseteq A$ such that $(\InducedSubgraph{H_{}\mathcal{C}_c}{A},\Omega_{\mathcal{C}_c}')$ has a vortex-free rendition in the disk. 
Let $S_c\coloneqq  A_c\cap B_c.$ 
Since $G$ is $H$-minor-free we can find the separation $(A_{c'},B_{c'})$ for every vortex in $\rho.$ 
Moreover, $B_{c'}\cap B_{c''}=\emptyset$ for all distinct vortices $c',c''\in C(\rho).$ 
Hence, by removing the union of all $B_{c'}$ from $G$ we obtain a vortex-free $\Sigma$-decomposition for the resulting graph $G'.$ 
Let $Z$ be the union of all $S_{c'}$ and $B$ be the union of all $B_{c'}$ over all vortex cells of $\rho.$ 
Since there are at most $8t^4$ vortices we have that  
\begin{align*} 
\Abs{Z}\leq 8t^4\cdot 12\cdot t^{10^8t^{78}}(t-1)\leq 96\cdot t^{10^8t^{78}+5}.
\end{align*} 
Now let $G'\coloneqq  G-(B\setminus S),$ moreover, let $A\coloneqq  A'\cup S.$ 
Notice that $\Abs{A}\leq \alpha -1.$ 
Then $\rho$ contains a vortex-free $\Sigma$-decomposition $\rho'$ of $G'-A.$ 
 
It follows that for every vortex cell $c$ we have that $\Abs{B_c\cap Z}\leq \alpha -1.$ 
Let $H_c$ be the subgraph of $G$ induced by $B_c\cup A\cup(B_c\cap Z).$ 
By the minimality of $G,$ there exists a tree decomposition $(T_c,\beta _c)$ of adhesion at most $\alpha $ such that the torso of every bag has Euler-genus at most $\gamma$ after the deletion of an apex set of size at most $4\alpha .$ 
Moreover, $T_c$ has a root $r_c$ such that $A\cup (B_c\cap Z)\in\beta _c(r_c)$ and it is a subset of the apex set of the torso of $H_c$ at $r_c.$ 
 
For every $o\in C(\rho')$ that is not a vertex let $H_o$ be the subgraph of $G$ induced by $V( \sigma(o))\cup A.$ 
Moreover, let $B_o$ be the collection of vertices of $G'$ that are drawn on the boundary of $o.$ 
Notice that $\Abs{B_o}\leq 3.$ 
As before $A\cup B_o$ is a separator of order at most $\alpha -1$ separating the vertices of $H_o$ from the rest of $G.$ 
Hence, $H_o$ may contain at most $\alpha -1$ vertices of $Z.$ 
So by the minimality of $G,$ there exists a tree decomposition $(T_o,\beta _o)$ of adhesion at most $\alpha $ such that the torso of every bag has Euler-genus at most $\gamma$ after the deletion of an apex set of size at most $4\alpha .$ 
Moreover, $T_o$ has a root $r_o$ such that $A\cup (V(H_o)\cap Z)\in \beta _o(r_o)$ and it is a subset of the apex set of the torso of $H_o$ at $r_o.$ 
 
Now let $G''$ be the graph obtained from $G'$ by deleting the vertex set $V(H_o)\setminus (A\cup Z\cup B_o)$ for every $o\in C(\rho')$ that is not a vertex and adding an edge between every pair of vertices of $A\cup Z$ as well as between every pair of vertices of $B_o$ for every $o.$ 
 
Let $T$ be a tree obtained by introducing a new vertex $r$ and joining it with an edge to every $r_o$ for every $o\in C(\rho')$ that is not a vertex and every $r_c$ where $c\in C(\rho)$ is a vortex. 
Let $\beta (r)\coloneqq  V(G''),$ $\beta (d)=\beta _o(d)$ if $d\in V(T_o)$ for some $o,$ and $\beta (d)=\beta _c(d)$ if $d\in V(T_c)$ for some $c.$ 
Then $(T,\beta )$ is a tree decomposition for $G$ with adhesion at most $\alpha .$ 
For all $d\neq r$ the torso of $G$ at $d$ has Euler-genus at most $\gamma$ after deleting an apex set of size at most $4\alpha ,$ we have $Z\subseteq \beta (r),$ and $G''$ has Euler-genus at most $\gamma$ after removing the vertices of $A\cup Z$ with $\Abs{A\cup Z}\leq 4\alpha .$  
Hence, $(T,\beta )$ is the desired tree decomposition for $G$ and our proof is complete.
\end{proof}

In \cite{kawarabayashiTW21quickly} it is mentioned that, in time $\mathcal{O}(f(\Abs{V(H)})\cdot \Abs{G}^3),$ where $f\colon \Bbb{N}\rightarrow\Bbb{N}$ is some computable function, one either finds the $\Sigma$-decomposition of \autoref{@mouthpiece} or a $K_{\Abs{V(H)}}$ minor for any graph $H.$
Since all other results and constructions used in the proof of \autoref{@diskussion} can be obtained in cubic time, we obtain the following corollary.

\begin{corollary}\label{@incinerated} 
There exists a computable function $f\colon\Bbb{N}\rightarrow\Bbb{N}$ such that for every positive integer $t,$ every graph $H$ which is a minor of the shallow vortex grid of order $t,$ and every graph $G,$ we can find in time $\mathcal{O}(f(t)\cdot \Abs{G}^3)$ either a \hyperref[@attachements]{minor model} of $H$ in $G,$ or a tree decomposition $(T,\beta )$ with a root $r\in V(T)$ and adhesion at most $\alpha $ such that for every $d\in V(T),$ the torso $G_d$ of $G$ at $d$ has a set $A_d\subseteq V(G_d)$ of size at most $4\alpha $ for which the graph $G_d-A_d$ has Euler-genus at most $\gamma,$ where $\gamma$ and $\alpha $ are defined as in \autoref{@diskussion}. 
Moreover, for every $(d_1,d_2)\in E(T)$ we have $\Abs{(\beta (d_1)\setminus A_{d_1})\cap(\beta (d_2)\setminus A_{d_2})}\leq 3,$ and if $\Abs{(\beta (d_1)\setminus A_{d_1})\cap(\beta (d_2)\setminus A_{d_2})}=3$ and $\beta (d_1)$ is larger than $4\alpha ,$ then $\beta (d_1)\setminus A_{d_1})\cap(\beta (d_2)\setminus A_{d_2})$ induces a triangle in $G_{d_1}-A_{d_1}$ which bounds a face.
\end{corollary}

Recall that, from \autoref{@diskussion}, in \autoref{@incinerated}, $\alpha (x)=\mathsf{poly}(x)$ and $\gamma(x)=2^{\mathsf{poly}(x)}.$

\subsection{Proof of the combinatorial lower bound}\label{@inevitably}

In this subsection we establish that \autoref{@lovelessly} is tight in the sense that the exclusion of a graph $H$ which is not a shallow vortex minor can never guarantee the absence of vortices while, at the same time, provide a global bound on the Euler-genus of all torsos.
This shows that the class of shallow vortex minors is exactly the class of all graphs whose exclusion as a minor allows for a version of the Graph Minors Structure Theorem without vortices.

Before we proceed to the proof we give some more definitions and make some observations.

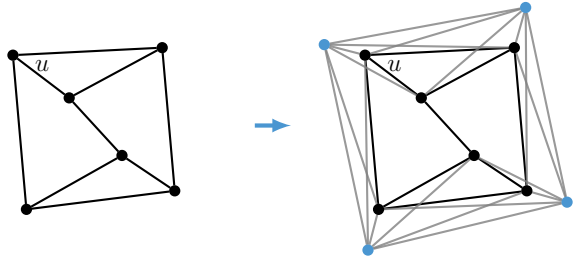
\begin{figure}[H] 
\centering 
\scalebox{0.8}{\begin{tikzpicture}[scale=1.3] 
\pgfdeclarelayer{background} 
\pgfdeclarelayer{foreground} 
\pgfsetlayers{background,main,foreground} 
 
\node [v:ghost] (Cleft) {}; 
\node [v:ghost,position=0:5mm from Cleft] (Cright) {}; 
\node [v:ghost,position=180:20mm from Cleft] (CL) {}; 
\node [v:ghost,position=0:20mm from Cright] (CR) {}; 
 
\node (aL) [v:main,position=135:5mm from CL] {}; 
\node (bL) [v:main,position=310:5mm from CL] {}; 
\node (cL) [v:main,position=320:13mm from CL] {}; 
\node (dL) [v:main,position=230:14mm from CL] {}; 
\node (eL) [v:main,position=140:14mm from CL] {}; 
\node (fL) [v:main,position=50:13mm from CL] {}; 
 
\node (uL) [v:ghost,position=340:4mm from eL] {\large $u$}; 
 
\node (aR) [v:main,position=135:5mm from CR] {}; 
\node (bR) [v:main,position=310:5mm from CR] {}; 
\node (cR) [v:main,position=320:13mm from CR] {}; 
\node (dR) [v:main,position=230:14mm from CR] {}; 
\node (eR) [v:main,position=140:14mm from CR] {}; 
\node (fR) [v:main,position=50:13mm from CR] {}; 
 
\node (vcR) [v:main,position=327:18mm from CR,color=celestialblue] {}; 
\node (vdR) [v:main,position=237:19mm from CR,color=celestialblue] {}; 
\node (veR) [v:main,position=147:19mm from CR,color=celestialblue] {}; 
\node (vfR) [v:main,position=57:18mm from CR,color=celestialblue] {}; 
 
\node (uR) [v:ghost,position=340:4mm from eR] {\large $u$}; 
 
\begin{pgfonlayer}{background} 
 
\draw[e:thick,->,color=celestialblue] (Cleft) to (Cright); 
 
\draw[e:main] (aL) to (bL); 
\draw[e:main] (bL) to (cL); 
\draw[e:main] (bL) to (dL); 
\draw[e:main] (cL) to (dL); 
\draw[e:main] (dL) to (eL); 
\draw[e:main] (eL) to (fL); 
\draw[e:main] (fL) to (cL); 
\draw[e:main] (aL) to (eL); 
\draw[e:main] (aL) to (fL); 
 
\draw[e:main] (aR) to (bR); 
\draw[e:main] (bR) to (cR); 
\draw[e:main] (bR) to (dR); 
\draw[e:main] (cR) to (dR); 
\draw[e:main] (dR) to (eR); 
\draw[e:main] (eR) to (fR); 
\draw[e:main] (fR) to (cR); 
\draw[e:main] (aR) to (eR); 
\draw[e:main] (aR) to (fR); 
 
\draw[e:main,opacity=0.8,color=gray] (cR) to (vcR); 
\draw[e:main,opacity=0.8,color=gray] (dR) to (vdR); 
\draw[e:main,opacity=0.8,color=gray] (eR) to (veR); 
\draw[e:main,opacity=0.8,color=gray] (fR) to (vfR); 
 
\draw[e:main,opacity=0.8,color=gray] (bR) to (vcR); 
\draw[e:main,opacity=0.8,color=gray] (bR) to (vdR); 
\draw[e:main,opacity=0.8,color=gray] (aR) to (veR); 
\draw[e:main,opacity=0.8,color=gray] (aR) to (vfR); 
 
\draw[e:main,opacity=0.8,color=gray] (cR) to (vdR); 
\draw[e:main,opacity=0.8,color=gray] (vcR) to (dR); 
\draw[e:main,opacity=0.8,color=gray] (vcR) to (vdR); 
 
\draw[e:main,opacity=0.8,color=gray] (dR) to (veR); 
\draw[e:main,opacity=0.8,color=gray] (vdR) to (eR); 
\draw[e:main,opacity=0.8,color=gray] (vdR) to (veR); 
 
\draw[e:main,opacity=0.8,color=gray] (eR) to (vfR); 
\draw[e:main,opacity=0.8,color=gray] (veR) to (fR); 
\draw[e:main,opacity=0.8,color=gray] (veR) to (vfR); 
 
\draw[e:main,opacity=0.8,color=gray] (fR) to (vcR); 
\draw[e:main,opacity=0.8,color=gray] (vfR) to (cR); 
\draw[e:main,opacity=0.8,color=gray] (vfR) to (vcR); 
 
\end{pgfonlayer} 
\end{tikzpicture}} 
\caption{The ring blowup of some cross-free drawing on a disk of a planar graph.} 
\label{@addressees}
\end{figure}

\begin{definition}[Ring Blowup Graphs]\label{@categorial}
Let $\gamma$ is the cross-free drawing on a disk $ \Delta$ of some (planar) graph $G.$
The face of $\gamma$ whose closure contains the boundary of $ \Delta$ is called \emph{external face} of $\gamma.$
Let $Q$ be the vertices of $G$ that are incident to this (unique) external face.  
 
The \emph{ring blowup} of $\gamma$ is the graph $G'$ obtained from $G$ by introducing for each $u\in Q$ a new vertex $v_u$ together with the edge $uv_u,$ the edges $\CondSet{wv_u}{w\in N_G(u)},$ and the edges $\CondSet{v_uv_w}{w\in N_G(u)\cap Q}.$ 
A graph $G$ is called a \emph{ring blowup graph} if it is a subgraph of some graph isomorphic to the ring blowup of a cross-free drawing on a disk $ \Delta$ of some planar graph. 
For instance $K_{7}$ is the  ring blowup of any cross-free drawing of  $K_{4}$ on a disk $ \Delta$  
(see \autoref{@addressees} for another example).

\begin{figure}[h] 
\centering
{\begin{tikzpicture}[scale=1.1] 
\pgfdeclarelayer{background} 
\pgfdeclarelayer{foreground} 
\pgfsetlayers{background,main,foreground} 
 
\foreach \r in {1,...,6} 
\foreach \x in {1,...,12} 
{  
\pgfmathsetmacro\Radius{\r*4+10} 
\node[v:main] () at (\x*30: \Radius mm){}; 
} 
 
\foreach \x in {1,...,12} 
{ 
\node[v:main,color=celestialblue] () at (\x*30:10mm){}; 
} 
 
\begin{pgfonlayer}{background} 
\foreach \r in {1,...,6} 
{  
\pgfmathsetmacro\Radius{\r*4+10} 
\draw[e:main] (0,0) circle ({\Radius mm}); 
} 
 
\foreach \x in {1,...,12} 
{ 
\draw[e:main] (\x*30:14mm) to (\x*30:34mm); 
\draw[e:main,color=gray] (\x*30:10mm) to (\x*30:14mm); 
\draw[e:main,color=gray,bend left=55] (\x*30:10mm) to (\x*30:18mm); 
\draw[e:main,color=gray,bend right=15] (\x*30:10mm) to (\x*30+30:14mm); 
\draw[e:main,color=gray,bend left=15] (\x*30+30:10mm) to (\x*30:14mm); 
} 
 
\draw[e:main,color=gray] (0,0) circle ({10mm}); 

\end{pgfonlayer} 
 
\end{tikzpicture}} 
\caption{A $(6\times 12)$-cylindrical grid ring blowup.} 
\label{@attributes}
\end{figure}
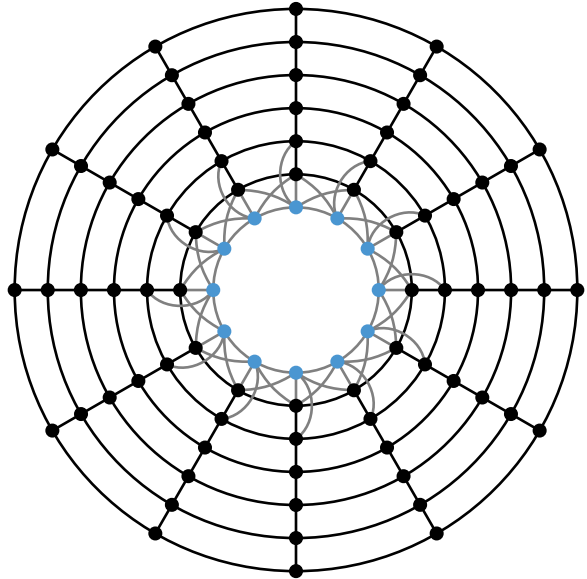 

Recall that, given $t\geq 1$ and $s\geq 3,$ we define the $(t\times s)$-cylindrical grid as the graph obtained if we take 
a $(t\times s)$-grid and then add the edge $(j,1)(j,s),$ for every $j\in[t].$
The \emph{extremal cycles} of a $(t\times s)$-cylindrical grid are 
the two cycles  that have all vertices of degree three (in \autoref{@industrial},  these cycles are depicted in \textcolor{CornflowerBlue}{blue}).

A \emph{standard cross-free drawing on a disk $ \Delta$} of the $(t\times s)$-cylindrical grid is one where of the extremal cycles is drawn on the boundary of $ \Delta$.
We say that a graph is a \emph{$(t\times s)$-cylindrical grid ring blowup}
if it is the ring blow up of a standard cross-free drawing on a disk $ \Delta$ of the $(t\times s)$-cylindrical grid. \end{definition}  

\begin{lemma}
\label{@incessantly}
There is a function $g:\Bbb{N}^{2}\to\Bbb{N}$ such that every $(t\times s)$-\hyperref[@consistently]{cylindrical grid ring blowup} is a minor of $\mathscr{S}_{g(t,s)}.$
\end{lemma}

\begin{proof}
The proof is depicted in \autoref{@geschichte}.
\end{proof}

\begin{figure}[th] 
\scalebox{0.78}{\includegraphics{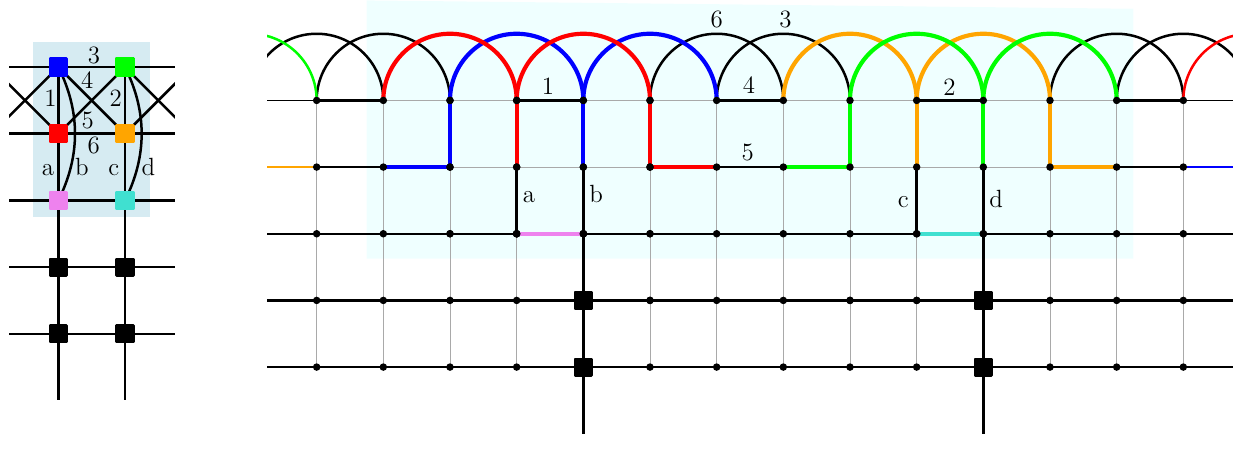}} 
\vspace{-3mm} 
\caption{A visualization of the proof of \autoref{@incessantly}. The leftmost picture is seen as a portion of the blown up part of the $(t\times s)$-\hyperref[@consistently]{cylindrical grid ring blowup} (see \autoref{@attributes}) and the right part is how this portion is routed through some big enough shallow vortex grid.} 
\label{@geschichte}
\end{figure}

\begin{definition}[The graph $Q_{s,r}$]
Given $s,r\geq 1$ we define the graph $Q_{s,r}$ as follows: 
We first consider the $(s\times s\cdot r)$-grid where we denote  
by $x_{1},\ldots,x_{s\cdot r}$ the vertices of some of the paths of length $s\cdot r$ where all internal vertices have degree three and the endpoints have degree two.
We then introduce $r$ pairs of vertices $\{t_{1},t_{1}',\ldots,t_{r},t_{r}'\}$ (we call them \emph{terminals}) and for every $i\in[r]$ we make $t_{i}$ and $t'_{i}$ 
adjacent with all the vertices in $\{x_{(i-1)+1},\ldots,x_{(i-1)+s}\}.$ For example, the graph $Q_{5,4}$ is depicted in \autoref{@germinating}, where the terminals of $Q_{s,r}$ are the squares vertices.
\end{definition}
Using \autoref{@attributes} and \autoref{@germinating}, one may easily verify the following:
\begin{lemma}
\label{@compensates}
For every $s,r\geq 1,$ $Q_{s,r}$ is a minor of an $((s+1)\times (r\cdot s))$-\hyperref[@consistently]{cylindrical grid ring blowup}.
\end{lemma}

We are now ready to prove the main result of this section.

\begin{figure}[h] 
\begin{center}
\hspace{-2.8cm}\scalebox{.76}{\input{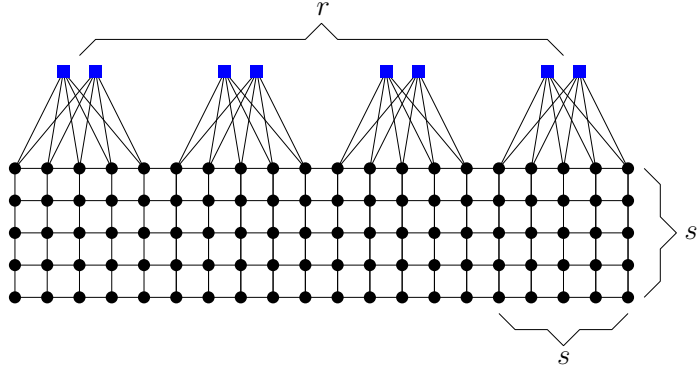}}
\end{center} 
\vspace{-4mm} 
\caption{The graph $Q_{s,r}$ (in particular, here $s=5$ and $t=4$).} 
\label{@germinating}
\end{figure}

\begin{lemma}
\label{@prophecies} 
There exists a function $f\colon\mathbb{N}\to\mathbb{N}$ such that for every $k\in\mathbb{N}$ and every graph $G,$ if $\mathsf{p}(G)\geq f(k),$ then ${\mathsf{p}}_{\mbox{\scriptsize \texttt{\tt -\!\!\! g\!\!\! a}}}(G)\geq k.$
In other words we have $\mathsf{p}\preceq {\mathsf{p}}_{\mbox{\scriptsize \texttt{\tt -\!\!\! g\!\!\! a}}}.$
\end{lemma}

\begin{proof}  
Let $g:\Bbb{N}\to\Bbb{N}$ be a function such that $\genus{K_{g(k)}}>k$ and $\genus{K_{3,g(k)}}>k$ (such a function exist because of the standard estimations on the Euler-genus of  
$K_{k}$ and $K_{3,k},$ see e.g., \cite{Harary91Grap}).
We consider the graph $Q_{s,r}$ where $s=g(k)+k$ and $r=k+1.$ 
By \autoref{@incessantly} and \autoref{@compensates}, there is a function $f:\Bbb{N}\to\Bbb{N}$ such that $Q_{s,r}$ is a minor of $\mathscr{S}_{f(k)}.$ 
Also, it is easy to verify that ${\mathsf{p}}_{\mbox{\scriptsize \texttt{-\!\!\! g\!\!\! a}}}$ is minor-closed.  
By the definition of $\mathsf{p},$ if, for some graph $G,$ $\mathsf{p}(G)\geq f(k),$ then $\mathscr{S}_{f(k)}$ is a minor of $G,$ which implies that $\mathsf{p}_{\mbox{\scriptsize \texttt{-\!\!\! g\!\!\! a}}}(G)\geq\mathsf{p}_{\mbox{\scriptsize \texttt{-\!\!\! g\!\!\! a}}}(\mathscr{S}_{f(k)})\geq \mathsf{p}_{\mbox{\scriptsize \texttt{-\!\!\! g\!\!\! a}}}(Q_{s,r}).$ 
It now remains to prove that $\mathsf{p}_{\mbox{\scriptsize \texttt{-\!\!\! g\!\!\! a}}}(Q_{s,r})> k.$  
 
Suppose to the contrary that $Q_{s,r}$ has a tree decomposition $(T,\beta )$ where every torso $G_{t}$ contains an apex set $A$ where $|A|\leq k$ and $\genus{G_{t}-A}\leq k.$  
 
We first rule out the possibility for two terminals $x,x'$ of $Q_{s,r}$ not to belong in the same torso of $(T,\beta ).$
Indeed, suppose to the contrary that there exist distinct $t,t'\in V(T)$ such that $x\in \beta (t)\setminus \beta (t'),$ $x'\in \beta (t')\setminus \beta (t).$ 
Notice that there are $s=g(k)+k$ internally disjoint paths in $Q_{s,r}$ between 
$x$ and $x',$ which implies that $\beta (t)\cap \beta (t')\geq g(k)+k.$ As $(T,\beta )$ 
is a tree decomposition, this implies that $\beta (t)\cap \beta (t')$ is a subset of some  
adhesion set of $G_{t}$ that induces a clique of size $\geq g(k)+k$ in $G_{t}.$ 
Let $A$ be the apex set of $G_{t}.$ The fact that $|A|\leq k$ implies that  
$G_{t}-A$ contains a clique of size $\geq g(k),$ therefore $\genus{G_{t}-A}\geq \genus{G_{g(k)}}> k,$ a contradiction. 
We conclude that there is some $t\in V(T)$ such that all $2r$ 
terminals of $Q_{s,r}$ belong in $G_{t}.$  
 
Notice now that the property that every two terminals are connected by $s$ internally disjoint paths in $Q_{s,r}$ holds for $G_{t}$ as well because their subpaths that do not belong to $G_{t}$ can be replaced by edges of cliques induced by the adhesion sets of $G_{t}.$ 
As $|A|\leq k$ and $r=k+1,$ at least one pair $x,x'$ of twin terminals is still present in $G_{t}- A$ as well as a third terminal $x''.$ Notice that $x''$ is connected with $x$ (as well as with $x'$) with $s=g(k)+k$ internall disjoint paths of $G_{t}$ and, among them, at least $g(k)$ also belong to $G_{t}-A.$
As $x$ and $x'$ are twins, these two collections of $g(k)$ paths each differ only on the edges incident to $x$ and $x'.$
Therefore, we may use them in order to find $K_{3,g(k)}$ as a minor of $G_{t}-A,$ therefore $\genus{G_{t}-A}\geq \genus{K_{k,g(k)}}> k,$ a contradiction.
\end{proof}

\section{Algorithmic consequences}\label{@homogenous}
With our structural results in place we are now able to discuss how to exploit the decomposition guaranteed by the absence of a large shallow vortex grid as a minor.
This section is divided into two parts.
First, in \autoref{@hairdressers}, we design a dynamic programming algorithm on the decomposition provided by \autoref{@diskussion}.
Then, in \autoref{@stiidentenzeitung}, we show that, for proper minor-closed classes, this result can be seen as best possible.

\subsection{Dynamic programming for the generating function}\label{@hairdressers}

We present the dynamic programming necessary for the algorithm of \autoref{@unencumbered} in iterative steps. The algorithm will be performed in a bottom up fashion along the decomposition of \autoref{@diskussion}.
First we describe the tables we compute in each step for some tuple $(G,\mathbf{p},X),$ where $G$ is a graph, $\mathbf{p}$ is a {particular} labelling of the edges of $G,$ and $X$ is some specified set of ``boundary'' vertices.
Then we provide a subroutine that computes the table of $(G,\mathbf{p},X)$ from the table of some $(H,\mathbf{p},Y)$ where the set $(V(G)\setminus V(H))\cup X\cup Y$ is of bounded size.
The next step is the introduction for a subroutine that produces a table for some graph of bounded Euler-genus which is then extended to the case where we encounter a bag of unbounded size in our decomposition, but have already computed all necessary tables for the subtrees below it.
The final piece of the algorithm will then be a procedure that merges the tables of several subtrees joined at a common adhesion set.

Pleace note that our choice to use fractional weights for the edges can be relaxed to rational weights by performing polynomial interpolation.
Moreover, the gadgets we implement in our proofs below to propagate partial solutions onto the pieces of bounded Euler genus but of undbounded size can possibly be replaced by Valiant-style matchgates as it is typically done in \cite{valiant2008holographic,CurticapeanX15param}.

\paragraph{The generating function of perfect matchings}

Our goal is to not only count the perfect matchings of a graph, {but to also differentiate between perfect matchings of different weight.}
This will be captured by the generating function of perfect matchings as explained below.

For our purposes it will be convenient to allow for a more general type of edge weighting 
where edges are labelled by fractions of integer polynomials. This {will allow us to encode parts of the tables of our dynamic programming directly on the edges}.
For this purpose let $\mathbb{Z}[x]$ be the set of all polynomials with integer coefficients and let
$$\Bbb{Z}(x)\coloneqq \CondSet{\frac{p}{q}}{p,q\in\Bbb{Z}[x],q\neq 0}.$$
Now let $\mathbf{p}\colon E(G)\rightarrow \Bbb{Z}(x),$ we call $(G,\mathbf{p})$ an \emph{labelled graph} and $\mathbf{p}$ is its \emph{labeling}.
If $(G,\mathbf{w})$ is an edge-weighted graph, we derive a labelled graph $(G,\mathbf{p}_{\mathbf{w}})$ from $(G,\mathbf{w})$ by setting, for every $e\in E(G),$  $$\mathbf{p}_{\mathbf{w}}(e)\coloneqq  x^{\mathbf{w}(e)}.$$
We will proceed to give general definitions in terms of labelings, but using the definition of $\mathbf{p}_{\mathbf{w}}$ one can easily derive the more conventional versions of these concepts.

Now assume we are given a labelled graph $(G,\mathbf{b})$ with a perfect matching $M.$
We express the total weight of $M$ under $\mathbf{p}$ as the monomial
\begin{align*} 
\MatchingMonomial{\mathbf{p}}{M}\coloneqq  \prod_{e\in M}\mathbf{p}(e).
\end{align*}

\begin{definition}[(labelled) Generating Function of Perfect Matchings]\label{@apocryphal} 
Let $(G,\mathbf{p})$ be a labelled graph.
The \emph{labelled generating function of the perfect matchings of $G$}, usually abbreviated as the \emph{$\mathbf{p}$-generating function} or the \emph{generating function}, is defined to be the polynomial 
\begin{align*} 
\GenerateMatchings{G,\mathbf{p}}\coloneqq  \sum_{M\in\Perf{G}}\MatchingMonomial{\mathbf{p}}{M}. 
\end{align*}  
Let $(G,\mathbf{w})$ be an edge-weighted graph, then the \emph{generating function of all weighted perfect matchings in $G$} is the polynomial 
\begin{align*} 
\GenerateMatchings{G,\mathbf{w}}\coloneqq  \sum_{M\in\Perf{G}}\MatchingMonomial{\mathbf{p}_{\mathbf{w}}}{M}. 
\end{align*}
\end{definition}

Our aim is to eventually compute the generating function of all weighted perfect matchings of an edge-weighted graph $(G,\mathbf{w}).$
For this we will compute $\GenerateMatchings{G,\mathbf{w}}$ from a series of partial generating functions of labelled subgraphs of $G$ with some specified set of ``boundary'' vertices, where the labelings are derived from $\mathbf{w}.$

Let $G$ be a graph with a perfect matching and $F\subseteq E(G)$ be some (not necessarily perfect) matching in $G.$
We say that a vertex $v\in V(G)$ is \emph{covered} by $F$ if $v$ is an endpoint of some edge in $F,$ the vertex $v$ is said to be \emph{exposed} by $F$ if it is not covered by it.
We say that $F$ is \emph{extendable} if there exists a perfect matching $M\in\Perf{G}$ such that $F\subseteq M.$
In other words, $F$ is extendable if it can be extended to a perfect matching of the entire graph.

Given a pair $(G,X)$ where $X\subseteq V(G)$ we say that a matching $F$ of $G$ is \emph{internally extendible} (or, simply, \emph{extendible}) in $(G,X)$ if $F$ is an extendable matching in the graph $G-(X\setminus V(F)).$
We also say that $F$ is \emph{aligned} with $(G,X)$ if it is internally extendible in $(G,X)$ and every edge of $F$
contains some endpoint in $X.$

Roughly speaking, $F$ being aligned with $(G,X)$ means that $F$ classifies exactly the matchings of $G$ which are extensions of $F,$ leave the vertices of $X\setminus V(F)$ exposed, but cover every other vertex $G-(X\setminus V(F)).$
Another way of explaining the idea behind aligned matchings is the following.
Let $G$ be a graph with a perfect matching, let $(A,B)$ be a separation of $G$ and let $H\coloneqq G[B].$
Consider the pair $(H,A\cap B).$
Now for every perfect matching $M$ of $G$ we can find the set $F_M\subseteq M$ of exactly those edges in $M$ with both endpoints in $H$ and at least one endpoint in $A\cap B.$
This allows us to define an equivalence relation on $\mathcal{M}(G)$ by saying that two perfect matchings $M$ and $M'$ are equivalent if $F_M=F_{M'}.$
Hence, each of these equivalence classes is uniquely defined by some set $F\subseteq E(H)$ which is aligned with $(H,A\cap B).$

Let us denote by $\Aligned{G,X}$ the collection of all matchings that are aligned with $(G,X).$

Please note that testing whether a given matching $F$ is extendable can be done in polynomial time by simply checking if $G-V(F)$ has a perfect matching using known polynomial time algorithms for maximum matching \cite{edmonds1965paths}.
It follows that also checking for the extendability of some pair $(F,X)$ as above is possible in polynomial time.

A pair $(G,X,\mathbf{p})$ where $(G,\mathbf{p})$ is a labelled graph and $X\subseteq V(G)$ is a set of vertices is called a \emph{(labelled) boundary graph}.
We say that $X$ is the \emph{boundary} of $(G,X,\mathbf{p}).$
Let $H\subseteq G$ be an induced subgraph of $G,$ and $Y\subseteq V(H)$ be a set of vertices in $H$ such that in $G-Y$ there is no edge with one endpoint in $H-Y$ and the other in $G-H.$
We call $(H,Y,\mathbf{p})$ an \emph{(labelled) boundary subgraph} of $(G,X,\mathbf{p}),$ with \emph{boundary} $Y.$
In a slight abuse of notation we use $\mathbf{p}$ in $(H,Y,\mathbf{p}),$ but actually mean the restriction $\mathbf{p}_{|_{E(H)}}$ of $\mathbf{p}$ to the edges of $H.$
Note that $(H,Y,\mathbf{p})$ is itself also a boundary graph.

\begin{definition}[Partial Generating Function]\label{@intelligibility} 
Let $(G,X,\mathbf{p})$ be a labelled boundary graph and $F\in\Aligned{G,X}.$ 
The \emph{partial generating function of $(G,X)$ under $F$} is defined to be 
\begin{align*} 
\GenerateMatchings{G,X,\mathbf{p},F}\coloneqq  (\prod_{e\in F} \mathbf{p}(e))\cdot\GenerateMatchings{G-V(F)-X,\mathbf{p}}. 
\end{align*}
We define the set
\begin{align*} 
\GenerateMatchingsBoundary{(G,X,\mathbf{p})}\coloneqq \CondSet{\GenerateMatchings{G,X,\mathbf{p},F}}{F\in\Aligned{H,X}}
\end{align*}
to be the collection of all partial generating functions of $(G,X,\mathbf{p})$ under its aligned matchings.
For every $F\in\Aligned{H,X}$ we use $\GenerateMatchingsBoundaryF{(G,X,\mathbf{p})}{F}$ to denote $\GenerateMatchings{G,X,\mathbf{p},F}.$
\end{definition}
Note that $\GenerateMatchingsBoundary{}$ is technically a mapping itself. 
We write its argument in the subscript to better distinguish $\GenerateMatchingsBoundary{(G,X,\mathbf{p})}$ and $\GenerateMatchingsBoundaryF{(G,X,\mathbf{p})}{F}.$

The definitions above correspond a set $\GenerateMatchingsBoundary{(G,X,\mathbf{p})}$ to every boundary graph $(G,X,\mathbf{p}).$
This collection of partial generating functions will serve as the table of our dynamic programming for $(G,X,\mathbf{p}).$

\paragraph{Bags of Bounded Size}
We start the discussion of our algorithm with a way to compute the table entries for our dynamic programming in the case where we are concerned with a single bag of bounded size and we are given the unified table for all of its children.
The lemma presented here could as well be regarded as a special case of the way we treat bags of unbounded size by simply assuming that the apex set equals the entire bag.
However, we deem this special case to be a nice illustration of how we compute our tables.

\begin{figure}[th]
\begin{center} 
\hspace{-.4cm}\scalebox{1}{\includegraphics{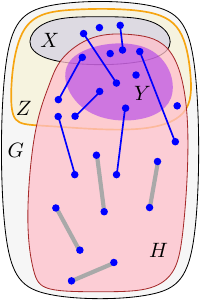}} 
\end{center} 
\vspace{-3mm} 
\caption{A visualization of the boundaried graphs $(G,X)$ and $(H,Y)$ and of the set $Z$ in the statement of \autoref{@celebration}. Also the matching $F$ (in turquoise) is aligned with $(G,Z).$} 
\label{@contrivance}
\end{figure}

\begin{lemma}\label{@celebration}
Let $k\geq 1$ be a positive integer.
Let $(G,X,\mathbf{p})$ be a labelled boundary graph and $(H,Y,\mathbf{p})$ be a labelled boundary subgraph of $(G,X,\mathbf{p})$ such that there exists a set $Z\subseteq V(G)$ with $X,Y\subseteq Z,$ $G-Z=H-Y,$ and $\Abs{Z}\leq k.$
There exists an algorithm that, given the set $\GenerateMatchingsBoundary{(H,Y,\mathbf{p})},$ computes the set $\GenerateMatchingsBoundary{(G,X,\mathbf{p})}$ in time $\Abs{G}^{\mathcal{O}(k)}.$
\end{lemma}

\begin{proof} 
Notice that $(G,Z,\mathbf{p})$ is also a labelled boundary graph. 
Let $\mathcal{F}$ be the collection of all matchings $F$ in $G$ such that $F\in\Aligned{G,Z}$ and observe that, with $\Abs{Z}\leq k,$ we have $\Abs{\Aligned{G,Z}}\in \mathcal{O}(\Abs{G}^{2k}).$ 
Moreover, it takes time $\mathcal{O}(\Abs{G}^c)$ to check for any such matching $F$ whether it is extendable in $G-(Z\setminus V(F)),$ for some constant $c.$ 
Hence, we can find $\mathcal{F},$ and $\Aligned{G,Z}$ as a consequence, in time $\mathcal{O}(\Abs{G}^{2k+c}).$
Using the same argument and the fact that $X,Y\subseteq Z,$ we are able to find the sets $\Aligned{G,X}$ and $\Aligned{H,Y}.$ (See \autoref{@contrivance} for a visualization of the boundaried graphs $(G,X)$ and $(H,Y).$) 
 
Now each $F\in\mathcal{F}$ can be covered by the sets $F_1, F_2,$ and $F_3$ as follows. 
\begin{itemize} 
\item Let $F_1$ be the set of edges in $F$ with at least one endpoint in $X,$ 
\item let $F_2$ be the set of edges in $F$ with both endpoints in $V(H),$ and 
\item let $F_3$ be the set of edges in $F$ with at most one endpoint in $Y.$ 
\end{itemize} 
It follows that, for every $F\in\mathcal{F},$ $F_1\in\Aligned{G,X},$ $F_2\in\Aligned{H,Y},$ and $F_3= F\setminus F_2.$ 
 
For each $R\in\Aligned{G,X}$ let $\mathcal{F}_R\subseteq \mathcal{F}$ be the collection of matchings $F\in\mathcal{F}$ such that $R=F_1.$ 
Notice that $\CondSet{\mathcal{F}_R}{R\in\Aligned{G,X}}$ is a partition of $\mathcal{F}.$
Moreover, let $F'\in \Aligned{G,X}.$
Then $G-(X\setminus V(F'))$ has a perfect matching, say $M,$ with $F'\subseteq M$ and we may set $F''$ to be the set of edges from $M$ with at least one endpoint in $Z.$
It follows that $F'\subseteq F''\in \Aligned{G,Z}=\mathcal{F}.$
Thus, for each $R\in\Aligned{G,X}$ there is at least one $F\in\mathcal{F}$ such that $R=F_1.$

By definition, for each $R\in\Aligned{G,X},$ $\GenerateMatchingsBoundaryF{(G,X,\mathbf{p})}{R}$ is the generating function for all perfect matchings of $G-(X\setminus V(R))$ that contain $R.$
Let $M$ be such a perfect matching of $G-(X\setminus V(R)),$ then $M$ can be partitioned in to four sets as follows
\begin{itemize}
    \item $M_1\coloneqq  R,$
    \item $M_2,$ which is the set of all edges in $M\setminus R$ with at most one endpoint in $H,$
    \item $M_3,$ which is the set of all edges of $M\setminus R$ with both endpoints in $H$ and at least one endpoint in $Y,$ and
    \item $M_4\coloneqq  M\setminus (M_1\cup M_2\cup M_3).$
\end{itemize}
We may further split $R$ into the sets $R_1\coloneqq  R\cap E(H)$ and $R_2\coloneqq  R\setminus E(H).$
Observe the following:
\begin{enumerate}
    \item By the definition of boundary subgraphs and the choice of $Z$ we have that every edge of $M$ with exactly one endpoint in $H$ has an endpoint in $Y.$
    \item $M_1\cup M_2\cup M_3$ is aligned with $(G,Z).$
    \item $R_1\cap M_3$ is aligned with $(H,Y)$ and the vertices in $Y\setminus V(R_1\cap M_3)$ are exactly those vertices of $H$ which either belong to $X\setminus V(R)$ or are covered by $M_1\cup M_2.$
\end{enumerate}
Now consider some $F\in\mathcal{F}_R$ and recall the bipartition of $F$ into $F_2$ and $F_3$ as above.
As $R\subseteq F$ it follows that $R_1=F_2\cap R$ and $R_2=F_3\cap R$ form a bipartition of the set $R$ itself.
By definition we have that
\begin{align*}
    \GenerateMatchingsBoundaryF{(G,Z,\mathbf{p})}{F}&=(\prod_{e\in F}\mathbf{p}(e))\cdot\frac{\GenerateMatchingsBoundaryF{(H,Y,\mathbf{p})}{F_2}}{\prod_{e'\in F_2}\mathbf{p}(e')}\\
    &= (\prod_{e\in F_3}\mathbf{p}(e))\cdot \GenerateMatchingsBoundaryF{(H,Y,\mathbf{p})}{F_2}.
\end{align*}
It follows that 
\begin{align*} 
\GenerateMatchingsBoundaryF{(G,X,\mathbf{p})}{R}&=(\prod_{e\in R}\mathbf{p}(e))\cdot\sum_{F\in\mathcal{F}_R}\frac{\GenerateMatchingsBoundaryF{(G,Z,\mathbf{p})}{F}}{\prod_{e\in R}\mathbf{p}(e)}\\
&=(\prod_{e\in R}\mathbf{p}(e))\cdot\sum_{F\in\mathcal{F}_R}\frac{(\prod_{e\in F_3}\mathbf{p}(e))\cdot \GenerateMatchingsBoundaryF{(H,Y,\mathbf{p})}{F_2}}{\prod_{e\in R}\mathbf{p}(e)}
\\
\phantom{\GenerateMatchingsBoundaryF{(G,Z,\mathbf{p})}{F}}&=(\prod_{e\in R}\mathbf{p}(e))\cdot\sum_{F\in\mathcal{F}_R}(\prod_{e\in F_3\setminus R}\mathbf{p}(e))\cdot\frac{\GenerateMatchingsBoundaryF{(H,Y,\mathbf{p})}{F_2}}{\prod_{e'\in F_2\cap R}\mathbf{p}(e')}\\
&=\sum_{F\in\mathcal{F}_R}(\prod_{e\in F_3}\mathbf{p}(e))\cdot\GenerateMatchingsBoundaryF{(H,Y,\mathbf{p})}{F_2}. 
\end{align*} 
Given the sets $\mathcal{F},$ $\Aligned{G,X},$ and $\GenerateMatchingsBoundary{(H,Y)}$ it follows that $\GenerateMatchingsBoundaryF{(G,X)}{R}$ can be found in time $\Abs{G}^{\mathcal{O}(k)}$ and, since $|\Aligned{G,X}|=\mathcal{O}(\Abs{G}^k),$ our claim follows.
\end{proof}

\paragraph{Generating Functions, Permanents, and Pfaffian Orientations.}

The generating function of planar graphs and, in general, of graphs of bounded Euler-genus, is usually computed using Pfaffian orientations (in the case of planar graphs) or a linear combination of many different orientations derived from Pfaffian orientations (in the case of general graphs of bounded Euler-genus).
Before we go on we introduce this notion and explain how the generating function of an edge-weighted graph of bounded Euler-genus can be found using this concept.
We then elaborate on this a bit more and explain how we can, essentially, replace the monomial of an edge by a more complicated polynomial in order to encode entire generating functions linked to this edge.
This last step is of particular importance as it provides us with a way to handle the (possibly) unboundedly many different sets of size at most three onto which subtrees below a bag of unbounded size might attach.
 
An \emph{orientation} of a graph $G$ is a digraph $\vec{G}$ with vertex set $V(G)$ whose edge set is obtained by introducing for every edge $uv\in E(G)$ exactly one of the edges $(u,v)$ or $(v,u).$ 
Let $C$ be an even cycle of $G$ and let $\vec{G}$ be an orientation of $G.$ 
$C$ is said to be \emph{oddly oriented} by $\vec{G}$ if it has an odd number of directed edges in agreement with the clockwise traversal of $C.$ 
Notice that, since $C$ is even, if it is oddly oriented it must also have an odd number of directed edges in agreement with the counterclockwise traversal, so the property of being oddly oriented does not depend on the direction of traversal after all. 
Similarly, we say that a cycle $C$ of $G$ is \emph{evenly oriented} by $\vec{G}$ if it is not oddly oriented. 
 
Let $G$ be a graph with a perfect matching and $C$ be an even cycle. 
We call $C$ a \emph{conformal cycle} if $G-V(C)$ has a perfect matching. 
Notice here that, since $C$ is even, it has two disjoint perfect matchings and thus every perfect matching of $G-V(C)$ can be completed to a perfect matching of $G$ by choosing one of these two matchings. 
If $M$ is a perfect matching of $G$ and $C$ is a conformal cycle in $G$ such that $M$ contains a perfect matching of $C$ we say that $C$ is \emph{$M$-conformal}.

\begin{definition}[Sign Polynomials of Matchings and Orientations] \label{@degradation} 
Let $(G,\mathbf{p})$ be a labelled graph with a perfect matching $M$ and let $\vec{G}$ be an orientation of $G.$ 
For every perfect matching $N$ of $G$ let $\MatchingSign{\vec{G},M,N}\coloneqq  (-1)^n,$ where $n$ is the number of $M$-conformal cycles of $G$ which are also $N$-conformal and evenly oriented by $\vec{G}.$ 
We define the \emph{$M$-polynomial} of $\vec{G}$ as follows 
\begin{align*} 
\SignPolynomial{\vec{G},M,\mathbf{p}}\coloneqq  \sum_{N\in\Perf{G}}\MatchingSign{\vec{G},M,N}\cdot\MatchingMonomial{\mathbf{p}}{N}. 
\end{align*}
\end{definition} 

\begin{definition}[Pfaffians of Skew-Symmetric Matrices]\label{@schoolboys}
Let $(G,\mathbf{p})$ be a labelled graph and $\vec{G}$ be an orientation of $G$ such that $\Abs{G}=2n$ for some $n\in\Bbb{N}.$
We denote by $\SkewMatrix{\vec{G},\mathbf{p}}$ the skew-symmetric matrix with rows and columns indexed by $V(G),$ where $a_{uv}=\mathbf{p}(uv)$ in case $(u,v)\in E(\vec{G}),$ $a_{uv}=-\mathbf{p}(uv)$ if $(v,u)\in E(\vec{G}),$ and $a_{uv}=0$ otherwise.
The \emph{Pfaffian} of $\SkewMatrix{\vec{G},\mathbf{p}}$ is defined as
\begin{align*} 
\Pfaffian{\SkewMatrix{\vec{G},\mathbf{p}}}\coloneqq  \sum_{\pi}\mathsf{s}(\pi)\cdot a_{i_1,j_1}\cdots a_{i_n,j_n}
\end{align*}
where $\pi=\Set{\Set{i_1,j_1},\dots,\Set{i_n,j_n}}$ is a partition of the set $[2n]$ into pairs $i_k <j_k$ for every $k\in[n],$ and $\mathsf{s}(\pi)$ equals the sign of the permutation $i_1j_1\dots i_nj_n$ of $12\dots(2n).$  

{Each non-zero term of the expression of the Pfaffian of $\SkewMatrix{\vec{G},\mathbf{p}}$ equals $\MatchingMonomial{\mathbf{p}}{M}$ or $-\MatchingMonomial{\mathbf{p}}{M}$ where $M$ is a perfect matching of $G.$}
We denote by $\mathsf{s}(\vec{G},M)$ the sign of the term $\MatchingMonomial{\mathbf{p}}{M}$ in said expression of the Pfaffian of $\SkewMatrix{\vec{G},\mathbf{p}}.$
Notice that $\mathsf{s}(\vec{G},M)$ is independent of $\mathbf{p}$ as it only depends on $M,$ the encoding the permutation, and $\vec{G}$ determining whether $\mathbf{p}(uv)$ appears with its original sign or the opposing one.
It has been observed (see e.g., \cite{GalluccioL99onthe}) that 
\begin{align*} 
\Pfaffian{\SkewMatrix{\vec{G},\mathbf{p}}}=\sum_{M\in\Perf{G}}\mathsf{s}(\vec{G},M)\cdot\MatchingMonomial{\mathbf{p}}{M}.
\end{align*}
\end{definition}

The following theorem was proven by Kasteleyn \cite{Kasteleyn61thest} for the case where $\mathbf{p}(e)=x^{w(e)},$ $w(e)\in\Bbb{Z}.$
It still holds for arbitrary labelings $\mathbf{p}$ and we state it here in this more general form.

\begin{proposition}[\!\!\cite{Kasteleyn61thest}]\label{@derogatory} 
Let $(G,\mathbf{p})$ be a labelled graph with two perfect matchings $M,$ $N,$ and an orientation $\vec{G}.$ 
Then 
$
\mathsf{s}(\vec{G},N)=\mathsf{s}(\vec{G},M)\cdot \MatchingSign{\vec{G},M,N}. 
$
Hence, {for every perfect matching $N,$} 
\begin{align*} 
\Pfaffian{\SkewMatrix{\vec{G},\mathbf{p}}}&=\sum_{N\in\Perf{G}}\mathsf{s}(\vec{G},N)\cdot\MatchingMonomial{\mathbf{p}}{N}\\&=\mathsf{s}(\vec{G},M)\sum_{N\in\Perf{G}}\MatchingSign{\vec{G},M,N}\cdot\MatchingMonomial{\mathbf{p}}{N}=\mathsf{s}(\vec{G},M)\cdot \SignPolynomial{\vec{G},M,\mathbf{p}}. 
\end{align*}
\end{proposition} 

While computing the permanent, and thus computing the generating function for perfect matchings is generally a $\#\mathsf{P}$-hard problem \cite{Valiant79theco}, the Pfaffian of a skew-symmetric matrix can be expressed through its determinant and is thereby efficiently computable.

\begin{proposition}[\!\!\cite{cayley1847,Kasteleyn67graph}] 
Let $(G,\mathbf{p})$ be a labelled graph and an orientation $\vec{G},$ then 
\begin{align*} 
\Pfaffian{\SkewMatrix{\vec{G},\mathbf{p}}}^2=\mathsf{det}(\SkewMatrix{\vec{G},\mathbf{p}}). 
\end{align*}
\end{proposition}

In case a graph $G$ has an orientation $\vec{G}$ such that the signs produced by $\mathsf{s}(\vec{G},\cdot)$ and $\MatchingSign{\vec{G},\cdot,\cdot}$ would be independent of the matchings, this technique could be used to find the generating functions for the perfect matchings of $G$ efficiently.
This fact lead Kasteleyn to the definition of the so called ``Pfaffian orientations'', which are orientations that ensure $\MatchingSign{\vec{G},M,N}=1$ for all choices of $M$ and $N.$

\begin{definition}[Pfaffian Orientation]\label{@recognized}
Let $G$ be a graph with a perfect matching $M.$
An orientation $\vec{G}$ is called a \emph{Pfaffian orientation} if every $M$-conformal cycle of $G$ is oddly oriented by $\vec{G}.$
If $G$ has a Pfaffian orientation, $G$ is called \emph{Pfaffian}.
\end{definition}

As a first application of this idea, Kasteleyn showed that every planar graph is Pfaffian \cite{Kasteleyn61thest}.
He generalized this idea and stated that the generating function of a graph $G$ of orientable genus $g$ could be expressed as a linear combination of $4^g$ Pfaffians of different orientations of $G.$
This was later turned into a theorem by Galluccio and Loebl \cite{GalluccioL99onthe} and adapted as a time  
$\mathcal{O}_{k}(|G|^{\mathcal{O}(1)})$ 
algorithm for edge-weighted graphs whose weights are bounded in size by some polynomial in the size of the graph by Galluccio, Loebl, and Vondr\'ak \cite{galluccio2001optimization}.
An extension of this idea to graphs of bounded Euler-genus was found by \cite{Tesler00match}.
He uses $2^g$ many orientations, where $g$ now is the Euler-genus of $G.$
Both approaches were unified into a framework that uses planarizing gadgets instead of orientations by Curticapean and Xia in \cite{CurticapeanX15param}.

These algorithms are the centerpiece of our own algorithm as they allow us to find the generating functions for the torsos of unbounded size in our decomposition.
We slightly adapt the formulation of the results of Galluccio, Loebl, Vondr\'ak, and Tesler to match our more general setting of labelled graphs.
Since the edge weights only come into play in the computation of the $2^g$ Pfaffians and their linear combination, the more general version of their result still holds.

Let $(G,\mathbf{p})$ be a labelled graph.
We say that $\mathbf{p}$ is \emph{polynomially bounded} if there exists a polynomial $\mathsf{p}$ such that the degrees of the polynomials in
\begin{align*} 
\CondSet{p,q\in\Bbb{Z}[x]}{\text{there exists }e\in E(G)\text{ s.\@t.\@ }\mathbf{p}(e)=\frac{p}{q}\text{ is fully reduced}}
\end{align*}
are bounded by $\mathsf{p}(\Abs{G}).$

\begin{proposition}[\!\!\cite{galluccio2001optimization,Tesler00match}]\label{@wittgensteiris}
Let $g\in\Bbb{N}$ be an integer and $(G,\mathbf{b})$ be a labelled graph whose labeling is polynomially bounded.
There exists an algorithm that computes the labelled generating function of all perfect matchings of $(G,\mathbf{p})$ in time $\mathcal{O}_{k}(|G|^{\mathcal{O}(1)}),$ where $k=\genus{G}.$
\end{proposition}

There are two important observations to take away from this quick introduction to Pfaffians and, in particular, from \autoref{@wittgensteiris}.
The first is, that the generating function of (labelled) perfect matchings can be expressed as a linear combination of Pfaffians of the skew-symmetric matrices corresponding to some orientations of a bounded-genus graph. 
The second is that the Pfaffian can be expressed as a function of the determinant of a skew-symmetric matrix, which  is polynomially computable.

\paragraph{A polynomial algorithm for apex-bounded-genus graphs.}

The next subroutine for our algorithm will be a way to produce the generating function for the labelled perfect matchings of graphs that have Euler-genus at most $t$ after the deletion of at most $t$ vertices.
The lemma we prove here can be seen as a slight generalization of \autoref{@celebration} which now also incorporates \autoref{@wittgensteiris}.

\begin{lemma}\label{@transformations} 
Let $k$ be a positive integer. 
Let $(G,X,\mathbf{p})$ be a labelled boundary graph with $\Abs{X}\leq k$ and assume there exists a set $A\subseteq V(G)$ with $\Abs{A}\leq k$ such that the Euler-genus of $G-A$ is at most $k.$ 
There exists an algorithm that computes in time $\Abs{G}^{\mathcal{O}(k)}$ the set $\GenerateMatchingsBoundary{(G,X,\mathbf{p})}.$
\end{lemma}

\begin{proof} 
Notice that $(G,A\cup X,\mathbf{p})$ is a boundary graph with a boundary of size at most $2k.$ 
Let $\mathcal{F}$ be the collection of all matchings $F$ such that $F\in\Aligned{G,A\cup X}$ and every vertex of $A\setminus X$ is covered by an edge of $F.$ 
As discussed in the proof of \autoref{@celebration} we have $\Abs{\mathcal{F}}\in\Abs{G}^{\mathcal{O}(k)}$ and $\mathcal{F},$ as well as the set $\Aligned{G,X},$ can be found in time $\Abs{G}^{\mathcal{O}(k)}.$ 
Each set $F\in\mathcal{F}$ can be partitioned into two sets as follows: 
\begin{itemize} 
\item Let $F_1$ be the set of edges in $F$ with at least one endpoint in $X,$ and 
\item let $F_2\coloneqq  F\setminus F_1.$ 
\end{itemize} 
Observe that for each $F\in\mathcal{F}$ we have $F_1\in\Aligned{G,X},$ and $G-(X\cup V(F))\subseteq G-A$ has Euler-genus at most $k.$ 
For every $R\in\Aligned{G,X}$ let $\mathcal{F}_R\subseteq\mathcal{F}$ be the collection of all sets $F\in \mathcal{F}$ with $F_1=R.$
It follows that 
\begin{align*} 
\GenerateMatchingsBoundaryF{(G,X,\mathbf{p})}{R}&= (\prod_{e\in R}\mathbf{p}(e))\cdot \GenerateMatchings{G-V(R)-X,\mathbf{p}}\\
&= (\prod_{e\in R}\mathbf{p}(e))\cdot \sum_{F\in\mathcal{F}_R}(\prod_{e\in F\setminus R}\mathbf{p}(e))\cdot\GenerateMatchings{G-V(F)-X,\mathbf{p}}\\
&=\sum_{F\in\mathcal{F}_R} 
\big(\prod_{e\in F}\mathbf{p}(e)\big)\cdot\GenerateMatchings{G-V(F)- X,\mathbf{p}}. 
\end{align*} 
Since $G-A$ has Euler-genus at most $k$ and there are at most $\Abs{G}^{\mathcal{O}(k)}$ many sets in $\mathcal{F}_R,$ by calling the algorithm from \autoref{@wittgensteiris} $\Abs{G}^{\mathcal{O}(k)}$ many times we can compute $\GenerateMatchings{G-(V(F)\cup X),\mathbf{p}}$ for every $R\in\Aligned{G,X}$ and every $F\in\mathcal{F}_R$ in time $\Abs{G}^{\mathcal{O}(k)}$ and thus our claim follows. 
\end{proof}

\paragraph{Bags of unbounded size.}

We are now ready to generalize \autoref{@transformations} to the setting where the boundary graph $(H,X)$ arises as the torso of some bag in the structural decomposition provided by \autoref{@diskussion}.
For a detailed introduction into the theory of graph minors we refer the reader back to \autoref{@constraint}.
The definition of a $\Sigma$-decomposition is given in \autoref{@custodians} and \autoref{@diskussion} is the relevant structure theorem for this part. 

As a quick reminder, \autoref{@diskussion} provides, for every graph $G$ which excludes the shallow vortex grid as a minor, a tree decomposition where the torso of the bag at any vertex $t$ is a graph of bounded Euler-genus after the deletion of a small set of vertices called the \textit{apex set}.
Moreover, each node of this tree decomposition is equipped with such an apex set such that the adhesion set of any edge $dt,$ that is the intersection of the bags of $t$ and $d,$ in the decomposition tree avoids the apex set of $t$ in at most three vertices.
Additionally, the torso of $G$ at $t,$ after deleting the corresponding apex set, is embedded in a surface of bounded Euler-genus in such a way that the, at most three, vertices in the adhesion of $dt$ sit on a common face.

More formally, this means that we have to adress the following issue.
Let $G$ be a graph excluding some shallow vortex minor and let $(T,\beta )$ be the decomposition of $G$ provided by \autoref{@diskussion} where $r\in V(T)$ is the root of $T.$
Now let $t\in V(T)$ be some vertex with $\Abs{\beta (t)}>4\alpha (t)$ and let $d_1,\dots,d_{\ell}$ be the children of $t.$
For each $i\in[\ell]$ the intersection $\beta {(t)}\cap\beta {(d_i)}$ may contain up to three vertices which do not belong to the apex set of the torso $G_t$ of $G$ at $t.$
Moreover, the number $\ell$ is unbounded.
Hence, if we were to generalize our approach from \autoref{@transformations} directly we would need to produce a set $\mathcal{F}$ of matchings not only covering the boundary $\beta {(t)}\cap\beta {(t')}$ to the ancestor of $t,$ but also covering $\beta {(t)}\cap\beta {(d_i)}$ for all $i \in[\ell].$
As a result, we would be unable to bound the size of $\mathcal{F}.$

To get a better grip on the situation let us introduce some more definitions.

\begin{definition}[Branching of a Graph]
\label{@ambivalent} 
Let $k\in\Bbb{N}$ be some integer. 
Let $(G,X,\mathbf{p})$ be a labelled boundary graph with an \emph{apex set} $A\subseteq V(G).$ 
Let $(B_1,Y_1,\mathbf{p}),\dots,(B_{\ell},Y_{\ell},\mathbf{p})$ be boundary subgraphs of $(G,X,\mathbf{p}).$ 
We call $\mathcal{B}=((G,X),A,\mathbf{p},(B_1,Y_1),\dots,(B_{\ell},Y_{\ell}))$ a \emph{$k$-branching} (\emph{of $(G,X,\mathbf{p})$}) if 
\begin{itemize} 
\item For every $i\in[\ell],$ $(V(B_i)\setminus Y_i)\cap A=\emptyset.$ 
\item $\Abs{X\setminus A}\leq 3$ and, for all $i\in[\ell],$ $\Abs{Y_i\setminus A}\leq 3,$ 
\item $\Abs{X},\Abs{A}\leq k$ and, for all $i\in[\ell],$ $\Abs{Y_i}\leq k,$ 
\item if $i\neq j\in[\ell]$ then $V(B_i)\cap V(B_j)\subseteq Y_i\cap Y_j,$ 
\item if $i\neq j\in[\ell]$ then neither $Y_i\setminus A\subseteq Y_j\setminus A$ nor $Y_j\setminus A\subseteq Y_i\setminus A$ holds, 
\item if $G_{\mathcal{B}}$ is the subgraph obtained from $G-(\bigcup_{i\in[\ell]}(V(B_i)\setminus Y_i))$ by turning every set $Y_i,$ $i\in[\ell],$ and the set $X$ into cliques, then $G_{\mathcal{B}}-A$  
has a drawing $\gamma,$ without crossings, on some surface $\Sigma,$ without boundary, of Euler-genus at most $k,$  
{ \item the vertices of $X$ are incident to a face of $\gamma$ and, for every $i\in[\ell],$ the vertices of $Y_i\setminus A$ are incident to a face of $\gamma.$
 
}
\end{itemize} 
See \autoref{@licentious} for a visualisation of \autoref{@ambivalent}.

\begin{figure}[th]
\begin{center} 
~~\scalebox{1}{\includegraphics{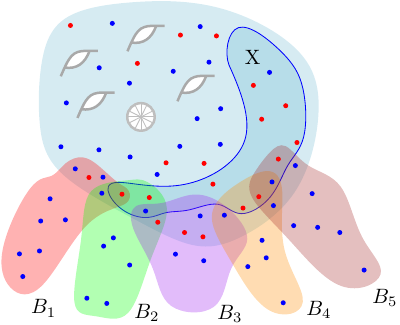}} 
\end{center} 
\vspace{-0mm} 
\caption{A visualization of \autoref{@ambivalent}. The \red{red} vertices depict the vertices in $A.$ 
For each $B_{i},$ the boundary $Y_{i}$ consists of the vertices in the intersection of the  \textcolor{celestialblue}{light blue}  
territory and the territory corresponding in $B_{i}.$} 
\label{@licentious}
\end{figure}

The $(B_i,Y_i,\mathbf{p})$ are called the \emph{branches} of $\mathcal{B}.$ 
Let $(G,X,\mathbf{p})$ be a labelled boundary graph and let $\mathcal{B}=((G,X),\mathbf{p},A,(B_1,Y_1),\dots,(B_{\ell},Y_{\ell}))$ be a $k$-branching of $(G,X,\mathbf{p}).$ 
We call a function $\SignPost\colon [\ell]\rightarrow 2^{\Bbb{Z}[x]}$ the \emph{sign post} of $\mathcal{B}$ if for every $i\in[\ell]$ we have $\SignPost(i)=\GenerateMatchingsBoundary{(B_i,Y_i,\mathbf{p})}.$
\end{definition} 
 
Our goal is to compute the family of all generating functions $\GenerateMatchingsBoundary{(G,X,\mathbf{p})}$ for the labelled boundary graph $(G,X,\mathbf{p})$ in the case where we are given a branching $\mathcal{B}$ together with its sign post $\SignPost.$

\paragraph{Matchgates.}
To manage the problem of the unbounded number of non-trivial residual boundaries in a branching, we employ the technique from \cite{StraubTW14count} which is similar to Valiant's ``matchgates'' \cite{valiant2008holographic}.
 matchgates were applied for the problem of counting perfect matchings in single crossing minor-free graphs in \cite{Curticapean14count}.
However, for our purpose of producing the entire generating function the gadgets from \cite{StraubTW14count} appear to be more suited.

If $(G,X,\mathbf{p})$ is some labelled boundary graph and $\mathcal{B}=((G,X),A,\mathbf{p},(B_1,Y_1),\dots,(B_{\ell},Y_{\ell}))$ is a $k$-branching of $(G,X,\mathbf{p}),$ then what remains of any $B_i$ in $G_{\mathcal{B}}-A$ is at most a triangle and this triangle bounds a face of $\gamma.$
Let $F$ be matching in $G$ such that 
\begin{itemize}
\item every edge in $F$ covers some vertex in $A\cup X,$ 
\item every vertex of $A\setminus X$ is covered by an edge of $F,$ and
\item there exists a perfect matching $M$ of $G-V(F)-X.$
\end{itemize}
Now define the {\emph{$F$-reduced branching}} 
$\mathcal{B}_F\coloneqq  ((G_F,\emptyset),\emptyset,\mathbf{p},(B_{1,F},Y_{1,F}),\dots,(B_{\ell,F},Y_{\ell,F})),$
where
\begin{align*} 
G_F\coloneqq & G-V(F)-X\text{,}\\ 
B_{i,F}\coloneqq & B_i-V(F)-X\text{ for all }i\in[\ell]\text{, and}\\ 
Y_{i,F}\coloneqq & Y_i\setminus(V(F)\cup X)\text{ for all }i\in[\ell]. 
\end{align*}
Note that every $(B_{i,F},Y_{i,F},\mathbf{p})$ {still} is a labelled boundary subgraph of $(G-V(F)-X,\emptyset,\mathbf{p}).$
Suppose we are given, for every $i\in[\ell],$ the set $\GenerateMatchingsBoundary{(B_{i,F},Y_{i,F},\mathbf{p})}.$
Let us consider some $i\in[\ell].$
There are eight possible cases that might arise.
In each of these cases, we replace the entire graph $B_{i,F}$ by a certain graph $J_{i,F}$ together with a labelling $λ_{i,F}$ of the edges of $J_{i,F}$ with elements from the field of quotients over $\Bbb{Z}[x]$, i.\@e.\@  $\Bbb{Z}(x).$
{These polynomial fractions will then be used to encode the generating functions of $(B_{i,F},Y_{i,F},\mathbf{p})$ for a skew-symmetric matrix.
The Pfaffian of said matrix will be used (via \autoref{@wittgensteiris}) to compute the generating function for the perfect matchings in $(G_F,\emptyset,\mathbf{p}).$}
Formally, we will replace the subgraph $B_{i,F}$ by the graph $J_{i,F},$ thereby producing some labelled boundary graph $(G',\emptyset,\mathbf{p}'),$ and adjust the function $\mathbf{p}'$ by setting $\mathbf{p}'(e)\coloneqq  λ_{i,F}(e)$ for all $e\in E(J_{i,F}),$ while $\mathbf{p}'$ equals $\mathbf{p}$ on all other edges of $G'.$

We now describe the \emph{matchgates} that we will be using.
In the non-trivial cases, we will attribute the respective  matchgate with a property called ``representativeness''.
In this property we subsume the different ways in which perfect matchings of $G'$ can interact with the  matchgate.
The role of the representativeness is to ensure that $(J_{i,F},λ_{i,F})$ correctly encodes the generating function of the boundary graph it represents.

\begin{description} 
\item[\textbf{Case 1}: $Y_{i,F}=\emptyset.$] In this case $J_{i,F}$ is just the empty graph without any vertices and $λ_{i,F}$ is empty as well. 
 
\item[\textbf{Case 2}: $Y_{i,F}=\Set{a}.$] In this case $J_{i,F}$ consists exactly of the vertex $a$ and $λ_{i,F}$ is empty. 
 
\item[\textbf{Case 3:} $Y_{i,F}=\Set{a,b}$ and $\Abs{V(B_{i,F})}$ is even.] Here we define $J_{i,F}$ to be the graph with vertex set $\Set{a,b,u,v}$ together with the edges $\Set{{au,uv,vb}},$ where $u$ and $v$ are newly introduced vertices. 
Note that, since $\Abs{V(B_{i,F})}$ is even, every perfect matching of $G_F$ either covers both $a$ and $b$ with edges of $B_{i,F},$ or none of them. 
Moreover, after replacing $B_{i,F}$ with $J_{i,F},$ every perfect matching of the resulting graph $G_F'$ either contains the edges $au$ and $bv,$ or the edge $uv,$ let us call this fact the \emph{representativeness} of $J_{i,F}.$ 
The labels $λ_{i,F}$ should now express two possible states; the contribution of the edges $au$ and $bv$ to the Pfaffian of the skew-symmetric matrix we want to construct should equal the labelled generating function of all perfect matchings of $B_{i,F},$ let us call this function $\mathsf{p}_{\emptyset}$ since no vertex from $\Set{a,b}$ is matched outside of $B_{i,F},$ while the contribution of the edge $uv$ should equal the labelled generating function of all perfect matchings of $B_{i,F}-a-b,$ we denote this function by $\mathsf{p}_{ab}.$ 
Finally, we set $λ_{i,F}(au)\coloneqq  \mathsf{p}_{\emptyset},$ $λ_{i,F}(bv)\coloneqq  1,$ and $λ_{i,F}\coloneqq  \mathsf{p}_{ab}.$ 
 
\item[\textbf{Case 4:} $Y_{i,F}=\Set{a,b}$ and $\Abs{V(B_{i,F})}$ is odd.] This case is similar to Case 3 with the difference that every perfect matching of $G_F$ must match exactly one of the two vertices $a$ and $b$ within $B_{i,F}$ while the other one cannot be matched within $B_{i,F}$ at the same time. 
To model this with our  matchgate $J_{i,F}$ we define its vertex set to be $\Set{a,b,u},$ where $u$ is a newly introduced vertex. 
The edge set is defined to be $E(J_{i,F})\coloneqq \Set{au,bu}.$ 
For each $x\in\Set{a,b}$ let $\mathsf{p}_x$ be the labelled generating function for all perfect matchings of $B_{i,F}-x$ and let $H_F'$ be the graph obtained from $H_F$ by replacing $B_{i,F}$ with $J_{i,F}.$ 
In this case the \emph{representativeness} of $J_{i,F}$ is the fact that every perfect matching of $H_F'$ must contain exactly one of the edges of $J_{i,F}.$ 
 
\item[\textbf{Case 5:} $Y_{i,F}=\Set{a,b,c},$ $\Abs{V(B_{i,F})}$ is even, and $B_{i,F}$ has a perfect matching.] For each $S\subseteq \Set{a,b,c},$ let $\mathsf{p}_S$ be the labelled generating function of all perfect matchings of $B_{i,F}-S.$  
Since $\Abs{V(B_{i,F})}$ is even, any matching $F'$ with $F'\in\Aligned{B_{i,F},Y_{i,F}}$ must be of odd size and thus $\mathsf{p}_S\neq 0$ is possible if and only if $\Abs{S}$ is even. 
In particular, we have that $\mathsf{p}_{\emptyset}\neq 0$ since $B_{i,F}$ has a perfect matching. 
We introduce three new vertices $u,$ $v,$ and $w$ and define $J_{i,F}$ and $λ_{i,F}$ as depicted in \textbf{Case 5} of \autoref{@transience}. 
Now let $G_F'$ be the graph obtained from $G_F$ by replacing $B_{i,F}$ with $J_{i,F}.$  
In this case the \emph{representativeness} of $J_{i,F}$ is slightly more complicated than in previous cases. 
The new vertices $u,$ $v,$ and $w$ must be covered by every perfect matching of $H_F'.$ 
Hence, every perfect matching $M$ of $G_F'$ contains exactly two edges of $J_{i,F}$ or three. 
This means there always exists a set $S\in\Set{\emptyset,\Set{a,b},\Set{a,c},\Set{b,c}}$ such that $M$ contains a perfect matching of $J_{i,F}-S.$ 
Moreover, in case $S=\emptyset,$ $M$ must contain the edges $au,$ $bv,$ and $cw$ whose labels under $λ_{i,F},$ in total, multiply to $\mathsf{p}_{\emptyset}.$ 
If $S=\Set{a,b}$ then $M$ contains $uv$ and $cw,$ evaluating to $\mathsf{p_{ab}},$ similarly $S=\Set{a,c}$ forces the labels of the edges of $M$ within $J_{i,F}$ to multiply to $\mathsf{p}_{ac}.$ 
Finally, $M$ contains the edge $vw$ if and only if $M$ contains the edge $au$ and thus, since $\mathsf{p}_{\emptyset}\neq 0,$ we obtain $λ_{i,F}(au)\cdot λ_{i,F}(vw)=\mathsf{p}_{bc}.$ 
 
\item[\textbf{Case 6:} $Y_{i,F}=\Set{a,b,c},$ $\Abs{V(B_{i,F})}$ is even, and $B_{i,F}$ has no perfect matching.] 
As for the above case, or each $S\subseteq \Set{a,b,c}$ let $\mathsf{p}_S$ be the labelled generating function of all perfect matchings of $B_{i,F}-S.$  
Since $\Abs{V(B_{i,F})}$ is even, any matching $F'$ with $F'\in\Aligned{B_{i,F},Y_{i,F}}$ must be of odd size and thus $\mathsf{p}_S\neq 0$ is possible if and only if $\Abs{S}$ is even. 
Moreover, since $B_{i,F}$ has no perfect matching we know $\mathsf{p}_{\emptyset}=0.$ 
We introduce a single new vertex $w$ and define $J_{i,F}$ and $λ_{i,F}$ as depicted in \textbf{Case 6} of \autoref{@transience}. 
Now let $G_F'$ be the graph obtained from $G_F$ by replacing $B_{i,F}$ with $J_{i,F}.$ 
The \emph{representativeness} of this case is the fact that the new vertex $w$ must be covered by every perfect matching of $G_F'.$ 
Hence, every such perfect matching must contain exactly one of the edges of $J_{i,F}.$ 
Since for every $x\in\Set{a,b,c}$ we have $λ_{i,F}(wx)=\mathsf{p}_{\Set{a,b,c}\setminus \Set{x}},$ the perfect matchings of $B_{i,F}-(\Set{a,b,c}\setminus\Set{x})$ are correctly represented by $(J_{i,F},λ_{i,F}).$ 
 
\item[\textbf{Case 7:} $Y_{i,F}=\Set{a,b,c},$ $\Abs{V(B_{i,F})}$ is odd, and $B_{i,F}-a$ has a perfect matching.] 
In this case for every $S\subseteq \Set{a,b,c},$ let $\mathsf{p}_S$ be the labelled generating function of all perfect matchings of $B_{i,F}-S.$  
Since $\Abs{V(B_{i,F})}$ is odd, any matching $F'$ with $F'\in\Aligned{B_{i,F},Y_{i,F}}$ must be of even size and thus $\mathsf{p}_S\neq 0$ is possible if and only if $\Abs{S}$ is odd. 
In particular, we have $\mathsf{p}_a\neq 0$ since $B_{i,F}-a$ has a perfect matching. 
We introduce two new vertices, $v$ and $w,$ and define $J_{i,F}$ and $λ_{i,F}$ as depicted in \textbf{Case 7} of \autoref{@transience}. 
Next let $G_F'$ be the graph obtained from $G_F$ by replacing $B_{i,F}$ with $J_{i,F}.$ 
Observe that every perfect matching of $G_F'$ must cover the vertices $v$ and $w,$ for this the two cases are possible; 
Let $M$ be a perfect matching of $G_F',$ then either $vw\in M,$ or there exist $x,y\in\Set{a,b,c}$ such that $vx,wy\in M.$ 
In the first case all three vertices of $\Set{a,b,c}$ are matched by $M$ with vertices outside of $J_{i,F},$ while in the second case only the single remaining vertex of $\Set{a,b,c}\setminus\Set{x,y}$ is matched to a vertex outside of $J_{i,F}.$ 
If $M$ contains the edge $aw,$ then $M$ must match $v$ within $J_{i,F}$ and the only way to do so is via the edge $bv.$ 
Hence, we obtain $λ_{i,F}(aw)\cdot λ_{i,F}(bv)=\mathsf{p}_c,$ correctly representing the case where exactly $c$ is matched by $M$ with a vertex not in $J_{i,F}.$ 
In case $av\in M$ we must have $cw\in M$ and $λ_{i,F}(av)\cdot λ_{i,F}(cw)=\mathsf{p}_b,$ and if $bv,cw\in M$ we have $λ_{i,F}(bv)\cdot λ_{i,F}(cw)=\mathsf{p}_a.$ 
This leaves only the case where $vw\in M$ which means all three vertices $a,b,$ and $c$ must be matched outside of $J_{i,F}$ and thus the contribution of $M\cap E(J_{i,F})$ in this case is exactly $λ_{i,F}(vw)=\mathsf{p}_{abc}.$ 
We refer to these observations as the \emph{representativeness} of $(J_{i,F},λ_{i,F})$ in this case. 
 
\item[\textbf{Case 8:} $Y_{i,F}=\Set{a,b,c},$ $\Abs{V(B_{i,F})}$ is odd, and $B_{i,F}-a$ has no perfect matching.] 
As in the previous cases, for each $S\subseteq \Set{a,b,c}$ let $\mathsf{p}_S$ be the labelled generating function of all perfect matchings of $B_{i,F}-S$ 
Since $\Abs{V(B_{i,F})}$ is odd, any matching $F'$ with $F'\in\Aligned{B_{i,F},Y_{i,F}}$ must be of even size and thus $\mathsf{p}_S\neq 0$ is possible if and only if $\Abs{S}$ is odd. 
Moreover, we have $\mathsf{p}_a= 0$ since $B_{i,F}-a$ has no perfect matching. 
We introduce two new vertices, $v$ and $w,$ and define $J_{i,F}$ and $λ_{i,F}$ as depicted in \textbf{Case 8} of \autoref{@transience}. 
Let $G_F'$ be the graph obtained from $G_F$ by replacing $B_{i,F}$ with $J_{i,F}.$ 
We complete the introduction of the  matchgates by discussing the \emph{representativeness} of $(J_{i,F},λ_{i,F})$ in this case. 
Notice that for every perfect matching $M$ of $G_F',$ the vertex $v$ must either be matched with $a$ or with $w.$ 
In case $vw\in M$ no other edge of $J_{i,F}$ can be contained in $M$ and thus all three vertices, $a,$ $b,$ and $c$ must be matched outside of $J_{i,F}.$ 
With $λ_{i,F}(vw)=\mathsf{p}_{abc}$ this is correctly represented. 
Hence, we may assume $av\in M.$ 
This means we must either have $bw$ or $cw$ in $M.$ 
Since $λ_{i,F}(a)=1,$ the total contribution of the two edges of $J_{i,F}$ in $M$ is exactly as intended in both cases.
\end{description}

\begin{figure} 
\centering 
\begin{tikzpicture}[scale=1] 
\pgfdeclarelayer{background} 
\pgfdeclarelayer{foreground} 
\pgfsetlayers{background,main,foreground}

\node [v:ghost] (G) {}; 
 
\node [v:ghost,position=180:57mm from G] (C1) {}; 
\node [v:ghost,position=180:21mm from G] (C2) {}; 
\node [v:ghost,position=0:21mm from G] (C3) {}; 
\node [v:ghost,position=0:57mm from G] (C4) {}; 
 
\node [v:main,position=90:5mm from C1] (w1) {}; 
\node [v:main,position=210:5mm from C1] (u1) {}; 
\node [v:main,position=330:5mm from C1] (v1) {}; 
 
\node[v:ghost,position=270:3mm from u1] (lu1) {\small $u$}; 
\node[v:ghost,position=13:3mm from v1] (lv1) {\small $v$}; 
\node[v:ghost,position=150:3mm from w1] (lw1) {\small $w$};

\node [v:main,position=90:15mm from C1] (c1) {}; 
\node [v:main,position=210:15mm from C1] (a1) {}; 
\node [v:main,position=330:15mm from C1] (b1) {}; 
 
\node[v:ghost,position=270:3mm from a1] (la1) {\small $a$}; 
\node[v:ghost,position=270:3mm from b1] (lb1) {\small $b$}; 
\node[v:ghost,position=90:3mm from c1] (lc1) {\small $c$}; 
 
\node [v:ghost,position=270:17mm from C1] {\textbf{Case 5}}; 
 
\draw [e:main] (w1) to (u1); 
\draw [e:main] (u1) to (v1); 
\draw [e:main] (v1) to (w1); 
 
\draw [e:main] (w1) to (c1); 
\draw [e:main] (v1) to (b1); 
\draw [e:main] (u1) to (a1); 
 
\node[v:ghost,position=55:5.3mm from a1] (lea1) {\small $\mathsf{p}_{\emptyset}$}; 
\node[v:ghost,position=173:5mm from b1] (leb1) {\small $1$}; 
\node[v:ghost,position=290:5mm from c1] (lec1) {\small $1$}; 
\node[v:ghost,position=270:9.5mm from w1] (leuv1) {\small $\mathsf{p}_{ab}$}; 
\node[v:ghost,position=27.5:10.5mm from u1] (levw1) {$\frac{\mathsf{p}_{bc}}{\mathsf{p}_{\emptyset}}$}; 
\node[v:ghost,position=157.5:10.5mm from v1] (lewu1) {\small $\mathsf{p}_{ac}$}; 
 
\node [v:main,position=90:5mm from C2] (w2) {}; 
 
\node[v:ghost,position=157.5:3mm from w2] (lw2) {\small $w$}; 
 
\node [v:main,position=90:15mm from C2] (c2) {}; 
\node [v:main,position=210:15mm from C2] (a2) {}; 
\node [v:main,position=330:15mm from C2] (b2) {}; 
 
\node[v:ghost,position=270:3mm from a2] (la2) {\small $a$}; 
\node[v:ghost,position=270:3mm from b2] (lb2) {\small $b$}; 
\node[v:ghost,position=90:3mm from c2] (lc2) {\small $c$}; 
 
\node [v:ghost,position=270:17mm from C2] {\textbf{Case 6}}; 
 
\draw [e:main] (w2) to (a2); 
\draw [e:main] (w2) to (b2); 
\draw [e:main] (w2) to (c2); 
 
\node[v:ghost,position=62:8.5mm from a2] (lea2) {\small $\mathsf{p}_{bc}$}; 
\node[v:ghost,position=155:8.5mm from b2] (leb2) {\small $\mathsf{p}_{ac}$}; 
\node[v:ghost,position=300:6mm from c2] (lec2) {\small $\mathsf{p}_{ab}$}; 
 
\node [v:main,position=90:5mm from C3] (w3) {}; 
\node [v:main,position=270:2mm from C3] (v3) {}; 
 
\node[v:ghost,position=270:3mm from v3] (lv3) {\small $v$}; 
\node[v:ghost,position=22.5:3mm from w3] (lw3) {\small $w$}; 
 
\node [v:main,position=90:15mm from C3] (c3) {}; 
\node [v:main,position=210:15mm from C3] (a3) {}; 
\node [v:main,position=330:15mm from C3] (b3) {}; 
 
\node[v:ghost,position=270:3mm from a3] (la3) {\small $a$}; 
\node[v:ghost,position=270:3mm from b3] (lb3) {\small $b$}; 
\node[v:ghost,position=90:3mm from c3] (lc3) {\small $c$}; 
 
\node [v:ghost,position=270:17mm from C3] {\textbf{Case 7}}; 
 
\draw [e:main] (w3) to (c3); 
\draw [e:main] (w3) to (v3); 
\draw [e:main] (v3) to (a3); 
\draw [e:main] (v3) to (b3); 
\draw [e:main] (a3) to (w3); 
 
\node[v:ghost,position=5:8mm from a3] (leav3) {\small $\mathsf{p}_{b}$}; 
\node[v:ghost,position=65:9mm from a3] (leaw3) {$\frac{\mathsf{p}_c}{\mathsf{p}_a}$}; 
\node[v:ghost,position=173:6.5mm from b3] (leb3) {\small $\mathsf{p}_a$}; 
\node[v:ghost,position=290:5mm from c3] (lec3) {\small $1$}; 
\node[v:ghost,position=315:5mm from w3] (levw3) {\small $\mathsf{p}_{abc}$}; 
 
\node [v:main,position=90:5mm from C4] (w4) {}; 
\node [v:main,position=225:9mm from w4] (v4) {}; 
 
\node[v:ghost,position=135:3mm from v4] (lv4) {\small $v$}; 
\node[v:ghost,position=22.5:3mm from w4] (lw4) {\small $w$}; 
 
\node [v:main,position=90:15mm from C4] (c4) {}; 
\node [v:main,position=210:15mm from C4] (a4) {}; 
\node [v:main,position=330:15mm from C4] (b4) {}; 
 
\node[v:ghost,position=270:3mm from a4] (la4) {\small $a$}; 
\node[v:ghost,position=270:3mm from b4] (lb4) {\small $b$}; 
\node[v:ghost,position=90:3mm from c4] (lc4) {\small $c$}; 
 
\node [v:ghost,position=270:17mm from C4] {\textbf{Case 8}}; 
 
\draw [e:main] (w4) to (c4); 
\draw [e:main] (w4) to (v4); 
\draw [e:main] (v4) to (a4); 
\draw [e:main] (w4) to (b4); 
 
\node[v:ghost,position=20:5mm from a4] (leav4) {\small $1$}; 
\node[v:ghost,position=120:8mm from b4] (leb4) {\small $\mathsf{p}_c$}; 
\node[v:ghost,position=295:5.3mm from c4] (lec4) {\small $\mathsf{p}_b$}; 
\node[v:ghost,position=270:5mm from w4] (levw4) {\small $\mathsf{p}_{abc}$}; 
 
\begin{pgfonlayer}{background} 
\end{pgfonlayer} 
\end{tikzpicture} 
\caption{We assume that $(G,X,\mathbf{p})$ is some labelled boundary graph and, moreover, let $\mathcal{B}=((G,X),A,\mathbf{p},(B_1,Y_1),\dots,(B_{\ell},Y_{\ell}))$ be a $k$-branching of $(G,X,\mathbf{p}).$ Let $F$ be an extendable matching such that every vertex of $A\setminus X$ is covered by an edge of $F,$ every edge in $F$ has an endpoint in $X\cup A,$ and $G-X-V(F)$ has a perfect matching. 
Let $i\in[\ell].$ 
The figure shows the four  matchgates for the case where $Y_i\setminus (V(F)\cup X)=\Set{a,b,c}$ for some branch $(B_i,Y_i,\mathbf{p})$ of $\mathcal{B}.$ In case $B_i-(V(F)\cup X)$ has an even number of vertices (\textbf{Cases} 5 and 6) and an odd number of vertices (\textbf{Cases} 7 and 8). 
If $\Abs{B_i-(V(F)\cup X)}$ is even we use \textbf{Case 5} except in the situation where $\mathsf{p}_{\emptyset}=0,$ then we use \textbf{Case 6} instead. 
Similarly, in case $\Abs{B_i-(V(F)\cup X)}$ is odd, we use \textbf{Case 7} except if $\mathsf{p}_{a}=0,$ then we use \textbf{Case 8} instead.} 
\label{@transience}
\end{figure}
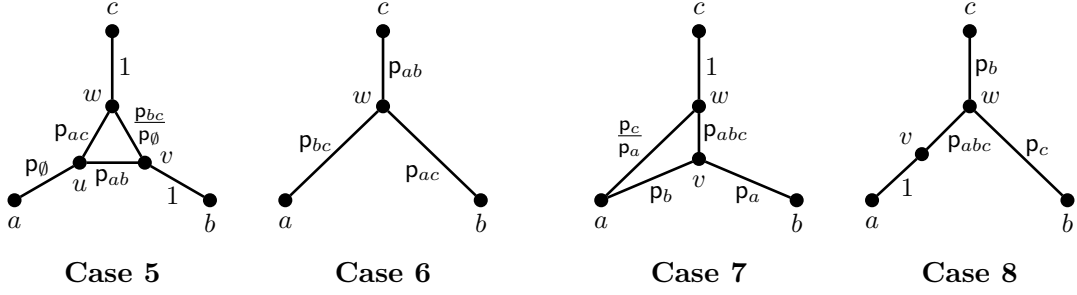

We can now employ the  matchgates to produce the labelled generating functions for a boundary graph with a $k$-branching, given the corresponding sign post.

\begin{lemma}\label{@incisiveness}
Let $k\in\Bbb{N}$ be some integer.
Let $(G,X,\mathbf{p})$ be a labelled boundary graph with an additional set $A\subseteq V(G)$ such that $\Abs{X},\Abs{A}\leq k.$
If we are given a $k$-branching $\mathcal{B}=((G,X),A,\mathbf{p},(B_1,Y_1),\dots,(B_{\ell},Y_{\ell}))$ together with the sign post $\SignPost$ of $\mathcal{B},$ then the set $\GenerateMatchingsBoundary{(G,X,\mathbf{p})}$ can be computed in time $\Abs{G}^{\mathcal{O}(k)}.$
\end{lemma}

\begin{proof}
Note that $(G,A\cup X,\mathbf{p})$ is also a labelled boundary graph.
Let $\mathcal{F}$ be the collection of all matchings $F$ in $G$ such that $F\in\Aligned{G,A\cup X}$ and $A\subseteq V(F).$
Since $\Abs{A\cup X}=\O(k)$ we have $\Abs{\mathcal{F}}\in\Abs{G}^{\mathcal{O}(k)}$ and $\mathcal{F}$ can be found in time $\Abs{G}^{\mathcal{O}(k)}.$
Similarly, we can find the set $\Aligned{G,X}.$
Now every set $F\in\mathcal{F}$ can be covered by $\ell+2$ sets $F_i$ as follows.
\begin{enumerate}[label=\textit{\roman*})] 
\item Let $F_1$ be the set of all edges in $F$ with at least one endpoint in $X,$ 
\item let $F_2$ be the set of all edges in $F$ with at most one endpoint in $\bigcup_{j\in[\ell]}Y_j,$ and 
\item for every $i\in[3,\ell+2]$ let $F_i$ be the set of all edges in $F$ with both endpoints in $B_{i-2}.$
\end{enumerate}
It follows that for every $F$ we have $F_1\in\Aligned{G,X}$ and $F_2=F\setminus(\bigcup_{i\in[3,\ell+2]}F_i).$

We now fix some $F\in\mathcal{F}.$
Let also $i\in[\ell].$ Then $\Abs{Y_i\setminus V(F)}\leq 3,$ but it is not necessary equal to zero.
Hence, there may be $\mathcal{O}({\Abs{G}^3})$ matchings $W$ such that $F_{i+2}\subseteq W$ and $W\in\Aligned{B_i,Y_i}.$
For every $i\in[\ell]$ let $\mathcal{W}_{i,F}$ be the collection of all matchings $W$ such that $F_{i+2}\subseteq W$ and $W\in\Aligned{B_i,Y_i}.$
Moreover, for every $S\subseteq Y_i\setminus V(F)$ we denote by $\mathcal{W}_{i,F,S}$ the subset of $\mathcal{W}_{i,F}$ such that for all $W\in \mathcal{W}_{i,F,S}$ we have $Y_i\setminus V(W)=S$ and $V(W)\cap X=V(F)\cap X.$
This means that the members of $\mathcal{W}_{i,F,S}$ are exactly the matchings of the members of $\Aligned{B_i,Y_i}$ that contain $F_{i+2}$ and expose the set $S\cup(X\setminus V(F)).$
We define
\begin{align*} 
\mathsf{p}^{i,F}_S\coloneqq  \sum_{W\in\mathcal{W}_{i,F,S}} (\prod_{\substack{e\in F_{i+2}\\\mathbf{p}(e)\neq 0}}\mathbf{p}(e))^{-1}\cdot\GenerateMatchingsBoundaryF{(B_i,Y_i,\mathbf{p})}{W}.
\end{align*}
Note that, since we are given the sign post of $\mathbf{B}$ and $\Abs{Y_i\setminus V(F)}\leq 3,$ we compute all $\mathsf{p}_S^{i,F}$ for all $i$ and $S$ in time $\Abs{V{G}}^{\mathcal{O}(k)}.$
Moreover, $\mathsf{p}_S^{i,F}$ is exactly the labelled generating function of all perfect matchings in the labelled graph $(B_i-V(F)-X-S,\mathbf{p}).$
Note that, for fixed $i\in[\ell]$ and $F\in\mathcal{F},$ the $\mathsf{p}_S^{i,F}$ can be used to replace the functions $\mathsf{p}_S$ used in the construction of the  matchgates and we maintain their representativeness.
We will make use of this observation to replace each $(B_i,Y_i)$ by some  matchgate in the graph $G-V(F)-X$ in order to produce a labelled graph of bounded Euler-genus.

The next steps of this proof are as follows:
We first formally describe the construction and discuss its validity.
Then we use \autoref{@wittgensteiris} to produce the labelled generating function of the resulting graph with its adjusted labeling.
Finally, we show that the resulting labelled generating functions can be used to correctly produce $\GenerateMatchingsBoundaryF{(G,X,\mathbf{p})}{R}$ for all $R\in\Aligned{G,X}.$

For some fixed $F\in\mathcal{F},$ recall according to the definition of the $F$-reduced branching $\mathcal{B}_F\coloneqq  ((G_F,\emptyset),\emptyset,\mathbf{p},(B_{1,F},Y_{1,F}),\dots,(B_{\ell,F},Y_{\ell,F})).$
Where
\begin{align*} 
G_F\coloneqq & G-V(F)-X\text{,}\\ 
B_{i,F}\coloneqq & B_i-V(F)-X\text{ for all }i\in[\ell]\text{, and}\\ 
Y_{i,F}\coloneqq & Y_i\setminus(V(F)\cup X)\text{ for all }i\in[\ell]. 
\end{align*}
Observe that, with $A\subseteq V(F)$ and because of the definition of $k$-branchings, the graph $G_F$ 
admits a drawing $\gamma_F$ without crossings in a surface $\Sigma$ (without boundary) of Euler-genus at most $k,$ and for every $i\in[\ell],$ the vertices of $Y_{i,F}$ lie on a common face of $\gamma_F.$
Now let $G_F'$ be obtained from $G_F$ by replacing, for every $i\in[\ell],$ the subgraph $\InducedSubgraph{G_F}{Y_{i,F}}$ with the appropriate  matchgate while using the functions $\mathsf{p}_S^{i,F}$ for the functions necessarily to produce the labellings of the  matchgates.
Let $\mathbf{p}'$ be the resulting labeling of $G_F'.$
Note that, since none of the $Y_{i,F}$ is contained in some other $Y_{j,F}$ by definition, and the fact that the  matchgates only introduce new vertices and never delete the vertices of $Y_{i,F},$ the different $Y_{i,F}$'s do not interfere with the introduction of  matchgates for other $Y_{j,F}$'s and thus $G_F'$ is well-defined.
Moreover, as the vertices of $Y_{i,F}$ lie on a common face of $\gamma_F$ and all  matchgates are planar with the vertices of $Y_{i,F}$ lying on the outer face, $G_F'$ also admits a drawing in $\Sigma$ without crossings, thus has Euler-genus at most $k.$
\medskip

Let $\GenerateMatchings{G_F',\mathbf{p}'}$ be the labelled generating function of the labelled graph $(G_F',\mathbf{p}').$
We claim that $\GenerateMatchings{G_F',\mathbf{p}'} = \GenerateMatchings{G_F,\mathbf{p}}.$
If this holds, we may use \autoref{@wittgensteiris} on $(G_F',\mathbf{p})$ to produce $\GenerateMatchings{G_F,\mathbf{p}}$ in time $\mathcal{O}_{k}(\Abs{G}^{\O(1)}).$
For each $i\in[\ell]$ and every $M\in\Perf{G_F'},$ let us denote by $S_{i,F,M}$ the set of all vertices of $Y_{i,F}$ which are \textbf{not} covered by an edge of the  matchgate $J_{i,F}.$
Recall that, by the representativeness of our  matchgates, we have 
\begin{align*}
\prod_{e\in M\cap E(J_{i,F})}\mathbf{p}'(e)=\mathsf{p}_{S_{i,F,M}^{i,F}}.
\end{align*}
Finally, let us denote by $M^-$ the set $M\setminus (\bigcup_{i\in[\ell]}E(J_{i,F})).$
Then
\begin{align*} 
\sum_{M\in\Perf{G_F'}}\mathbf{p}'(M) =& \sum_{M\in\Perf{G_F'}} \prod_{e\in M}\mathbf{p}'(e)\\ 
=& \sum_{M\in\Perf{G_F'}}(\prod_{e\in M^-}\mathbf{p}'(e))(\prod_{i\in[\ell]}\prod_{e\in M\cap E(J_{i,F})}\mathbf{p}'(e))\\ 
=& \sum_{M\in\Perf{G_F'}}(\prod_{e\in M^-}\mathbf{p}(e))(\prod_{i\in[\ell]}\prod_{e\in M\cap E(J_{i,F})}\mathbf{p}'(e))\\ 
=& \sum_{M\in\Perf{G_F'}}(\prod_{e\in M^-}\mathbf{p}(e))(\prod_{i\in[\ell]}\mathsf{p}_{S_{i,F,M}}^{i,F})\\ 
=& \sum_{M\in\Perf{G_F'}}(\prod_{e\in M^-}\mathbf{p}(e))(\prod_{i\in[\ell]}\GenerateMatchings{B_{i,F}-S_{i,F,M},\mathbf{p}})\\ 
=& \sum_{M\in\Perf{G_F'}}(\prod_{e\in M^-}\mathbf{p}(e))(\prod_{i\in[\ell]}\sum_{N\in\Perf{B_{i,F}-S_{i,F,M}}}\mathbf{p}(N))
\end{align*}
\begin{align*}
\phantom{\sum_{M\in\Perf{G_F'}}\mathbf{p}'(M)}=& \sum_{M\in\Perf{G_F'}}(\prod_{e\in M^-}\mathbf{p}(e))(\prod_{i\in[\ell]}\sum_{N\in\Perf{B_{i,F}-S_{i,F,M}}}\prod_{e\in N}\mathbf{p}(e))\\ 
=& \sum_{M\in\Perf{G_F'}}(\prod_{e\in M^-}\mathbf{p}(e))(\sum_{N\in\Perf{\bigcup_{i\in[\ell]}(B_{i,F}-S_{i,F,M})}}\prod_{e\in N}\mathbf{p}(e))\\ 
=& \sum_{M\in\Perf{G_F'}}\sum_{N\in\Perf{\bigcup_{i\in[\ell]}(B_{i,F}-S_{i,F,M})}}(\prod_{e\in M^-}\mathbf{p}(e))(\prod_{e\in N}\mathbf{p}(e))\\ 
=& \sum_{M\in\Perf{G_F}}\prod_{e\in M}\mathbf{p}(e)\\ 
=& \GenerateMatchings{G_F,\mathbf{p}}
\end{align*}
Now fix some $R\in\Aligned{G,X}$ and let $\mathcal{F}_R\subseteq\mathcal{F}$ be the set of all $F\in\mathcal{F}$ with $R=F_1.$
If follows that
\begin{align*} 
    \GenerateMatchingsBoundaryF{(G,X,\mathbf{p})}{R}&= (\prod_{e\in R}\mathbf{p}(e))\cdot \GenerateMatchings{G-V(R)-X,\mathbf{p}}\\
    &= (\prod_{e\in R}\mathbf{p}(e))\cdot \sum_{F\in\mathcal{F}_R}(\prod_{e\in F\setminus R}\mathbf{p}(e))\cdot\GenerateMatchings{G-V(F)-X,\mathbf{p}}\\
    &=\sum_{F\in\mathcal{F}_R} 
    \big(\prod_{e\in F}\mathbf{p}(e)\big)\cdot\GenerateMatchings{G-V(F)- X,\mathbf{p}}. 
\end{align*}
Since $\Abs{\mathcal{F}_R}\in\Abs{G}^{\mathcal{O}(k)}$ and, as we have seen in the discussion above, $\GenerateMatchings{G-V(F)-X,\mathbf{p}}$ can be computed in time $\mathcal{O}_{k}(\Abs{G}^{\O(1)})$ using \autoref{@wittgensteiris} on the labelled graph $(G_F',\mathbf{p}'),$ $\GenerateMatchingsBoundaryF{(G,X,\mathbf{p})}{R}$ can be computed in time $\Abs{G}^{\mathcal{O}(k)}.$
This, together with the bounded size of $\Aligned{G,X}$ completes our proof.
\end{proof}

\paragraph{Merging bags.}

The remaining piece for our algorithm is to merge the tables of subtrees that attach to a common boundary.
For the sake of simplicity, in all algorithms above we have always assumed that the boundaries of boundary subgraph are always distinct and never nested.
As the merging of subtrees can be performed iteratively on two subtrees, it suffices to consider the case where we have two boundary graphs with nested boundaries. Towards this, we will prove the following stronger version that contains the nested assumption as a special case. While being more general, it enjoys some notational symmetry
on the role of the sets $X_{1}$ and $X_{2},$ which makes more easy the presentation of the proof.

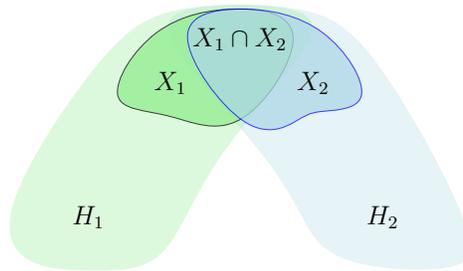
\begin{figure}[th]
\begin{center} 
\hspace{-3.4cm}\scalebox{0.84}{\tikzstyle{ipe stylesheet} = [
  ipe import,
  even odd rule,
  line join=round,
  line cap=butt,
  ipe pen normal/.style={line width=0.4},
  ipe pen heavier/.style={line width=0.8},
  ipe pen fat/.style={line width=1.2},
  ipe pen ultrafat/.style={line width=2},
  ipe pen normal,
  ipe mark normal/.style={ipe mark scale=3},
  ipe mark large/.style={ipe mark scale=5},
  ipe mark small/.style={ipe mark scale=2},
  ipe mark tiny/.style={ipe mark scale=1.1},
  ipe mark normal,
  /pgf/arrow keys/.cd,
  ipe arrow normal/.style={scale=7},
  ipe arrow large/.style={scale=10},
  ipe arrow small/.style={scale=5},
  ipe arrow tiny/.style={scale=3},
  ipe arrow normal,
  /tikz/.cd,
  ipe arrows, 
  <->/.tip = ipe normal,
  ipe dash normal/.style={dash pattern=},
  ipe dash dotted/.style={dash pattern=on 1bp off 3bp},
  ipe dash dashed/.style={dash pattern=on 4bp off 4bp},
  ipe dash dash dotted/.style={dash pattern=on 4bp off 2bp on 1bp off 2bp},
  ipe dash dash dot dotted/.style={dash pattern=on 4bp off 2bp on 1bp off 2bp on 1bp off 2bp},
  ipe dash normal,
  ipe node/.append style={font=\normalsize},
  ipe stretch normal/.style={ipe node stretch=1},
  ipe stretch TikZ-normal/.style={ipe node stretch=1},
  ipe stretch normal,
  ipe opacity 10/.style={opacity=0.1},
  ipe opacity 30/.style={opacity=0.3},
  ipe opacity 50/.style={opacity=0.5},
  ipe opacity 75/.style={opacity=0.75},
  ipe opacity opaque/.style={opacity=1},
  ipe opacity opaque,
]
\definecolor{red}{rgb}{1,0,0}
\definecolor{blue}{rgb}{0,0,1}
\definecolor{green}{rgb}{0,1,0}
\definecolor{yellow}{rgb}{1,1,0}
\definecolor{orange}{rgb}{1,0.647,0}
\definecolor{gold}{rgb}{1,0.843,0}
\definecolor{purple}{rgb}{0.627,0.125,0.941}
\definecolor{gray}{rgb}{0.745,0.745,0.745}
\definecolor{brown}{rgb}{0.647,0.165,0.165}
\definecolor{navy}{rgb}{0,0,0.502}
\definecolor{pink}{rgb}{1,0.753,0.796}
\definecolor{seagreen}{rgb}{0.18,0.545,0.341}
\definecolor{turquoise}{rgb}{0.251,0.878,0.816}
\definecolor{violet}{rgb}{0.933,0.51,0.933}
\definecolor{darkblue}{rgb}{0,0,0.545}
\definecolor{darkcyan}{rgb}{0,0.545,0.545}
\definecolor{darkgray}{rgb}{0.663,0.663,0.663}
\definecolor{darkgreen}{rgb}{0,0.392,0}
\definecolor{darkmagenta}{rgb}{0.545,0,0.545}
\definecolor{darkorange}{rgb}{1,0.549,0}
\definecolor{darkred}{rgb}{0.545,0,0}
\definecolor{lightblue}{rgb}{0.678,0.847,0.902}
\definecolor{lightcyan}{rgb}{0.878,1,1}
\definecolor{lightgray}{rgb}{0.827,0.827,0.827}
\definecolor{lightgreen}{rgb}{0.565,0.933,0.565}
\definecolor{lightyellow}{rgb}{1,1,0.878}
\definecolor{black}{rgb}{0,0,0}
\definecolor{white}{rgb}{1,1,1}
\begin{tikzpicture}[ipe stylesheet]
  \fill[lightblue, ipe opacity 30]
    (229.896, 376)
     .. controls (239.2293, 362.6667) and (259.6147, 341.3333) .. (271.1407, 325.3333)
     .. controls (282.6667, 309.3333) and (285.3333, 298.6667) .. (292, 292)
     .. controls (298.6667, 285.3333) and (309.3333, 282.6667) .. (320, 281.3333)
     .. controls (330.6667, 280) and (341.3333, 280) .. (349.3333, 281.3333)
     .. controls (357.3333, 282.6667) and (362.6667, 285.3333) .. (356.474, 302.6667)
     .. controls (350.2813, 320) and (332.5627, 352) .. (319.2293, 370.6667)
     .. controls (305.896, 389.3333) and (296.948, 394.6667) .. (281.8073, 397.3333)
     .. controls (266.6667, 400) and (245.3333, 400) .. (233.8073, 397.3333)
     .. controls (222.2813, 394.6667) and (220.5627, 389.3333) .. cycle;
  \fill[lightgreen, ipe opacity 30]
    (193.0073, 373.3333)
     .. controls (179.674, 357.3333) and (161.837, 330.6667) .. (155.5852, 313.3333)
     .. controls (149.3333, 296) and (154.6667, 288) .. (162.6667, 284)
     .. controls (170.6667, 280) and (181.3333, 280) .. (192, 280)
     .. controls (202.6667, 280) and (213.3333, 280) .. (221.3333, 288)
     .. controls (229.3333, 296) and (234.6667, 312) .. (246.2518, 329.3333)
     .. controls (257.837, 346.6667) and (275.674, 365.3333) .. (283.674, 377.3333)
     .. controls (291.674, 389.3333) and (289.837, 394.6667) .. (278.2518, 397.3333)
     .. controls (266.6667, 400) and (245.3333, 400) .. (230.2518, 397.3333)
     .. controls (215.1703, 394.6667) and (206.3407, 389.3333) .. cycle;
  \filldraw[fill=lightgreen, ipe opacity 75]
    (218.6667, 386.6667)
     .. controls (208, 378.6667) and (200, 365.3333) .. (201.3333, 358.6667)
     .. controls (202.6667, 352) and (213.3333, 352) .. (224, 349.3333)
     .. controls (234.6667, 346.6667) and (245.3333, 341.3333) .. (256.6667, 348)
     .. controls (268, 354.6667) and (280, 373.3333) .. (280, 384)
     .. controls (280, 394.6667) and (268, 397.3333) .. (255.3333, 397.3333)
     .. controls (242.6667, 397.3333) and (229.3333, 394.6667) .. cycle;
  \filldraw[draw=blue, fill=lightblue, ipe opacity 75]
    (293.3333, 386.6667)
     .. controls (282.6667, 394.6667) and (269.3333, 397.3333) .. (256.6667, 397.3333)
     .. controls (244, 397.3333) and (232, 394.6667) .. (232.6667, 384)
     .. controls (233.3333, 373.3333) and (246.6667, 354.6667) .. (257.3333, 347.3333)
     .. controls (268, 340) and (276, 344) .. (280.6667, 346.6667)
     .. controls (285.3333, 349.3333) and (286.6667, 350.6667) .. (292.6667, 351.3333)
     .. controls (298.6667, 352) and (309.3333, 352) .. (310.6667, 358.6667)
     .. controls (312, 365.3333) and (304, 378.6667) .. cycle;
  \node[ipe node, font=\large]
     at (281.778, 361.149) {$X_2$};
  \node[ipe node, font=\large]
     at (217.778, 361.149) {$X_1$};
  \node[ipe node, font=\large]
     at (235.045, 380.617) {$X_{1}\cap X_{2}$};
  \node[ipe node, font=\large]
     at (180.985, 300.618) {$H_{1}$};
  \node[ipe node, font=\large]
     at (312.985, 300.618) {$H_{2}$};
\end{tikzpicture}} 
\end{center} 
\vspace{-3mm} 
\caption{A visualization of \autoref{@penitentiary}.} 
\label{@pelagianism}
\end{figure}

\begin{lemma}\label{@penitentiary}
Let $k\in\Bbb{N}$ be some integer.
Let $(H_1,X_1,\mathbf{p})$ and $(H_2,X_2,\mathbf{p})$ be two labelled boundary graphs with $V(H_1)\cap V(H_2)\subseteq X_1\cap X_2$ and let $X \coloneqq  X_1\cap X_2$ where $\Abs{X_1},\Abs{X_2}\leq k.$
Suppose we are given $\GenerateMatchingsBoundary{(H_1,X_1,\mathbf{p})}$ and $\GenerateMatchingsBoundary{(H_2,X_2,\mathbf{p})}.$ Then $\GenerateMatchingsBoundary{(H_1\cup H_2,X,\mathbf{p})}$ can be computed in time $\Abs{V(H_1\cup H_2)}^{\mathcal{O}(k)}.$
\end{lemma}

\begin{proof}
Similarly to the arguments before, we start by computing the set $\Aligned{H_1\cup H_2,X}$ and extract the set $\mathcal{F}$ of all matchings $F$ with $F\in\Aligned{H_1\cup H_2,X}$ in time $\Abs{V(H_1\cup H_2)}^{\mathcal{O}(k)}.$
Now fix some $F\in\mathcal{F}$ and observe that for every $i\in[2],$ there might be several extendable pairs $F_i\in\Aligned{H_i,X_i}$ such that $F\cap E(H_i)\subseteq F_i$ and $X\setminus V(F)\subseteq X_i\setminus V(F_i\cup F).$
Moreover, note that for every such $F_i$ we have $V(F_i)\setminus V(F)\subseteq V(H_i)\setminus V(H_{3-i})$ by the definition of $X.$
Finally, any extendable pair $F_i\in\Aligned{H_i,F_i}$ such that $X\setminus V(F)\subsetneq X_i\setminus V(F_i\cup F)$ will not make any contribution to $\GenerateMatchings{H_1\cup H_2,X,\mathbf{p},F},$ since $\GenerateMatchings{H_i,X_i,\mathbf{p},F_i}$ counts all perfect matchings of $H_i-(X_i\setminus V(F_i))$ and therefore includes matchings that expose vertices of $X_i\setminus X.$ (See \autoref{@pelagianism} for a visualization of $(H_1,X_1,\mathbf{p})$ and $(H_2,X_2,\mathbf{p}).$)

Thus, for each $i\in[2]$ and every $F\in\mathcal{F}$ let $\mathcal{F}_{i,F}$ be the collection of all $F_i$ such that $F_i\in\Aligned{H_i,X_i},$ $F\cap E(H_i)\subseteq F_i,$ and $X_i\setminus V(F_i\cup F)=X\setminus V(F).$
Since $\Abs{X_i}\leq k$ we can be sure that $\Abs{\mathcal{F}_{i,F}}\in\Abs{V(H_i)}^{\mathcal{O}(k)}.$
We define
\begin{align*} 
\mathsf{p}_{i,F}\coloneqq  \sum_{F_i\in\mathcal{F}_{i,F}}(\prod_{e\in F\cap E(H_i)}\mathbf{p}(e))^{-1}\cdot \GenerateMatchingsBoundaryF{H_i,X_i,\mathbf{p}}{F_i}.
\end{align*}
Hence, $\mathsf{p}_{i,F}$ is the labelled generating function for the perfect matchings of $(H_i-(V(F\cap E(H_i))\cup X),\mathbf{p}).$
From this observation we obtain that
\begin{align*} 
\GenerateMatchings{H_1\cup H_2,X,\mathbf{p},F}=\mathsf{p}_{1,F}\cdot\mathsf{p}_{2,F}\cdot \prod_{e\in F}\mathbf{p}(e),
\end{align*}
which can be computed in time $\Abs{V(H_1)\cup V(H_2)}^{\mathcal{O}(k)}$ from $\GenerateMatchingsBoundary{(H_1,X_1,\mathbf{p})}$ and $\GenerateMatchingsBoundary{(H_2,X_2,\mathbf{p})}$ and thus our claim follows.
\end{proof}

\begin{proof}[Proof of \autoref{@unencumbered}] 
We are given a weighted graph $(G,\mathbf{w})$ as input. 
In the following we will be working with the labelled graph $(G,\mathbf{p}_{\mathbf{w}})$ instead. 
 
Let $\alpha $ and $\gamma$ be defined as in \autoref{@diskussion}. 
Then by \autoref{@incinerated} we can find a tree decomposition $(T,\beta ),$ where we see $T$  
as rooted on a root $r\in V(T).$ Also, adhesion at most $\alpha $ such that for every $d\in V(T),$ the torso $G_d$ of $G$ at $d$ has a set $A_d\subseteq V(G_d)$ of size at most $4\alpha $ for which the graph $G_d-A_d$ has Euler-genus at most $\gamma$ in time $\mathcal{O}_{t}\Abs{G}^3).$ 
Moreover, for every $(d_1,d_2)\in E(T)$ we have $\Abs{(\beta (d_1)\setminus A_{d_1})\cap(\beta (d_2)\setminus A_{d_2})}\leq 3,$ and if $\Abs{(\beta (d_1)\setminus A_{d_1})\cap(\beta (d_2)\setminus A_{d_2})}=3$ and $\beta (d_1)$ is larger than $4\alpha ,$ then $(\beta (d_1)\setminus A_{d_1})\cap(\beta (d_2)\setminus A_{d_2})$ induces a triangle in the corresponding drawing of $G_{d_1}-A_{d_1}$ which bounds a face. 
Observe that $\Abs{V(T)}\in\mathcal{O}(\Abs{G}).$ 
 
Set $k=\max\Set{4\alpha ,\gamma}$ and let $t\in V(T)$ be an arbitrary leaf. 
If $\Abs{\beta (t)}\leq 4\alpha $ we may use \autoref{@celebration} to initialize the tables at $t,$ otherwise we use \autoref{@transformations}. 
Now let $t\in V(T)$ be some internal vertex such that the tables for the subtrees rooted at the children of $t$ have already been computed. 
We proceed by iteratively applying \autoref{@celebration} and \autoref{@penitentiary}  
as long as we deal with torsos of  size at most $4\alpha .$ 
Then we find a $k$-branching $\mathcal{B}=((H,X),A_t,\mathbf{p}_{\mathbf{w}},(B_1,Y_1),\dots,(B_{\ell},Y_{\ell})),$ where the $(B_i,Y_i)$ represent the graphs induced by the union of the bags in the subtrees rooted at the children of $t$ which were merged into the adhesion set $Y_i\subseteq \beta (t),$ and $(H,X)$ is the boundary graph where $H$ is the subgraph of $G$ induced by the union of all bags in the subtree of $T$ rooted at $t$ and $X$ is the adhesion of $\beta (t)$ to the bag of the parent node of $t.$ 
In case $t=r$ we have $X=\emptyset.$ 
Using the procedure from \autoref{@incisiveness} allows us to correctly compute the table for $\mathcal{B}.$ 
 
Hence, by induction, we are able to compute $\GenerateMatchingsBoundary{(G,\emptyset,\mathbf{p}_{\mathbf{w}})}$ in time $\Abs{G}^{\mathcal{O}(\alpha )}.$ 
Notice that the only member of $\GenerateMatchingsBoundary{(G,\emptyset,\mathbf{p}_{\mathbf{w}})}$ is exactly $\GenerateMatchings{G,\mathbf{w}}$ and our proof is complete.
\end{proof}

\subsection{The complexity lower bound}\label{@stiidentenzeitung}

In the previous section we have seen how to compute the generating function for all perfect matchings efficiently on a graph that excludes some shallow vortex minor.
Once the generating function is known, the number of perfect matchings can be found by evaluating the function for $x=1.$
Hence, if $\mathcal{G}$ is a proper minor-closed class of graphs excluding some shallow vortex minor, computing the number of perfect matchings is tractable on $\mathcal{G}.$
Towards completing the proof of our main result (that is \autoref{@controllers}), it remains to prove \autoref{@winterhilfiwerlr}, i.e., show that on any proper minor-closed class of graphs which contains all shallow vortex minors, the problem of counting perfect matchings remains $\#\mathsf{P}$-hard.

To show this we make use of a recent complexity result from  \cite{curticapean2022parameterizing} which was used by Curticapean and Xia to prove that counting perfect matchings is $\#\mathsf{P}$-hard on $K_8$-minor free graphs.

\begin{proposition}[\!\!\cite{curticapean2022parameterizing}]
\label{@contradict}
Counting perfect matchings is $\#\mathsf{P}$-hard on the class of \hyperref[@consistently]{ring blowup} graphs.
\end{proposition}

The next observation follows by straightforward graph drawing arguments (see e.g., \cite[Lemma 5.5]{BasteST23Hitting}). Recall that we defined a \textsl{standard cross-free drawing on a disk $ \Delta$} of the $(t\times s)$-cylindrical grid one where of the extremal cycles is drawn on the boundary of $ \Delta$
\begin{lemma} 
\label{@proletariat}
There exists a function $f:\Bbb{N}\to\Bbb{N}$ such that the following holds:
Suppose that $\gamma$ is the embedding of some 2-connected planar graph $G$ on $n$ vertices in a closed disk $ \Delta$ and let $C$ be the cycle of $G$ whose drawing is the  boundary of $ \Delta.$  Let also $\gamma'$ be a standard cross-free drawing on a disk of the \hyperref[@consistently]{$(f(n)\times f(n))$-cylindrical grid} $G'$ in a closed disk  $ \Delta'$
and let $C'$ be the cycle of $G'$ defined by the boundary of the external face of $\gamma'.$
Then there is a \hyperref[@attachements]{minor model} $\CondSet{X_v}{v\in\V{H}}$ 
of $G$ to $G'$ such that for every $v\in V(C),$ $X_{v}$ is a subpath of $C'.$
Moreover, $f$ is a quadratic function.
\end{lemma}

It is interesting to observe that, according to \autoref{@incessantly} and \autoref{@proletariat}, 
the class $\Scal$ of shallow vortex minors is \textsl{exactly} the class of graphs 
that are minors of  ring blowups.

We now have what we need for the proof of \autoref{@winterhilfiwerlr}.

\begin{proof}[Proof of \autoref{@winterhilfiwerlr}]
From \autoref{@contradict}, it is enough to prove that the class of all \hyperref[@consistently]{ring blowup} graphs is a subset of $\mathcal{S}.$ 
In the special case where $G$ is not a 2-connected graph, we may create a new graph $G^+$ by adding in $G$ a minimum number of edges 
that can make it 2-connected, while maintaining its planarity and the property of having a cross-free drawing $\gamma^+$ on a disk where $Q$ is the set of vertices incident in its external face. Notice that the \hyperref[@consistently]{ring blowup} of $\gamma$ is a subgraph of the ring blowup of $\gamma^+.$ This implies that we may assume that $G$ is a 2-connected planar graph and prove that 
the \hyperref[@consistently]{ring blowup} $\hat{G}$ of $\gamma$ of $G$ is a minor of some, big enough, shallow vortex grid.
The 2-connectivity of $G$ permits us to assume that the vertices that are incident with the external 
face of $\gamma$ are the vertices of some cycle $C$ of $G.$ By \autoref{@proletariat},
there is a \hyperref[@consistently]{$(f(n)\times f(n))$-cylindrical grid} $G'$ whose standard cross-free drawing on a disk is $\gamma'$ and, given that $C'$ is the cycle defined by the boundary of the external face of $\gamma',$ there 
exists some \hyperref[@attachements]{minor model} $\CondSet{X_v}{v\in\V{H}}$ 
of $G$ to $G'$ such that for every $v\in V(C),$ $X_{v}$ is a subpath of $C'.$ 
It is now easy to observe that, because of this last property, $\hat{G}$ is a minor of the  
\hyperref[@consistently]{ring blowup} $\hat{G}'$ of $\gamma'.$ The result follows as $\hat{G}'$ is a $(f(n)\times f(n))$-\hyperref[@consistently]{cylindrical grid ring blowup}, which from \autoref{@incessantly} is the minor of $\mathscr{S}_{g(f(n))}\in \mathcal{S}$ (where $g$ is the function of \autoref{@incessantly}).
\end{proof}

\section{Conclusion}\label{@predominant}
Notice that, by definition, $\mathcal{S}$ is a minor-closed graph class. Let $\mathcal{Q}=\mathsf{obs}(\mathcal{S})$
be its minor-obstruction set, that is the set of all minor-minimal graphs not contained in $\mathcal{S}.$
We know, from Robertson and Seymour's theorem that $\mathcal{Q}$ is a finite set \cite{RobertsonS04wagner}. This set provides a ``finite'' version of the characterization in \autoref{@controllers} as follows.

\begin{corollary}
Let ${\cal F}$ be a finite set of graphs.
$\text{\textsc{\#\emph{Perfect Matching}}}(\mathcal{F})$ is polynomial-time solvable if $\mathcal{F}$ contains some $\mathcal{Q}$-minor free graph; otherwise it is $\#\mathsf{ P}$-complete.
\end{corollary}

The exact identification of the finite set $\mathcal{Q}$ seems to be an interesting, however hard, combinatorial problem. 

As already mentioned by Curticapean and Xia in~\cite{CurticapeanX21param},  it is interesting to investigate the complexity dichotomies of other families of counting problems in the realm of minor-closed graph classes. 
Such a framework is the one of  \textsl{Holant problems}: $\text{\textsc{\#{Perfect Matching}}}$ is a typical example
of this family of problems \cite{CaiLX09hola,CaiLX16holog}.
\smallskip

Another direction on the study of $\text{\textsc{\#{Perfect Matching}}}$ is to look for classes of bipartite graphs, ordered by the \textsl{matching-minor} relation.
The most general result in this direction is the one of McCuaig, Robertson, Seymour, and Thomas \cite{RobertsonST99perma,McCuaigRST97perma}, implying that such an algorithm exists for Pfaffian bipartite graphs, that, according to Little~\cite{Little06anext} are exactly the graphs excluding $K_{3,3}$ as a \textsl{matching minor}.
Moreover, the number of perfect matchings can be found efficiently on bipartite graphs excluding a planar  bipartite matching minor \cite{giannopoulou2021excluding}.
This induces an alternative line of research, asking for more general 
matching-minor-closed bipartite graph classes where $\text{\textsc{\#{Perfect Matching}}}$ is polynomially solvable.
This line of research may be of particular importance as the complexity of the permanent is directly linked to the complexity of counting perfect matchings on bipartite graphs.

\paragraph{Acknowledgements.} The first author wishes to thank \href{https://www.aueb.gr/en/faculty_page/mourtos-ioannis}{Ioannis Mourtos} for old discussions that offered some early inspiration on this project.

We wish to thank the anonymous referees for helping to greatly improve the presentation of the paper.

%
%
\end{document}